\documentclass[11pt]{amsart}
\usepackage{amsmath,amsfonts,amssymb,amscd,verbatim,multicol}
\usepackage[arrow,matrix,cmtip]{xy}

\usepackage{fullpage}

\newcommand{\trdeg}{\operatorname{tr.deg}}
\newcommand{\la}{\leftarrow}

\newcommand{\bpr}{\begin{proof}}
\newcommand{\epr}{\end{proof}}

\newcommand{\spec}{\operatorname{Spec}}

\newcommand{\ann} {\operatorname{ann}}

\newcommand{\lra} {\longrightarrow}
\newcommand{\lla}{\longleftarrow}

\newcommand{\D}{\displaystyle}
\newcommand{\mc}{\mathcal}

\newcommand{\cha}{\operatorname{char}}
\newcommand{\mb}{\mathbb}

\newcommand{\GK}{\operatorname{GKdim}}

\newcommand{\wt}{\widetilde}

\newcommand{\cproj}{\operatorname{proj}}

\newcommand{\coh}{\operatorname{coh}}

\newcommand{\gldim}{\operatorname{gl.dim}}

\newcommand{\pdim}{\operatorname{proj.dim}}

\newcommand{\rGr}{\operatorname{Gr-}\hskip -2pt}
\newcommand{\rQgr}{\operatorname{Qgr-}\hskip -2pt}
\newcommand{\rqgr}{\operatorname{qgr-}\hskip -2pt}

\newcommand{\rtors}{\operatorname{tors-}\hskip -2pt}

\newcommand{\tor}{\operatorname{Tor}}

\newcommand{\fig}{finitely graded}

\renewcommand{\P} {\mathcal{P}}

\newcommand{\Qcoh}{\operatorname{Qcoh}}
\newcommand{\beq}{\begin{equation}}
\newcommand{\eeq}{\end{equation}}

\newcommand{\uHom}{\underline{\rm Hom}}
\newcommand{\Hom}{{\rm Hom}}

\newcommand{\End}{{\rm End}}

\newcommand{\uExt}{\underline{\rm Ext}}
\newcommand{\Ext}{{\rm Ext}}

\DeclareMathOperator{\Tor}{Tor}

\DeclareMathOperator{\rgr}{gr-\!}     

\DeclareMathOperator{\im}{Im}

\newcommand{\coker}{\operatorname{coker}}
\newcommand{\pd}{\operatorname{proj.dim}}
\newcommand{\rgl}{\operatorname{r.gl.dim}}
\newcommand{\lgl}{\operatorname{l.gl.dim}}
\newcommand{\red}{\operatorname{red}}

\numberwithin{equation}{section}
 \theoremstyle{plain}
\newtheorem{theorem}[equation]{Theorem}

\newtheorem{lemma}[equation]{Lemma}

\newtheorem{corollary}[equation]{Corollary}
\newtheorem{proposition}[equation]{Proposition}

\theoremstyle{definition}
\newtheorem{question}[equation]{Question}
\newtheorem{definition}[equation]{Definition}

\newtheorem{remark}[equation]{Remark}

\newtheorem{standing-hypothesis}[equation]{Standing Hypothesis}
\newtheorem{example}[equation]{Example}
\newtheorem{examples}[equation]{Examples}

\title{An introduction to noncommutative projective algebraic geometry}
\author{Daniel Rogalski}
\begin{document}
\maketitle
\tableofcontents

These notes are a significantly expanded version of the author's lectures at the graduate workshop ``Noncommutative algebraic geometry" held at the Mathematical Sciences Research Institute in June 2012.  The main point of entry to the subject we chose was the idea of an Artin-Schelter regular algebra.   The introduction of such algebras by Artin and Schelter motivated many of the later developments in the subject.  Regular algebras are sufficiently rigid to admit classification in dimension at most 3, yet this classification is non-trivial and uses many interesting techniques.   There are also many open questions about regular algebras, including the classification in dimension 4.

Intuitively, regular algebras with quadratic relations can be thought of  as the coordinate rings of noncommutative projective spaces; 
thus, they provide examples of the simplest, most fundamental noncommutative projective varieties.  In addition, regular algebras provide some down-to-earth examples of Calabi-Yau algebras.  This is a class of 
algebras defined by Ginzburg more recently, which is related to several of the other lecture courses given at the workshop.

Our first lecture reviews some important background and introduces noncommutative Gr{\"o}bner bases.  We also include as part of Exercise set 1 a few exercises using the computer algebra system GAP.  Lecture 2 presents some of the main ideas of the theory of Artin-Schelter regular algebras.  Then, using regular algebras as examples and motivation, in Lectures 3 and 4 we discuss two important aspects of the geometry of noncommutative graded rings:  the parameter space of point modules for a graded algebra, and the noncommutative projective scheme associated to a noetherian graded ring.  Finally, in the last lecture we discuss some aspects of the classification of noncommutative curves and surfaces, including a review of some more recent results.

We have tried to keep these notes as accessible as possible to readers of varying backgrounds.  In particular, Lectures 1 and 2 
assume only some basic familiarity with noncommutative rings and homological algebra.  Only knowledge of the concept of a projective space is needed to understand the main ideas about point modules in the first half of Lecture 3.   In the 
final two lectures, however, we will of necessity assume that the reader has a more thorough background in algebraic geometry  including the theory of schemes and sheaves as in Hartshorne's textbook \cite{Ha}.

We are indebted to Toby Stafford, from whom we first learned this subject in a graduate course at the University of Michigan.   Other sources that have influenced these notes include some lecture notes of Darrin Stephenson \cite{Ste1}, and the survey article of Stafford and Van den Bergh \cite{StV}; we thank all of these authors.   We also thank Susan Elle, Matthew Grimm, Brendan Nolan, and Robert Won for reading an earlier version of these notes and giving helpful comments.  

\section{Lecture 1:   Review of basic background and the Diamond Lemma}
\label{sec:lec1}

\subsection{Graded algebras}
\label{sec:graded}

In this lecture, we review several topics in the theory of rings and homological algebra which are needed 
before we can discuss Artin-Schelter regular algebras in Lecture 2.  We also include an introduction to noncommutative Gr{\"o}bner bases and the Diamond Lemma.

Throughout these notes we work for simplicity over an algebraically closed base field $k$.  Recall that a $k$-algebra is a (not necessarily commutative) ring $A$ with identity which has a copy of $k$ as a subring of its center; then $A$ is also a $k$-vector space such that scalar multiplication $\cdot$ satisfies  $(\lambda \cdot a)b = \lambda \cdot (ab) = a (\lambda \cdot b)$ for all $\lambda \in k$, $a, b \in A$.   (The word algebra is sometimes used for objects with nonassociative multiplication, in particular Lie algebras, but 
for us all algebras are associative.)  

\begin{definition}
A $k$-algebra $A$ is \emph{$\mb{N}$-graded} if it has a $k$-vector space 
decomposition $A = \bigoplus_{n \geq 0} A_n$ such that $A_i A_j \subseteq A_{i+j}$ for all $i, j \geq 0$.  
%It is not hard to see that 
%this forces the copy of $k$ that makes $A$ into an algebra to be contained in $A_0$.  
We say that $A$ is \emph{connected} if $A_0 = k$.  An element $x$ in $A$ is \emph{homogeneous} if $x \in A_n$ for some $n$.  A right or left ideal $I$ of $A$ is called \emph{homogeneous} or \emph{graded} if it is generated by homogeneous elements, or equivalently if $I = \bigoplus_{n \geq 0} (I \cap A_n)$.
\end{definition}

\begin{comment}
In commutative algebraic geometry, affine varieties are constructed by taking the prime spectrum $\spec R$ of a commutative $k$-algebra $R$, while one way to get projective varieties is to start with an $\mb{N}$-graded algebra $R$ and take $\operatorname{Proj} R$, the spectrum of homogeneous primes excluding the irrelevant ideal $R_{\geq 1} = \bigoplus_{n \geq 1} R_n$.   Thus there are many important correspondences between the theory of projective varieties and the theory of commutative graded algebras.  The subject of noncommutative projective geometry aims to develop a geometric theory corresponding to the theory of noncommutative graded algebras; in particular, we will be primarily interested in $\mb{N}$-graded algebras here.
\end{comment}

\begin{example}
Recall that the \emph{free algebra} in $n$ generators $x_1, \dots, x_n$ is 
the ring $k \langle x_1, \dots, x_n \rangle$, whose underlying $k$-vector space has as basis 
the set of all words in the variables $x_i$, that is, expressions $x_{i_1} x_{i_2} \dots x_{i_m}$ for some $m \geq 1$, where $1 \leq i_j \leq n$ for all $j$.  The \emph{length}  of a word $x_{i_1} x_{i_2} \dots x_{i_m}$ is $m$.   We include among the words a symbol $1$, which we think 
of as the empty word, and which has length $0$.    The product of two words is concatenation, and this operation is extended linearly to define an associative product on all elements.

%For example, in $k \langle x, y \rangle$ we have $(xyx + 5)(x^2y + x) = xyx^3y + 5x^2y + xyx^2 + 5x)$.

The free algebra $A = k \langle x_1, \dots, x_n \rangle $ is connected $\mb{N}$-graded, where $A_i$ is the $k$-span of all words of length $i$.  For a more general grading, one can put \emph{weights} $d_i \geq 1$ on the variables $x_i$ and define $A_i$ to be the $k$-span of all words $x_{i_1} \dots x_{i_m}$ 
such that $\sum_{j = 1}^m d_{i_j} = i$.
\end{example}

\begin{definition}
A $k$-algebra $A$ is \emph{finitely generated} (as an algebra) if there is a finite set of elements $a_1, \dots a_n \in A$ such that 
the set $\{ a_{i_1} a_{i_2} \dots a_{i_m} | 1 \leq i_j \leq n,  m \geq 1 \} \cup \{ 1 \}$ spans $A$ as a $k$-space.  It is clear that if $A$ is finitely generated and $\mb{N}$-graded, then it has a finite set of homogeneous elements that generate it.   
Then it is easy to see that a connected $\mb{N}$-graded $k$-algebra 
$A$ is finitely generated if and only if there is a degree preserving surjective ring homomorphism $k \langle x_1, \dots, x_n \rangle \to A$ 
for some free algebra $k \langle x_1, \dots, x_n \rangle$ with some weighting of the variables, and thus $A \cong k \langle x_1, \dots x_n \rangle/I$ for some homogeneous ideal $I$.  If $I$ is generated by finitely many homogeneous elements (as a $2$-sided ideal), say 
$I = (f_1, \dots, f_m)$, then we say that $A$ is \emph{finitely presented}, and we call $k \langle x_1, \dots, x_n \rangle/(f_1, \dots, f_m)$ a 
\emph{presentation} of $A$ with generators $x_1, \dots, x_n$ and relations $f_1, \dots, f_m$.
\end{definition}

\begin{definition}
For the sake of brevity, in these notes we say that an algebra $A$ is \emph{\fig} if it is connected $\mb{N}$-graded and finitely generated as a $k$-algebra.   Note that if $A$ is finitely graded, then $\dim_k A_n < \infty$ for all $n$, since this is true already for the free algebra.
\end{definition}

In Section~\ref{sec:ex} below, we will give a number of important examples of algebras defined by presentations.

\begin{comment}
\begin{example}
To give a trivial example, the commutative polynomial ring $k[x_1, \dots, x_n]$ is finitely generated by the $x_i$ and 
has the explicit presentation $k \langle x_1, \dots, x_n \rangle/( \{ x_j x_i - x_i x_j | 1 \leq i  < j \leq n \})$ (it is possible to prove this directly; it also follows from the diamond lemma discussed below).  It is obviously 
a graded ring where $\deg x_i = 1$; this follows also since we are modding out by a homogeneous ideal in the free algebra.
\end{example}
\end{comment}

\subsection{Graded modules, GK-dimension, and Hilbert series}
\label{sec:hilb}

\begin{definition}
Let $A$ be an $\mb{N}$-graded $k$-algebra.   A right $A$-module $M$ is \emph{graded} 
if $M$ has  a $k$-space decomposition $M = \bigoplus_{n \in \mb{Z}} M_n$ such that $M_i A_j \subseteq M_{i+j}$ for all $i \in \mb{Z}, j \in \mb{N}$.  
%($M$ is nonzero only in degrees at least as big as the generator of lowest degree.)

Given a graded $A$-module $M$, we define $M(i)$ to be the graded module which is isomorphic to $M$ as an abstract $A$-module, but which has degrees shifted so that $M(i)_n = M_{i+n}$.  Any such module is called a \emph{shift} of $M$.
(Note that if we visualize the pieces of $M$ laid out along the integer points of the usual number line, then to obtain $M(i)$ 
one shifts all pieces of $M$ to the left $i$ units if $i$ is positive, and to the right $|i|$ units if $i$ is negative.) 

A homomorphism of $A$-modules $\phi: M \to N$ is  a \emph{graded homomorphism} if $\phi(M_n) \subseteq N_n$ for all $n$.
\end{definition}

We will mostly be concerned with graded $A$-modules $M$ which are finitely generated.  In this case, we can find a finite set of homogeneous generators of $M$, say $m_1, \dots, m_r$ with $m_i \in M_{d_i}$,  and thus define a surjective graded right $A$-module homomorphism $\bigoplus_{i = 1}^r A(-d_i) \to M$, where the $1$ of the $i$th summand maps to the generator $m_i$.  This shows that any finitely generated graded $A$-module $M$ over a \fig\ algebra $A$ has $\dim_k M_n < \infty$ for all $n$ and $\dim_k M_n = 0$ for $n \ll 0$, and so the following definition makes sense.  

\begin{definition}
Let $A$ be \fig.
If $M$ is a finitely generated graded $A$-module,  then 
the \emph{Hilbert series} of $M$ is the formal Laurent series $h_M(t) = \sum_{n \in \mb{Z}} (\dim_k M_n)t^n$.
\end{definition}

We consider the Hilbert series of a finitely generated graded module as a generating function (in the sense of combinatorics) for the integer sequence $\dim_k M_n$, and it is useful to manipulate it in the ring of Laurent series $\mb{Q}((t))$.

\begin{example}
The Hilbert series of the commutative polynomial ring $k[x]$ is $1 + t + t^2 + \dots$, 
which in the Laurent series ring has the nicer compact form $1/(1-t)$.  More generally, if $A = k[x_1, \dots, x_m]$ then $h_A(t) = 1/(1-t)^m$.  On the other hand, the free associative algebra $A = k \langle x_1, \dots, x_m \rangle$ has Hilbert series 
\[
h_A(t) = 1 + mt + m^2t^2 + \dots = 1/(1-mt). 
\]
In particular, $\dim_k A_n$ grows exponentially as a function of $n$ if $m \geq 2$.
% and we say that $A$ has \emph{exponential growth}.
In Exercise 1.1, the reader is asked to prove more general versions of these formulas for weighted polynomial rings 
and free algebras.  
\end{example}

\begin{definition}
If $A$ is a finitely generated (not necessarily graded) $k$-algebra, the Gelfand-Kirillov (GK) dimension of $A$ is $\GK(A) = \limsup_{n \to \infty} \log_n (\dim_k V^n)$, where $V$ is any finite dimensional $k$-subspace of $A$  which generates $A$ as an algebra and has $1 \in V$. The algebra $A$ has \emph{exponential growth} if  $\limsup_{n \to \infty} (\dim_k V^n)^{1/n} > 1$; 
otherwise, clearly $\limsup_{n \to \infty} (\dim_k V^n)^{1/n} = 1$ and we say that $A$ has \emph{subexponential growth}.
The book \cite{KL} is the main reference for the basic facts about the GK-dimension.  In particular, the definitions above do not depend on the choice of $V$ \cite[Lemma 1.1, Lemma 2.1]{KL}.  Also, if $A$ is a commutative finitely generated algebra, then $\GK A$ is the same as the Krull dimension of $A$ \cite[Theorem 4.5(a)]{KL}.

If $A$ is \fig, then one may take $V$ to be $A_0 \oplus \dots \oplus A_m$ for some $m$, and using this 
one may prove that $\GK A = \limsup_{n \to \infty} \log_n(\sum_{i = 0}^n \dim_k A_i)$ \cite[Lemma 6.1]{KL}.  
This value is easy to calculate if we have a formula for the dimension of the $i$th graded piece of $A$.
In fact, in most of the examples in which we are interested below, there is a polynomial $p(t) \in \mb{Q}(t)$ such that $p(n) = \dim_k A_n$ for all $n \gg 0$, in which case $p$ is called the \emph{Hilbert polynomial} of $A$.  When $p$ exists 
then it easy to see that $\GK(A) = \deg(p) + 1$.  For example, for the commutative 
polynomial ring $A = k[x_1, \dots, x_m]$, one has $\dim_k A_n = \binom{n +m-1}{m-1}$, which agrees with a polynomial $p(n)$ of degree $m-1$ for all $n \geq 0$, so that $\GK A = m$.
\end{definition}

We briefly recall the definitions of noetherian rings and modules.
\begin{definition}
A right module $M$ is \emph{noetherian} if it has the ascending chain condition (ACC) on submodules, or equivalently if every submodule of $M$ 
is finitely generated.  A ring is right noetherian if it is noetherian as a right module over itself, or equivalently if it has ACC on right ideals.  The left noetherian property is defined analogously, and a ring is called noetherian if it is both left and right noetherian.
\end{definition}
The reader can consult \cite[Chapter 1]{GW} for more information on the noetherian property.  By the Hilbert basis theorem, the polynomial ring $k[x_1, \dots, x_m]$ is noetherian, and thus all finitely generated commutative $k$-algebras are noetherian.  On the other hand, a free algebra $k \langle x_1, \dots, x_m \rangle$ in $m \geq 2$ variables is not noetherian, and consequently noncommutative finitely generated algebras need not be noetherian.  The noetherian property still holds for many important noncommutative examples of interest, but one often has to work harder to prove it.

\subsection{Some examples, and the use of normal elements}
\label{sec:ex}

\begin{examples}
\label{ex:qplane}
\label{ex:jordan}
\label{ex:qpoly}
\
%(quantum and Jordan plane)
\begin{enumerate}
\item For any constants $0 \neq q_{ij} \in k$, the algebra 
\[
A = k \langle x_1, x_2, \dots, x_n \rangle/(x_jx_i - q_{ij} x_ix_j | 1 \leq i < j \leq n)
\]
is called a \emph{quantum polynomial ring}.   
The set $\{ x_1^{i_1} x_2^{i_2} \dots x_n^{i_n} |  i_1, i_2, \dots, i_n \geq 0 \}$  is a $k$-basis for $A$, as we 
will see in Example~\ref{ex:firsths}.  Then $A$ has the same Hilbert series as a commutative polynomial ring in $n$ variables, and thus $h_A(t) = 1/(1-t)^n$.

\item The special case $n = 2$ of (1), so that 
\[
 A =   k \langle x, y \rangle/(yx- qxy)
\] 
for some $0 \neq q$, is called the \emph{quantum plane}.

\item 
The algebra 
\[
A = k \langle x, y \rangle/(yx - xy - x^2)
\]
 is called the \emph{Jordan plane}.  We will also see in Example~\ref{ex:firsths} that if $A$ is the Jordan plane, then $\{ x^i y^j | i, j \geq 0 \}$ is a $k$-basis for $A$, and so $h_A(t) = 1/(1-t)^2$.
\end{enumerate}

(Note that we often use a variable name such as $x$ to indicate both an element in the free algebra and the corresponding coset in the factor ring;  this is often convenient in a context where there is no chance of confusion.)  
\end{examples}

All of the examples above have many properties in common with a commutative polynomial ring in the same number of generators.  For example, they all have the Hilbert series of a polynomial ring and they are all noetherian domains.  The standard 
way to verify these facts is to express these examples as iterated Ore extensions.   We omit a discussion of Ore extensions here, since the reader can find a thorough introduction to these elsewhere, for example in \cite[Chapters 1-2]{GW}.  Instead, we mention a different method here, which will also apply to some important examples which are not iterated Ore extensions.  
Given a ring $A$, an element $x \in A$ is \emph{normal} if $xA = Ax$ (and hence the right or left ideal generated by $x$ is an ideal). Certain properties can be lifted from a factor ring of a graded ring to the whole ring, when one factors by an ideal 
generated by a homogeneous normal element.

\begin{lemma}
\label{lem:passup}
Let $A$ be a \fig\ $k$-algebra, and let $x \in A_d$ be a homogeneous normal element for some $d \geq 1$.
\begin{enumerate}
\item If $x$ is a nonzerodivisor in $A$, then if $A/xA$ is a domain then $A$ is a domain.
\item If $A/xA$ is noetherian, then $A$ is noetherian.
\end{enumerate}
\end{lemma}
\begin{proof}
We ask the reader to prove part (1) as Exercise 1.2.  

We sketch the proof of part (2), which is \cite[Theorem 8.1]{ATV1}.   First, by symmetry we need only show that $A$ is right noetherian.   An easy argument, which we leave to the reader, shows that it suffices to show that every \emph{graded} right ideal of $A$ is finitely generated.  Suppose that $A$ has an infinitely generated graded right ideal; then by Zorn's lemma, we may choose 
a maximal element of the set of such right ideals, say $I$.  Then $A/I$ is a noetherian right $A$-module.  Consider the short exact sequence
\begin{equation}
\label{eq:noeth}
0 \to (Ax \cap I)/Ix   \to I/Ix \to  I/(Ax \cap I) \to 0.
\end{equation}
We have $I/(Ax \cap I) \cong (I + Ax)/Ax$, which is a right ideal of $A/Ax$ and hence is noetherian by hypothesis.
Now $(Ax \cap I) = Jx$ for some subset $J$ of $A$ which is easily checked to be a graded right ideal of $A$.  Then 
$(Ax \cap I)/Ix = Jx/Ix \cong Mx$, where $M = J/I$ is a noetherian $A$-module since it is a submodule of $A/I$.   Given 
an $A$-submodule $P$ of $Mx$, it is easy to check that $P = Nx$ where $N = \{ m \in M | mx \in P \}$.  Thus 
since $M$ is noetherian, so is $Mx$.  Then 
\eqref{eq:noeth} shows that $I/Ix$ is a noetherian $A$-module, in particular finitely generated.  Thus we can choose a finitely generated graded right ideal $N \subseteq I$ such that $N + Ix = I$.  An easy induction proof shows that $N + Ix^n = I$ for all $n \geq 0$, 
and then one gets $N_m = I_m$ for all $m \geq 0$ since $x$ has positive degree.  In particular, $I$ is finitely generated, a contradiction.  (Alternatively, once one knows $I/Ix$ is finitely generated, it follows that $I$ is finitely generated 
by the graded Nakayama lemma described in Lemma~\ref{lem:nak} below.)
\end{proof}

\begin{corollary}
\label{cor:nd}
The algebras in Examples~\ref{ex:qplane} are noetherian domains.
\end{corollary} 
\begin{proof}
Consider for example the quantum plane $A = k \langle x, y \rangle/(yx - qxy)$.  Then it is easy to 
check that $x$ is a normal element and that $A/xA \cong k[y]$.  We know $h_A(t) = 1/(1-t)^2$ and 
$h_{A/xA}(t) = 1/(1-t)$.  Obviously $h_{xA}(t)$ is at most as large as $t/(1-t)^2$, with equality if and only if 
$x$ is a nonzerodivisor in $A$.  Since $h_A(t) = h_{xA}(t) + h_{A/xA}(t)$, equality is forced, so 
$x$ is a nonzerodivisor in $A$.  Since $k[y]$ is a noetherian domain, so is $A$ by Lemma~\ref{lem:passup}.
The argument for the general skew polynomial ring is a simple inductive version of the above, and 
the argument for the Jordan plane is similar since $x$ is again normal.
\end{proof}

\begin{example} %(Sklyanin algebra of dimension $3$)
\label{skl-ex}
\label{ex:skl}
The algebra 
\[ 
S = k \langle x, y, z \rangle/(ayx + bxy + cz^2, axz + bzx + cy^2, azy + byz + cx^2),
\] for any $a, b,c \in k$, is called a Sklyanin algebra, after E. J. Sklyanin.  As long as $a, b, c$ are sufficiently general, $S$ also has many properties in common with a polynomial ring in $3$ variables, for example  $S$ is a noetherian domain with $h_S(t) = 1/(1-t)^3$.  These facts are much more difficult to prove than for the other examples above, since $S$ does not have such a simply described $k$-basis of words in the generators.  In fact $S$ does have a normal element $g$ of degree $3$ (which is not easy to find) and in the end Lemma~\ref{lem:passup} can be applied, but 
it is hard to show that $g$ is a nonzerodivisor and the factor ring $S/gS$ takes effort to analyze.   Some of the techniques of noncommutative geometry that are the subject of this course, in particular the study of point modules, were developed precisely to better understand the properties of Sklyanin algebras.    See the end of Lecture 3 for more details. 
\end{example}

\subsection{The Diamond Lemma}

In general, it can be awkward to prove that a set of words that looks like a $k$-basis for a presented algebra really is linearly independent.  The \emph{Diamond Lemma} gives an algorithmic method for this.  
It is part of the theory of noncommutative Gr{\"o}bner bases, which is a direct analog of the theory of Gr{\"o}bner bases which was first studied for commutative rings.  George Bergman gave an important  exposition of the method in \cite{Be}, which we follow closely below.  The basic idea behind the method goes back farther, however; in particular, the theory is also known by the name Gr{\"o}bner-Shirshov bases, since A. I. Shirshov was one of its earlier proponents.

%It is often associated with George Bergman, who published a well-known (the first?) exposition of it in the 1970's, although the idea %definitely goes back farther.  It is also a fairly direct generalization of the commutative Grobner basis theory.  It suffices to prove that %most of the examples above have the claimed basis of words, but it does not appear to work well for the Sklyanin algebra.  This is one %reason why the Sklyanin algebra is a much more difficult ring to understand (it is also more interesting).

Consider the free algebra $F = k \langle x_1, \dots, x_n \rangle$.  While we stick to free algebras on finitely many indeterminates below for notational convenience, everything below goes through easily for arbitrarily many indeterminates.  Fix an ordering on the variables, say  $x_1 < x_2 < \dots < x_n$.  
%(though sometimes making another choice will make the diamond lemma work better; this is a matter of experiment.)    
Also, we choose an total order on the set of words in the $x_i$ which extends the ordering on the variables and has the following properties: (i) for all words $u, v,$ and $w$, if $w < v$, then $uw < uv$ and $wu < vu$, and (ii) for each word $w$ the set of words $\{ v | v < w \}$ is finite. We call such an order \emph{admissible}.

If we assign weights to the variables, and thus assign a degree to each word, then one important choice of such an ordering is the degree lex order, where $w < v$ if $w$ has smaller degree than $v$ or if $w$ and $v$ have the same degree but $w$ comes earlier than $v$ in the dictionary ordering with respect to the fixed ordering on the variables.  For example, 
in $k \langle x, y \rangle$ with $x < y$, the beginning of the degree lex order is 
\[
1 < x < y < x^2 < xy < yx < y^2 < x^3 < x^2y < xyx  < xy^2 < \dots
\]
Given an element $f$ of the free algebra, its \emph{leading word} is the largest word under the ordering which appears in $f$ with nonzero coefficient.  For example, the leading word of $xy + 3 x^2y + 5 xyx$ in $k \langle x, y \rangle$ with the degree lex ordering is $xyx$.

Now suppose that the set $\{ g_{\sigma} \}_{\sigma \in S} \subseteq F$ generates an ideal $I$ of $F$ (as a 2-sided ideal).
By adjusting each $g_{\sigma}$ by a scalar we can assume that each $g_{\sigma}$ has a leading term with coefficient $1$, 
and so we can write $g_{\sigma} = w_{\sigma} - f_{\sigma}$, where $w_{\sigma}$ is the leading word of $g_{\sigma}$ and $f_{\sigma}$ is a linear combination of words $v$ with 
$v < w_{\sigma}$.   
%We will also assume that the $w_i$ are all distinct.  If not, say $w_1 = w_2$, we can replace $g_1, g_2$ by the pair 
%$g_1, g_1- g_2$ without changing the ideal $I$; making such replacements as necessary (possibly infinitely many times) we obtain 
%a list of $g_i$ with distinct leading words.
A word is \emph{reduced} (with respect to the fixed set of relations $\{ g_{\sigma}\}_{\sigma \in S}$) if it does not contain any of the $w_{\sigma}$ as a subword.   If a word $w$ is not reduced, but say $w = uw_{\sigma} v$, then $w$ is equal modulo $I$ to $u f_{\sigma} v$, which is a linear combination of strictly smaller words.  
Given an index $\sigma \in S$ and words $u, v$, the corresponding \emph{reduction} $r = r_{u\sigma v}$ 
is the $k$-linear endomorphism of the free algebra which takes $w = u w_{\sigma} v$ to $u f_{\sigma} v$ and sends every other word to itself.   Since every word has finitely many words less than it, it is not hard to see that given any element $h$ of the free algebra, some finite composition of reductions will send $h$ to a $k$-linear combination of reduced words.   
Since a reduction does not change the class of an element modulo $I$, we see that the images of the reduced words in $k \langle x_1, x_2, \dots, x_n \rangle/I$ are a $k$-spanning set.    The idea of noncommutative Gr{\"o}bner bases is to find good sets of generators $g_{\sigma}$ of the ideal $I$ such that the images of the corresponding reduced words in $k \langle x_1, \dots, x_n \rangle/I$ are $k$-independent, and thus a $k$-basis.  The Diamond Lemma gives a convenient way to verify if a set of generators $g_{\sigma}$ has this property.  Moreover, if a set of generators does not, the Diamond Lemma leads to a (possibly non-terminating) algorithm for enlarging the set of generators to one that does.   

The element $h$ of the free algebra is called \emph{reduction unique} if given any two finite compositions of reductions 
$s_1$ and $s_2$ such that $s_1(h)$ and $s_2(h)$ consist of linear combinations of reduced words, we have $s_1(h) = s_2(h)$.  
In this case we write $\red(h)$ for this uniquely defined linear combination of reduced words.
Suppose that $w$ is a word which can be written as $w = t v u$ for some nonempty words $t, u, v$, where $w_{\sigma} = tv$ and $w_{\tau} = vu$ for some $\sigma, \tau \in S$.  We call this situation an \emph{overlap} ambiguity.   Then there are (at least) two different possible reductions one can apply to reduce $w$, namely 
$r_1 = r_{1\sigma u}$ and $r_2 = r_{t\tau1}$.   If there exist compositions of reductions $s_1, s_2$ with the property 
that $s_1 \circ r_1(w) = s_2 \circ r_2(w)$, then we say that this ambiguity is \emph{resolvable}.
Similarly, we have an \emph{inclusion} ambiguity  when we have $w_{\sigma} = t w_{\tau} u$ for some words $t, u$ and some $\sigma, \tau \in S$.  Again, the ambiguity is called resolvable if there are compositions of reductions $s_1$ and $s_2$ such that $s_1 \circ r_{1\sigma1}(w) = s_2 \circ r_{t\tau u}(w)$.  

We now sketch the proof of the main result underlying the method of noncommutative Gr{\"o}bner bases. 
\begin{theorem}  (Diamond Lemma)
\label{thm:diamond}
Suppose that $\{ g_{\sigma}\}_{\sigma \in S} \subseteq F = k \langle x_1, \dots, x_n \rangle$ generates an ideal $I$, where $g_{\sigma} = w_{\sigma} - f_{\sigma}$ with $w_{\sigma}$ the leading word of $g_{\sigma}$ under some fixed admissible ordering 
on the words of $F$.  Consider reductions with respect to this fixed set of generators of $I$.  Then the following are equivalent:
\begin{enumerate}
\item All overlap and inclusion ambiguities among the $g_{\sigma}$ are resolvable.
\item All elements of $k \langle x_1, \dots, x_n \rangle$ are reduction unique.
\item The images of the reduced words in $k \langle x_1, \dots, x_n \rangle/I$ form a $k$-basis.
\end{enumerate}
When any of these conditions holds, we say that $\{g_{\sigma} \}_{\sigma \in S}$ is a Gr{\"o}bner basis for the ideal $I$ of $F$. 
\end{theorem}
\begin{proof}
$(1) \implies (2)$    First, it is easy to prove that the set of reduction unique elements of $k \langle x_1, \dots, x_n \rangle$ is a $k$-subspace, and that $\red( - )$ is a linear function on this subspace (Exercise 1.3).  Thus it is enough to prove that all words are reduction unique.  This is proved by induction 
on the ordered list of words, so assume that $w$ is a word such that all words $v$ with $v < w$ are 
reduction unique (and so any linear combination of such words is).   Suppose that $r = r_k \circ \dots \circ r_2 \circ r_1$ 
and $r' = r'_j \circ \dots \circ r'_2 \circ r'_1$ are two compositions of reductions such that 
$r(w)$ and $r'(w)$ are linear combinations of reduced words.  if $r_1 = r'_1$, then since $r_1(w) = r'_1(w)$ is reduction 
unique by the induction hypothesis, clearly $r(w) = r'(w)$.  Suppose instead that $r_1 = r_{y \sigma uz}$ and $r'_1 = r_{yt \tau z}$ where $w = yw_{\sigma}uz = ytw_{\tau}z$ and the subwords $w_{\sigma}$ and $w_{\tau}$ overlap.  By the hypothesis that all overlap ambiguities resolve, there are compositions of reductions $s_1$ and $s_2$ such that $s_1 \circ r_{1 \sigma u}(w_{\sigma} u) = s_2 \circ r_{t \tau 1}(tw_{\tau})$.  Then replacing each reduction $r_{a\rho b}$ among those occurring in the compositions $s_1$ and $s_2$ by $r_{y a \rho b z}$, we obtain compositions of reductions $s'_1$ and $s'_2$ such that $v = s'_1 \circ r_{y \sigma uz}(w) = s'_2 \circ r_{yt \tau z}(w)$.    Since $r_{y \sigma uz}(w), r_{yt \tau z}(w),$ and $v$ are reduction unique, we get 
$r(w) = \red(r_{y \sigma uz}(w)) = \red(v) = \red(r_{yt \tau z}(w)) = r'(w)$.  Similarly, if $w = ytw_{\sigma}uz = yw_{\tau} z$ and 
$r_1 = r_{yt \sigma uz}$ and $r'_1 = r_{y \tau z}$, then using the hypothesis that all inclusion ambiguities resolve we get $r(w) = r'(w)$.
Finally, if $w = yw_{\sigma} u w_{\tau} z$ and $r_1 = r_{y \sigma uw_{\tau} z}$, $r'_1 = r_{yw_{\sigma} u \tau z}$, then 
$r_1(w) = yf_{\sigma} uw_{\tau} z$ and $r'_1(w) = yw_{\sigma} u f_{\tau} z$ are linear combinations of reduction unique words.  Then since $\red(-)$ is linear, we get $r(w) = \red(yf_{\sigma} uw_{\tau} z) = \red(yf_{\sigma} uf_{\tau} z) = \red(yw_{\sigma} u f_{\tau} z) = r'(w)$.  Thus $r(w) = r'(w)$ in all cases, and $w$ is reduction unique, completing the induction step.

$(2) \implies (3)$  Let $F = k \langle x_1, \dots, x_n \rangle$.  
By hypothesis there is a well-defined $k$-linear map $\red: F \to F$.   Let $V$ be its image $\red(F)$, in order words the $k$-span of all reduced words.  Obviously $\red \vert_V$ is the identity map $V \to V$ and so $\red$ is a projection; thus 
$F \cong K \oplus V$ as $k$-spaces, where $K = \ker (\red)$.  We claim that $K = I$.  First, every element of $I$ is a linear combination of expressions $ug_{\sigma} v$ for words $u$ and $v$.  Obviously the reduction $r_{u \sigma v}$ sends $ug_{\sigma} v$ to $0$ and thus $u g_{\sigma} v \in K$; 
since $K$ is a subspace we get $I \subseteq K$.  Conversely, since every reduction changes an element to another one congruent modulo $I$, we must have $K \subseteq I$.  Thus $K = I$ as claimed.   Finally, since $F \cong I \oplus V$, 
clearly the basis of $V$ consisting of all reduced words has image in $F/I$ which is a $k$-basis of $F/I$.  
%the statement $F \cong I \oplus V$ is equivalent 
%to the statement that the map $V \to V + I/I \subseteq F/I$ is an isomorphism, in other words that 

The reverse implications $(3) \implies (2)$ and $(2) \implies (1)$ are left as an exercise (Exercise 1.3).
\end{proof}

\begin{example}
\label{ex:firsths}
Let $A = k \langle x, y, z \rangle/(f_1, f_2, f_3)$ be a quantum polynomial ring in three variables, with 
$f_1 = yx - pxy, f_2 = zx - q xz$, and $f_3 = zy - ryz$.  Taking $x < y< z$ and degree lex order, the leading terms of these relations are $yx, zx, zy$.  There is one ambiguity among these three leading words, the overlap ambiguity $zyx$.
Reducing the $zy$ first and then continuing 
to do more reductions, we get 
\[
zyx = ryzx = rqyxz = rqpxyz,
\]
while reducing the $yx$ first we get
\[
zyx = pzxy = pqxzy = pqrxyz.
\]
Thus the ambiguity is resolvable, and by Theorem~\ref{thm:diamond} the set $\{f_1, f_2, f_3 \}$ is a Gr{\"o}bner basis for the ideal it generates.
The same argument applies to a general quantum polynomial ring in $n$ variables: choosing degree lex order with 
$x_1 < \dots < x_n$, there is one overlap ambiguity $x_k x_j x_i$ for each triple of variables with $x_i < x_j < x_k$, which resolves by the same argument as above.   Thus the corresponding set of reduced words, $\{ x_1^{i_1} x_2^{i_2} \dots x_n^{i_n} |  i_j \geq 0 \}$,  is a $k$-basis for the quantum polynomial ring $A$.  In particular, $A$ has the same Hilbert series as a polynomial ring in $n$ variables, $h_A(t) = 1/(1-t)^n$, as claimed in Examples~\ref{ex:qpoly}.

An even easier argument shows that the Jordan plane in Examples~\ref{ex:jordan} has the claimed $k$-basis $\{ x^i y^j | i, j \geq 0 \}$:  taking $x < y$, its single relation $yx -xy -x^2$ has leading term $yx$ and there are no ambiguities.
\end{example}

\begin{example}
\label{ex:cubic}
Let $A = k \langle x, y \rangle/(yx^2 - x^2y, y^2x - xy^2)$.
Taking degree lex order with $x < y$, there is one overlap ambiguity $y^2x^2$ in the Diamond Lemma, which resolves, as is easily checked.  Thus the set of reduced words is $\{ x^i (yx)^j y^k | i, j, k \geq 0 \}$ is a $k$-basis for $A$.  Then $A$ has the same Hilbert series as a commutative polynomial ring in variables of weights $1, 1, 2$, namely $h_A(t) = 1/(1-t)^2(1-t^2)$ (Exercise 1.1).
\end{example}

If a generating set $\{ g_{\sigma} \}_{\sigma \in S}$ for an ideal $I$ has non-resolving ambiguities and so is not a Gr{\"o}bner basis, the proof of the Diamond Lemma also leads to an algorithm for expanding the generating set to get a Gr{\"o}bner basis.  Namely, suppose the overlap 
ambiguity $w = w_{\sigma} u = t w_{\tau}$ does not resolve.  Then for some compositions of reductions $s_1$ and $s_2$, 
$h_1 = s_1 \circ r_{1 \sigma u}(w)$ and $h_2 = s_2 \circ r_{t \tau 1}(w)$ are distinct linear combinations of reduced words.  Thus 
$0 \neq h_1 - h_2 \in I$ is a new relation, whose leading word is necessarily different from any leading word $w_{\rho}$ with $\rho \in S$.  Replace this relation with a scalar multiple so that its leading term has coefficient 1, and add this new relation to the generating set of $I$.  The previously problematic overlap now obviously resolves, but there may be new ambiguities involving the leading word of the new relation, and one begins the process of checking ambiguities  again.  Similarly, a nonresolving inclusion ambiguity will lead to a new relation in $I$ with new leading word, and we add this new relation to the generating set.  This process may terminate after finitely many steps, and thus produce a set of relations with no unresolving ambiguities, and hence a Gr{\"o}bner basis.   Alternatively, the process may not terminate.  It is still true in this case that the infinite set of relations produced by repeating the process infinitely is a Gr{\"o}bner basis, but this is not so helpful unless there is a predictable pattern to the new relations produced at each step, so that 
one can understand what this infinite set of relations actually is.  
%This is one of the main difference between the noncommutative and commutative theories; for commutative rings, the process 
%of finding a Grobner basis for an ideal of $k[x_1, \dots, x_n]$  (also known as Buchberger's algorithm) always terminates.  

\begin{example} 
\label{ex:finiteprocess}
Let $A = k \langle x, y, z \rangle/(z^2 -xy - yx, zx - xz, zy- yz)$.  The reader may easily check that $g_1 = z^2 - xy -yx, g_2 = zx - xz, g_3 = zy-yz$ is not a Gr{\"o}bner basis under degree lex order with $x < y < z$, but attempting to resolve the overlap ambiguities $z^2x$ and $z^2y$ lead by the process described above 
to new relations $g_4 = yx^2 - x^2y$ and $g_5 = y^2x - xy^2$ such that $\{ g_1, \dots, g_5 \}$ is a Gr{\"o}bner basis 
(Exercise 1.4).
\end{example}

\begin{example}
Let $A = k \langle x, y \rangle/(yx-xy -x^2)$ be the Jordan plane, but take degree lex order with $y < x$ 
instead so that $x^2$ is now the leading term.  It overlaps itself, and one may check that the overlap ambiguity does not resolve.  In 
this case the algorithm of checking overlaps and adding new relations never terminates, but there 
is a pattern to the new relations added, so that one can write down the infinite Gr{\"o}bner basis given by the infinite process.
The reader may attempt this calculation by hand, or by computer (Exercise 1.7(b)).   This example shows that whether or not the algorithm terminates is sensitive even to the choice of ordering.  
\end{example}

\subsection{Graded Ext and minimal free resolutions}

The main purpose of this section is to describe the special features of projective resolutions and Ext for graded modules 
over a graded ring.  Although we remind the reader of the basic definitions, the reader encountering Ext for the first 
time might want to first study the basic concept in a book such as \cite{Rot}.    Some facts we need below are 
left as exercises for a reader with some experience in homological algebra, or they may be taken on faith.

Recall that a right module $P$ over a ring $A$ is \emph{projective} if, whenever $A$-module homomorphisms $f: M \to N$ and 
$g: P \to N$ are given, with $f$ surjective, then there exists a homomorphism 
$h: P \to M$ such that $f \circ h = g$.  It is a basic fact that a module $P$ is projective if and only if there is a module $Q$ such that 
$P \oplus Q$ is a free module; in particular, free modules are projective.
%\cite[]{Rot}.
%, and in fact we will actually only need free modules in this course
%for reasons that will be apparent shortly.
A \emph{projective resolution} of an $A$-module $M$ is a complex of $A$-modules and $A$-module homomorphisms, 
\begin{equation}
\label{ex1-eq}
\dots \to P_n \overset{d_{n-1}}{\to} P_{n-1} \to \dots \overset{d_1}{\to} P_1 \overset{d_0}{\to} P_0 \to 0,
\end{equation}
together with a surjective \emph{augmentation map} $\epsilon: P_0 \to M$, such that 
each $P_i$ is projective, and the sequence 
\begin{equation}
\label{ex2-eq}
\dots \to P_n \overset{d_{n-1}}{\to} P_{n-1} \to \dots \overset{d_1}{\to} P_1 \overset{d_0}{\to} P_0 \overset{\epsilon}{\to} M \to 0
\end{equation}
is exact.  
%We call \eqref{ex2-eq} the \emph{augmented} projective resolution of $M$.
%Once we discuss Ext in a moment it will be clear why it is convenient to refer to \eqref{ex1-eq} as the projective resolution %rather than \eqref{ex2-eq}.
% note that both sequences contain the same information.  
Another way of saying that \eqref{ex1-eq} is a projective resolution of $M$ is to say that it is a complex $P_{\bullet}$ of projective $A$-modules, with homology groups $H_i(P_{\bullet}) = 0$ for $i \neq 0$ and $H_0(P_{\bullet}) \cong M$.
Since every module is a homomorphic image of a free module, every 
module has a projective resolution.   
\begin{comment}
The \emph{projective dimension} $\pd(M)$ of a nonzero module $M$ is the smallest $n \geq 0$, if any, such that there exists a projective resolution of $M$ of the form
\[
0 \to P_n \overset{d_{n-1}}{\to} P_{n-1} \to \dots \overset{d_1}{\to} P_1 \overset{d_0}{\to} P_0 \to 0; 
\]
if no such $n$ exists we put $\pd(M) = \infty$.  The (right) global dimension of $A$, $\rgl(A)$, is the supremum of the projective dimensions of all right $A$-modules $M$.  The left projective dimension $\lgl(A)$ is defined analogously in terms 
of left modules.  
\end{comment}

Given right $A$-modules $M$ and $N$, there are Abelian groups $\Ext^i_A(M,N)$ for each $i \geq 0$.  
To define them, take any projective resolution of $M$, say 
\[
\dots \to P_n \overset{d_{n-1}}{\to} P_{n-1} \to \dots \overset{d_1}{\to} P_1 \overset{d_0}{\to} P_0 \to 0,
\]
and apply the functor $\Hom_A(-, N)$ to the complex (which reverses the direction of the maps), obtaining a complex 
of Abelian groups
\[
\dots \la \Hom_A(P_n, N) \overset{d_{n-1}^*}{\la} \Hom_A(P_{n-1},N) \la \dots \overset{d_1^*}{\la} \Hom_A(P_1, N) \overset{d_0^*}{\la} \Hom_A(P_0, N) \la 0.
\]
Then $\Ext^i(M,N)$ is defined to be the $i^{\small \text{th}}$ homology group $\ker d_i^*/\im d_{i-1}^*$ of this complex.   These groups 
do not depend up to isomorphism on the choice of projective resolution of $M$, and moreover $\Ext^0_A(M,N) \cong \Hom_A(M,N)$ \cite[Corollary 6.57,  Theorem 6.61]{Rot}.   %One can also compute $\Ext^i(M, N)$ using an injective resolution of $N$, but we will not need this.

For the rest of the section we consider the special case of the definitions above where $A$ is \fig.  
%A homomorphism of graded modules $\phi: M \to N$ is a \emph{graded homomorphism} if $\phi(M_n) %\subseteq N_n$ for all $n$.  
%The main new feature of the graded case is the following: given a graded $A$-module $M$, it is important to construct a projective resolution of $M$ in which the projective modules are graded and the maps are graded homomorphisms.    
A \emph{graded free module} over a finitely graded algebra $A$ 
is a direct sum of shifted copies of $A$, that is $\bigoplus_{\alpha \in I} A(i_{\alpha})$.  
A graded module $M$ is \emph{left bounded} if $M_n = 0$ for $n \ll 0$.    
For any $m \geq 0$, we write $A_{\geq m}$ as shorthand for $\bigoplus_{n \geq m} A_n$.
Because $A$ has a unique homogeneous maximal ideal $A_{\geq 1}$, in many ways the theory for \fig\ algebras mimics the theory for local rings.  For example, we have the following graded version of Nakayama's lemma.  
\begin{lemma}
\label{lem:nak}
Let $A$ be a \fig\ $k$-algebra.  Let $M$ be a left bounded graded $A$-module.   If $M A_{\geq 1} = M$, then $M = 0$.  Also, a set of homogeneous elements $\{m_i \} \subseteq M$ generates $M$ as an $A$-module if and only if the images of the $m_i$ span  $M/MA_{\geq 1}$ as a $A/A_{\geq 1}= k$-vector space.
\end{lemma}
\begin{proof}
The first statement is easier to prove than the ungraded version of Nakayama's lemma: if $M$ is nonzero, and $d$ is the minimum degree such that $M_d \neq 0$, then $MA_{\geq 1}$ is contained in degrees greater than or equal to  $d+1$, so $M A_{\geq 1} = M$ is impossible.  The second statement follows by applying the first statement to $N = M/(\sum m_i A)$.
\end{proof}

Given a left bounded graded module $M$ over a \fig\ algebra $A$, a set $\{ m_i \} \subseteq M$ of homogeneous generators is said to \emph{minimally} generate $M$ if the images of the $m_i$ in $M/MA_{\geq 1}$ are a $k$-basis.  
In this case we can construct a surjective graded $A$-module homomorphism $\phi: \bigoplus_i A(-d_i) \to M$, where $d_i = \deg(m_i)$ and the $1$ of the $i^{\small \text{th}}$ copy of $A$ maps to $m_i \in M$; then the $k$-vector space map $\bigoplus_i A/A_{\geq 1}(-d_i) \to M/MA_{\geq 1}$ induced by $\phi$ is an isomorphism.  We call $\phi$ a \emph{minimal} surjection of a graded free module onto $M$
%, and clearly the graded free module is determined up to isomorphism by $M$.  
%Note in particular that this applies to any finitely generated graded module $M$ over $A$, and that the free module has finite rank %in this case.  
A graded free resolution of $M$ of the form 
\begin{equation}
\label{eq:grfree}
\dots \to \bigoplus_i A(-a_{n,i}) \overset{d_{n-1}}{\to} \dots \overset{d_1}{\to} \bigoplus_i A(-a_{1,i}) \overset{d_0}{\to} \bigoplus_i A(-a_{0,i}) \to 0
\end{equation}
is called \emph{minimal} if each $d_i$ is a minimal surjection onto $\im d_i$ for all $i \geq 0$, 
and the augmentation map $\epsilon: \bigoplus_i A(-a_{0,i}) \to M$ is a minimal surjection onto $M$.  
For any left bounded graded module $M$ over $A$, the kernel of a minimal surjection of a graded free module onto $M$ 
is again left bounded.  Thus, by induction a minimal graded free resolution of $M$ always exists.

\begin{comment}

\item Every projective, graded, left bounded $A$-module is a graded free module.
\end{comment}

In general, projective resolutions of a module $M$ are 
unique only up to homotopy \cite[Theorem 6.16]{Rot}, but for minimal graded free resolutions we have the following stronger form of uniqueness.
\begin{lemma}
\label{lem:min}
Let $M$ be a left bounded graded right module over a \fig\ algebra $A$.
\begin{enumerate}
\item A graded free resolution $P_{\bullet}$ of $M$ is minimal if and only if $d_i: P_{i+1} \to P_i$ has 
image inside $P_i A_{\geq 1}$ for each $i \geq 0$.

\item Any two minimal graded 
free resolutions $P_{\bullet}$ and $Q_{\bullet}$ of $M$ are isomorphic as complexes;  that is, 
there are graded module isomorphisms $f_i: P_i \to Q_i$ for each $i$ giving a chain map.   In particular, 
the ranks and shifts of the free modules occurring in a free resolution of $M$ are invariants of $M$.
\end{enumerate}
\end{lemma}
\begin{proof}
We provide only a sketch, leaving some of the details to the reader.    The proof is similar to the proof of the analogous fact for commutative local rings, 
for example as in \cite[Theorem 20.2]{Ei}.

(1) Consider a minimal surjection $\phi$ of a graded free module $P$ onto a left bounded module $N$, and the resulting exact sequence $0 \to K \overset{f}{\to} P \overset{\phi}{\to} N \to 0$, where $K = \ker \phi$.  Applying $- \otimes_A A/A_{\geq 1}$ gives an exact sequence $K/KA_{\geq 1} \overset{\overline{f}}{\to} P/PA_{\geq 1} \overset{\overline{\phi}}{\to} N/NA_{\geq 1} \to 0$.  By definition $\overline{\phi}$ is an isomorphism, forcing $\overline{f} = 0$ and $K \subseteq PA_{\geq 1}$.   If $P_{\bullet}$ is a minimal projective resolution, then each $\im d_i$ is the kernel of some minimal surjection and so $\im d_i \subseteq P_i A_{\geq 1}$ for each $i \geq 0$ by the argument above.  The converse is proved by essentially the reverse argument.

(2)  As in part (1), consider the exact sequence $0 \to K \overset{f}{\to} P \overset{\phi}{\to} N \to 0$ with $\phi$ a minimal surjection, and suppose that there is another minimal surjection $\psi: Q \to N$ leading to an exact sequence 
$0 \to K' \overset{g}{\to} Q \overset{\psi}{\to} N \to 0$.  If $\rho: N \to N$ is any graded isomorphism,  then there is an induced 
isomorphism $\overline{\psi}^{-1} \circ \overline{\rho} \circ \overline{\phi} : P/PA_{\geq 1} \to Q/QA_{\geq 1}$, which lifts by projectivity to an isomorphism of graded free modules $h: P \to Q$.  Then $h$ restricts to an isomorphism $K \to K'$.  Part (2) follows by applying this argument inductively to construct the required isomorphism of complexes, beginning with the 
identity map $M \to M$.
\end{proof}

In the graded setting that concerns us in these notes, it is most appropriate to use a graded version of $\Hom$.  For graded modules $M,N$ over a finitely graded algebra $A$, let $\Hom_{\rgr A}(M, N)$ be the 
vector space of graded (that is, degree preserving) module homomorphisms from $M$ to $N$.
  Then we define 
$\underline{\Hom}_A(M,N) = \bigoplus_{d \in \mb{Z}} \Hom_{\rgr A}(M, N(d))$, as a graded vector space.
It is not hard to see that there is a natural inclusion $\underline{\Hom}_A(M,N) \subseteq \Hom_A(M,N)$, and that this is 
an equality when $M$ is finitely generated (or more generally, generated by elements in some finite set of degrees).    We are primarily concerned with finitely generated modules in these notes, 
and so in most cases so there is no difference between $\underline{\Hom}$ and $\Hom$, but for consistency we will use the graded Hom functors $\underline{\Hom}$ throughout.   Similarly as in the ungraded case, for graded modules $M$ and $N$ we define $\underline{\Ext}^i_A(M,N)$ by taking a graded free resolution of $M$, applying the functor $\underline{\Hom}_A(-, N)$, 
and taking the $i^{\small \text{th}}$ homology.  Then $\underline{\Ext}^i_A(M,N)$ is a graded $k$-vector space.

\begin{comment}
As in the ungraded case, $\underline{\Ext}^i(M,N)$ is again a module (and not just a $k$-space) when $M$ or $N$ has 
a bimodule structure.  Recall that an abelian group $M$ is an $(A, B)$-bimodule for rings $A$ and $B$ if $M$ is a left 
$A$-module and a right $B$-module, and $(am)b = a(mb)$ for all $m \in M, a \in A, b \in B$.   Suppose that $M$ and $N$ 
are graded right $A$-modules, and consider $\underline{\Hom}^i(M,N)$.  Then if $M$ is a $(B, A)$-bimodule, then 
$\underline{\Hom}^i(M,N)$ is right $B$-module via $[\phi \dot b](m) = \phi(mb)$, while if $N$ is a $(C, A)$-bimodule, then 
$\underline{\Hom}^i(M,N)$ is a left $C$-module via $[b \dot \phi](m) = b\phi(m)$.  Of course, if both $M$ and $N$ have 
such extra structure then $\underline{\Hom}^i(M,N)$ is an $(C, B)$-bimodule.   In each case, $\underline{\Ext}^i(M,N)$ obtains 
the same extra module structures as $\underline{\Hom}(M,N)$.  See \cite{Ro} for more information.  
\end{comment}

As we will see, the graded free resolution of the \emph{trivial module} $k = A/A_{\geq 1}$ of a finitely graded algebra $A$ will be 
of primary importance.  In the next example, we construct such a resolution explicitly.   In general, given a right $A$-module $M$ we 
sometimes write $M_A$ to emphasize the side on which $A$ is acting.  Similarly, $_A N$ indicates a left 
$A$-module.  This is especially important for bimodules such as the trivial module $k$, which has 
both a natural left and right $A$-action.
\begin{example}
\label{ex:jordan-res}
Let $A = k \langle x, y \rangle/(yx - xy - x^2)$ be the Jordan plane.
Then we claim that the minimal graded free resolution of $k_A$ has the form  
\begin{equation}
\label{jor-eq}
0 \to A(-2) \overset{d_1 = \begin{pmatrix} -y-x \\ x 
\end{pmatrix}}{\lra} A(-1)^{\oplus 2} \overset{d_0 = \begin{pmatrix} x & y \end{pmatrix}}{\lra} A \to 0.
\end{equation}
Here, we think of the free right modules as column vectors of elements in $A$, and the maps 
as left multiplications by the indicated matrices.   Thus $d_1$ is $c \mapsto \begin{pmatrix} (-y-x)c \\ xc \end{pmatrix}$ and $d_0$ is $\begin{pmatrix} a \\b \end{pmatrix} \mapsto xa + yb$.
Note that left multiplication by a matrix is a right module homomorphism.  

If we prove that \eqref{jor-eq} is a graded free resolution of $k_A$, it is automatically minimal by Lemma~\ref{lem:min}. 
Recall that $A$ is a domain with $k$-basis $\{ x^i y^j  | i, j \geq 0 \}$ (Examples~\ref{ex:jordan} and Corollary~\ref{cor:nd}).
It is easy to see that the sequence above is a complex; this amounts to the calculation 
\[
\begin{pmatrix} x & y \end{pmatrix} \begin{pmatrix} -y-x \\ x \end{pmatrix} = (x(-y-x) + yx) = (0), 
\]
using the relation.  The injectivity of $d_1$ is clear because $A$ is a domain, and the cokernel of $d_0$ 
is obviously isomorphic to $k = A/A_{\geq 1}$.  Exactness in the middle may be checked in the following way.
Suppose we have an element $\begin{pmatrix} a \\b \end{pmatrix} \in \ker d_0$.  Since $\im d_1 = \bigg \{ \begin{pmatrix} (-y-x)h \\ xh \end{pmatrix} \bigg| h \in A  \bigg \}$, using that $\{ x^i y^j \}$ is a $k$-basis for $A$, after subtracting an 
element of $\im(d_1)$ from $\begin{pmatrix} a \\b \end{pmatrix}$ we can assume that $b = f(y)$ is a polynomial 
in $y$ only.  Then $xa + yb = xa + y f(y) = 0$, which using the form of the basis again implies that $a = b = 0$.
In the next lecture we will prove a general result, Lemma~\ref{lem:begin}, from which exactness in the middle spot 
follows automatically.
\end{example}

One may always represent maps between graded free modules as matrices, as we did in the previous result.  However, one needs to be careful about conventions.  Given a finitely graded $k$-algebra $A$, it is easy to verify that any graded right $A$-module homomorphism $\phi: \bigoplus_{i = 1}^m A(-s_i) \to \bigoplus_{j = 1}^n A(-t_j)$ between two graded free right $A$-modules of finite rank can be represented as left multiplication by an $n \times m$ matrix of elements of $A$, where the 
elements of the free modules are thought of as column vectors.   On the other hand, any graded left module homomorphism 
$\psi: \bigoplus_{i = 1}^m A(-s_i) \to \bigoplus_{j = 1}^n A(-t_j)$ can be represented as \emph{right} multiplication by an $m \times n$ matrix of elements of $A$, where the 
elements of the free modules are thought of as \emph{row} vectors.   A little thought shows that the side we multiply the matrix on, and whether we consider row or column vectors, is forced by the situation.
The following is an easy reinterpretation of Lemma~\ref{lem:min}(1):
\begin{lemma}
\label{lem:inA1}
Let $A$ be a \fig\ $k$-algebra, and let $M$ be a left bounded graded right $A$-module.  If 
$M$ has a graded free resolution $P_{\bullet}$ with all $P_i$ of finite rank, then 
the resolution is minimal if and only if all entries of the matrices representing the maps $d_i$ lie in $A_{\geq 1}$.  \hfill $\Box$
\end{lemma}

It will be important for us to understand the action of the functor $\underline{\Hom}_A(-, A)$ on maps between 
graded free modules.    If $N$ is a graded right module over the finitely graded algebra $A$, then $\underline{\Hom}_A(N, A)$ is a left $A$-module via $[a \cdot \psi](x) = a\psi(x)$.  Given a homomorphism $\phi: N_1 \to N_2$ of graded right $A$-modules, then the induced map 
$\underline{\Hom}_A(N_2, A) \to \underline{\Hom}_A(N_1, A)$ given by $f \mapsto f \circ \phi$ is a left $A$-module homomorphism.  
Then from the definition of $\underline{\Ext}$, it is clear that $\underline{\Ext}^i_A(N,A)$ is a graded left $A$-module, for each graded right module $N$.  A similar argument shows that in general, if $M$ and $N$ are right $A$-modules and $B$ is another ring, then $\uExt^i_A(M,N)$ obtains a left $B$-module structure if $N$ is a $(B, A)$-bimodule, or a right $B$-module structure if $M$ is a $(B, A)$-bimodule (use \cite[Proposition 2.54]{Rot} and 
a similar argument), but we will need primarily the special case we have described explicitly.

\begin{lemma}
\label{lem:homtoA}
Let $A$ be a finitely graded $k$-algebra.  
\begin{enumerate}
\item For any graded free right module $\bigoplus_{i = 1}^m A(-s_i)$, 
there is a canonical graded left $A$-module isomorphism 
\[
\uHom_A(\bigoplus_{i = 1}^m A(-s_i), A) \cong \bigoplus_{i = 1}^m A(s_i).
\]

\item Given a graded right-module homomorphism $\phi: P = \bigoplus_{i = 1}^m A(-s_i) \to Q = \bigoplus_{j = 1}^n A(-t_j)$, represented by left multiplication by the matrix $M$, then applying $\underline{\Hom}( - , A)$ gives a left module 
homomorphism 
$\phi^*: \underline{\Hom}_A(Q,A) \to \underline{\Hom}_A(P, A)$, which we can canonically identify with a graded left-module map 
$\phi^*: \bigoplus_{j = 1}^n A(t_j) \to \bigoplus_{i = 1}^m A(s_i)$, using part (1).  Then $\phi^*$ is given by right multiplication by the same matrix $M$.
\end{enumerate}
\end{lemma}
\begin{proof}
The reader is asked to verify this as Exercise 1.5.
\end{proof}

We now revisit Example~\ref{ex:jordan-res} and calculate $\underline{\Ext}^i_A(k, A)$.
\begin{example}
\label{ex:jordan-ext}
Let $A = k \langle x, y \rangle/(yx - xy - x^2)$ be the Jordan plane, and consider 
the graded free resolution \eqref{jor-eq} of $k_A$.  Applying $\uHom_A(-, A)$ and using Lemma~\ref{lem:homtoA}, we get the complex of left modules
\[
0 \leftarrow A(2) \overset{d_1^* = \begin{pmatrix} -y-x \\ x \end{pmatrix}}{\longleftarrow} A(1)^{\oplus 2} \overset{d_0^* = \begin{pmatrix} x & y \end{pmatrix}}{\longleftarrow} A \leftarrow 0,
\]
where as described above, the free modules are row vectors and the maps are right multiplication by the matrices.   
By entirely analogous arguments as in Example~\ref{ex:jordan-res}, 
we get that this complex is exact except at the $A(2)$ spot, where $d_1^*$ has image $A_{\geq 1}(2)$.  
Thus $\uExt^i(k_A, A) = 0$ for $i \neq 2$, and $\uExt^2(k_A, A) \cong {}_A k(2)$.

Notice that this actually shows that the complex above is a graded free resolution of the left $A$-module $_A k(2)$.  We will see 
in the next section that this is the key property of an Artin-Schelter regular algebra: the minimal free resolutions of $_A k$ and $k _A$ are interchanged  (up to shift of grading) by applying $\uHom_A(-, A)$.
\end{example}

%We close this section with an important fact about global dimension which is specific to the graded case.
Let $A$ be a \fig\ algebra.  We will be primarily concerned with the category of graded modules over $A$, and so we will work 
exclusively throughout these lectures with the following graded versions of projective and global dimension.  Given a $\mb{Z}$-graded right $A$-module $M$, its \emph{projective dimension} $\pd(M)$ is the minimal $n$ such that there is a projective resolution of $M$ of length $n$, that is 
\[
0 \to P_n \overset{d_{n-1}}{\to} P_{n-1} \to \dots \overset{d_1}{\to} P_1 \overset{d_0}{\to} P_0 \to 0,
\]
in which all projectives $P_i$ are graded and all of the homomorphisms of the complex, as well as the augmentation map, 
are graded homomorphisms.  If no such $n$ exists, then $\pd(M) = \infty$.  The \emph{right global dimension} of $A$, 
$\rgl(A)$, is defined to be the supremum of the projective dimensions of all $\mb{Z}$-graded right $A$-modules, and 
the \emph{left global dimension} $\lgl(A)$ is defined analogously in terms 
of graded left modules.  
\begin{proposition}
\label{prop:gldim}
Let $A$ be a \fig\ $k$-algebra.  Then 
\[
\rgl(A) = \pd(k_A) = \pd({}_A k) = \lgl(A),
\]
and this number is equal to the length of the minimal graded free resolution of $k_A$.  
\end{proposition}
\begin{proof}
The reader familiar with the basic properties of $\Tor$ can attempt this as Exercise 1.6, or see \cite[Section 2]{ATV1} for the basic idea.
\end{proof}
\noindent  Because of the result above, for a \fig\ algebra we will write the common value 
of $\rgl(A)$ and $\lgl(A)$ as $\gldim(A)$ and call this the \emph{global dimension} of $A$.
For example, the result above, together with Example~\ref{ex:jordan-res}, shows that the Jordan plane 
has global dimension $2$.

\subsection{Exercise Set 1}

\
\bigskip

1.  Let $A = k[x_1, \dots, x_n]$ be a commutative polynomial algebra with weights $\deg x_i = d_i$, 
and let $F = k \langle x_1, \dots, x_n \rangle$ be a free associative algebra with weights $\deg x_i = d_i$.

(a).  Prove that $h_A(t) = 1/p(t)$ where $p(t) = \prod_{i = 1}^n (1-t^{d_i})$.  (Hint:  induct on $n$).

(b).  Prove that $h_F(t) = 1/q(t)$ where $q(t) = 1 - \sum_{i = 1}^n t^{d_i}$.  (Hint:  write $h_F(t) = \sum a_i t^i$ and prove 
by counting that $a_j = \sum_{i = 1}^n a_{j-d_i}$).

\bigskip

2.  Prove Lemma~\ref{lem:passup}(1).  Also, use Lemma~\ref{lem:passup} to prove that Example~\ref{ex:cubic} is a noetherian domain (consider the element $xy -yx$).

\bigskip

3.  Complete the proof of Theorem~\ref{thm:diamond}.  Namely, prove that the set of reduction unique elements is a $k$-subspace and that the operator $\operatorname{red}(-)$ is linear on this subspace, and prove the implications $(3) \implies (2)$ and $(2) \implies (1)$.

\bigskip

4.  Verify the calculations in Example~\ref{ex:finiteprocess}, find explicitly the set of reduced words with respect to the Grobner basis $g_1, \dots, g_5$, and show that the Hilbert series of the algebra is $1/(1-t)^3$.

\bigskip

5.  Prove Lemma~\ref{lem:homtoA}.

\bigskip

6.  (This exercise assumes knowledge of the basic properties of $\tor$).
Prove Proposition~\ref{prop:gldim}, using the following outline.  

Assume that the global dimension of $A$ is equal to the supremum of the projective dimensions of 
all finitely generated graded $A$-modules.  (The general argument that it suffices to consider finitely generated modules 
in the computation of global dimension is due to Auslander, see \cite[Theorem 1]{Aus}.)  In particular, we only need to look at left bounded graded $A$-modules.

(a).  Show that if $M_A$ is left bounded, then its minimal graded free resolution $P_{\bullet}$ is isomorphic to a direct 
summand of any graded projective resolution.  Thus the length of the minimal graded free resolution is equal to $\pdim(M_A)$. 

(b).  Show that if $M_A$ is left bounded, then $\pdim(M_A) = \max \{i | \tor^i_A(M_A, {}_A k)  \neq 0 \}$.

(c).  Show that if $M$ is left bounded then $\pdim(M_A) \leq \pdim({}_A k)$.  A left-sided version of the same argument shows 
that $\pdim({}_A N) \leq \pdim(k_A)$ for any left bounded graded left module $N$.

(d).  Conclude that $\pdim({}_A k) = \pdim(k_A)$ and that this is the value of both $\lgl(A)$ and $\rgl(A)$.

\bigskip

\begin{comment}
8.  (This exercise requires basic knowledge of the theory of Lie algebras).  
Let $L$ be a finite-dimensional Lie algebra, say with $k$-basis $x_1, \dots, x_n$, and let $U(L)$ be its universal enveloping algebra, 
which is the finitely presented (ungraded) algebra $k \langle x_1, \dots, x_n \rangle/(g_{ij} | 1 \leq i < j \leq n)$ where
$g_{ij} = x_jx_i - x_ix_j - [x_i,x_j]$, where $[-,-]$ is the bracket in $L$.  The Poincare-Birkhoff-Witt (PBW) theorem is one of the basic results 
in the subject of Lie algebras:  it states that $\{x_1^{i_1} \dots x_n^{i_n} | i_j \geq 0 \}$ is a $k$-basis for $U(L)$.  While there are messy direct proofs of the PBW theorem, once one develops the diamond lemma there is a very simple proof, as follows:

Show that with respect to the ordering $x_1 < \dots < x_n$, the set of relations $\{ g_{ij} \}$ is a Gr{\"o}bner basis for $U(L)$, and show that 
this implies the PBW theorem.
\end{comment}

7.  In this exercise, the reader will use the computer algebra system GAP to get a feel for how one can calculate noncommutative Gr{\"o}bner bases using software.   To do these exercises you need a basic GAP installation, together with the 
noncommutative Gr{\"o}bner bases package (GBNP).   After opening a GAP window, 
the following code should be run once:

\begin{verbatim}
LoadPackage("GBNP");
SetInfoLevel(InfoGBNP,0);
SetInfoLevel(InfoGBNPTime,0);
\end{verbatim}

(a).  Type in and run the following code, which shows that the three defining relations of a quantum polynomial ring 
are already a  Gr{\"o}bner basis.  It uses degree lex order with $x < y < z$, and stops after checking all overlaps (and possibly adding new relations if necessary) up to degree 12.  The vector $[1, 1, 1]$ is the list of weights of the variables.

\begin{verbatim}
A:=FreeAssociativeAlgebraWithOne(Rationals, "x", "y", "z");
x:=A.x;; y:=A.y;; z:=A.z;; o:=One(A);;
uerels:=[y*x - 2*x*y, z*y - 3*y*z, z*x - 5*x*z];
uerelsNP:=GP2NPList(uerels);;
PrintNPList(uerelsNP);
GBNP.ConfigPrint(A);
GB:=SGrobnerTrunc(uerelsNP, 12, [1,1,1]);;
PrintNPList(GB);
\end{verbatim}

\medskip

(b).  Change the first three lines of the code in (a) to 
\begin{verbatim}
A:=FreeAssociativeAlgebraWithOne(Rationals, "y", "x");
x:=A.x;; y:=A.y;; o:=One(A);;
uerels:=[y*x - x*y - x*x];
\end{verbatim}
and run it again.  This attempts to calculate a Gr{\"o}bner basis for the Jordan plane with the ordering $y < x$ 
on the variables.  Although it stops in degree 12, the calculation suggests that this calculation would continue infinitely, 
and the pattern to the new relations produced suggest what the infinite Gr{\"o}bner basis produced by the process is.   (One can easily prove for certain that this infinite set of relations is inded a Gr{\"o}bner basis by induction.)  Calculate 
the corresponding $k$-basis of reduced words and verify that it gives the same Hilbert series $1/(1-t)^2$ we already know.

\medskip

(c).   Change the third line in code for (a) to 
\begin{verbatim}
uerels:=[2*y*x + 3*x*y + 5*z*z, 2*x*z + 3*z*x + 5*y*y, 2*z*y + 3*y*z + 5*x*x];
\end{verbatim}
and change the number $12$ in the seventh line to $8$.
Run the code again and scroll through the output.  Notice how quickly the new relations added in the Diamond 
Lemma become unwieldy and complicated.  Certainly the calculation suggests that the algorithm is unlikely to terminate, and 
it is hard to discern any pattern in the new relations added by the algorithm.  This shows the resistance of the Sklyanin algebra to Gr{\"o}bner basis methods.

\medskip

(d).  Replace the third line in (a) with 
\begin{verbatim}
uerels:=[z*x  + x*z, y*z  + z*y, z*z - x*x - y*y];
\end{verbatim}
and run the algorithm again.  This is an example where the original relations are not a Gr{\"o}bner basis, but the algorithm 
terminates with a finite Gr{\"o}bner basis.  This basis can be used to prove that the algebra has Hilbert series $1/(1-t)^3$, similarly as in Exercise 1.4.

\section{Lecture 2: Artin-Schelter regular algebras}

With the background of Lecture 1 in hand, we are now ready to discuss regular algebras.

\begin{definition}
\label{def:AS}
Let $A$ be a \fig\ $k$-algebra, and let $k = A/A_{\geq 1}$ be the trivial module.  We say that $A$ is \emph{Artin-Schelter 
regular} if
\begin{enumerate}
\item $\gldim(A) = d < \infty$; 
\item $\GK(A) < \infty$; and 
\item $\uExt^i_A(k_A, A_A) \cong \begin{cases} 0 & i \neq d \\ {} _A k(\ell) & i = d \end{cases}$ (as left $A$-modules),
\end{enumerate}
for some shift of grading $\ell \in \mb{Z}$.

We call an algebra satisfying (1) and (3) (but not necessarily (2)) \emph{weakly Artin-Schelter regular}.
\end{definition}
 
This definition requires some further discussion.  From now on, we will abbreviate 
``Artin-Schelter" to ``AS", or sometimes even omit the ``AS", but the reader should be 
aware that there are other notions of regularity for rings.  We want to give a more intuitive interpretation 
of condition (3), which is also known as the \emph{(Artin-Schelter) Gorenstein} condition.   
By Proposition~\ref{prop:gldim}, the global dimension of $A$ is 
the same as the length of the minimal graded free resolution of $k_A$.  
Thus condition (1) in Definition~\ref{def:AS} is equivalent to the requirement that 
the minimal graded free resolution of the trivial module $k_A$ has the form
\begin{equation}
\label{eq:res1}
%\label{res1-eq}
0 \to P_d \to P_{d-1} \to \dots \to P_1 \to P_0 \to 0,
\end{equation}
for some graded free modules $P_i$.   It is not immediately obvious that the $P_i$ must have finite rank, unless we also happen to know that $A$ is noetherian.   It is a fact, however, that for a weakly AS-regular algebra the $P_i$ must have finite rank (\cite[Proposition 3.1]{SZ1}).   This is not a particularly difficult argument, but we will simply assume this result here, since in the main examples we consider it will obviously hold by direct computation.  The minimality of the resolution means that the matrices representing the maps in the complex $P_{\bullet}$ have all of their entries in $A_{\geq 1}$  (Lemma~\ref{lem:inA1}).

To calculate $\uExt^i_A(k_A, A_A)$, we apply $\uHom_A(-, A)$ to \eqref{eq:res1} and take homology.  Thus putting $P_i^* = \uHom_A(P_i,  A)$, we get a complex of graded free left modules 
\begin{equation}
\label{eq:res2}
0 \la P_d^* \la P_{d-1}^* \la \dots \la P_1^* \la P_0^* \la 0.
\end{equation}
By Lemma~\ref{lem:homtoA}, the matrices representing the maps of free left modules in this complex are the same as in \eqref{eq:res1}  and so also have all of their entries in $A_{\geq 1}$.
The AS-Gorenstein condition (3) demands that the homology of this complex should be $0$ in all 
places except at $P_d^*$, where the homology should be isomorphic to $_A k(\ell)$ for some $\ell$. This is 
equivalent to \eqref{eq:res2} being a graded free resolution of $_A k(\ell)$ as a left module.
Moreover, this is necessarily a minimal graded free resolution because of the left-sided version of Lemma~\ref{lem:inA1}.
Equivalently, applying the shift operation $(-\ell)$ to the entire complex, 
\[
0 \la P_d^*(-\ell) \la P_{d-1}^*(-\ell) \la \dots \la P_1^*(-\ell) \la P_0^*(-\ell) \la 0
\]
is a minimal graded free resolution of $_A k$.  

Thus parts (1) and (3) of the definition of AS-regular together assert that the minimal free resolutions of $k_A$ and $_A k$ are finite in length, consist of finite rank free modules (\cite[Proposition 3.1]{SZ1}), and that the operation $\uHom_A(-, A_A)$  together with a shift of grading sends the first to the second.  Furthermore, the situation 
is automatically symmetric: one may also show that $\uHom_A(-, {}_A A)$ sends a minimal free resolution of $_A k$ to a shift of a minimal free resolution of $k_A$, or equivalently that $\uExt^i_A({}_A k, {}_A A) \cong  \begin{cases} 0 & i \neq d \\ {} k_A(\ell) & i = d, \end{cases}$ as right modules (Exercise 2.1).  
%Thus the definition of AS-regular is actually left-right symmetric, and the choice to state it using right modules was arbitrary.

When we talk about the dimension of an AS-regular algebra, we mean its global dimension.  However, all known examples 
satisfying Definition~\ref{def:AS} have $\gldim(A) = \GK(A)$, so there is no chance for confusion.
The term \emph{weakly AS-regular} is not standard.  Rather, some papers simply omit condition (2) in the definition 
of AS-regular.  It is useful for us to have distinct names for the two concepts here.
There are many more examples of weakly AS-regular algebras than there are of AS-regular algebras.  For example, see Exercise 2.4 below.

Condition (2) of Definition~\ref{def:AS} can also be related to the minimal free resolution of $k$: if the ranks and shifts of the $P_i$ occurring in \eqref{eq:res1} are known, then one may determine the Hilbert series of $A$, as follows.  Recall that when we have a finite length complex of finite-dimensional vector spaces over $k$, 
say 
\[
C:  \ \  0 \to V_n \to V_{n-1} \to \dots \to V_1 \to V_0 \to 0,
\]
then the alternating sum of the dimensions of the $V_i$ is the alternating sum of the dimensions of the homology groups, in other words 
\begin{equation}
\label{eq:euler}
\sum (-1)^i \dim_k V_i = \sum (-i)^i \dim_k H_i(C).
\end{equation}
%(If this fact is new to the reader, it is a straightforward exercise:  prove it first for a short exact sequence and then 
%use that $C$ may be stitched together from short exact sequences.)
Now apply \eqref{eq:euler} to the restriction to degree $n$ of the minimal free resolution \eqref{eq:res1} of $k_A$, obtaining 
\[
\sum_i (-1)^i \dim_k (P_i)_n = \begin{cases} 0 & n \neq 0 \\ 1 & n = 0 \end{cases}. 
\]
These equations can be joined together into the single power series equation 
\begin{equation}
\label{eq:hilbert}
1 = h_{k}(t) = h_{P_0}(t)  - h_{P_1}(t) + \dots + (-1)^dh_{P_d}(t).
\end{equation}
Moreover, writing $P_i = \bigoplus_{j = 1}^{m_i} A(-s_{i,j})$ for some integers $m_i$ and $s_{i,j}$, then 
because $h_{A(-s)}(t) = h_{A}(t) t^{s}$ for any shift $s$, we have 
$h_{P_i}(t) = \sum_j h_{A}(t) t^{s_{i,j}}$.  Thus \eqref{eq:hilbert} can be solved for $h_A(t)$, yielding  
\begin{equation}
\label{eq:hs}
h_A(t) = 1/p(t),\ \text{where}\ p(t) = \sum_{i,j} (-1)^i t^{s_{i,j}}.
\end{equation}
Note also that $P_0 = A$, while all other $P_i$ are sums of negative shifts of $A$.  Thus the constant term of $p(t)$ is equal to $1$.

%Note that we can write down the polynomial $p(t)$ immediately if we know which shifts of $A$ occur in the graded free modules 
%$P_i$ in the minimal free resolution of $k_A$.  
The GK-dimension of $A$  is easily determined from knowledge of $p(t)$, as follows.
\begin{lemma}
\label{lem:finGK}
Suppose that $A$ is a \fig\ algebra with Hilbert series $h_A(t) = 1/p(t)$ for some polynomial $p(t) \in \mb{Z}[t]$ with constant term $1$.  Then precisely one of the following occurs:
\begin{enumerate}
\item  All roots of $p(t)$ in $\mb{C}$ are roots of unity, $\GK(A) < \infty$, and $\GK(A)$ 
is equal to the multiplicity of $1$ as a root of $p(t)$; or else 

\item $p(t)$ has a root $r$ with $|r| < 1$, and $A$ has exponential growth; in particular, $\GK(A) = \infty$.
\end{enumerate}
\end{lemma}
\begin{proof}
This is a special case of \cite[Lemma 2.1, Corollary 2.2]{SZ1}.   Since $p(t)$ has constant term $1$, we may write $p(t) = \prod_{i =1}^m (1-r_i t)$, where the $r_i \in \mb{C}$ are the reciprocals of the roots of $p(t)$. 
%We also write $1/p(t) = \sum a_n t^n$ as an explicit power series.

Suppose first that $|r_i| \leq 1$ for all $i$.  Since the product of the $r_i$ is the leading coefficient of $p(t)$ (up to sign), it is an integer.  This forces $|r_i| = 1$ for all $i$.   This also implies that the leading coefficient of $p(t)$ must be $\pm 1$, so each root $1/r_i$ of $p(t)$ is an algebraic integer.   Since the geometric series $1/(1-r_it) = 1 + r_it + r_i^2 t^2 + \dots $ has coefficients of norm $1$, and $1/p(t)$ is the product of $m$ such series, it easily follows that  $1/p(t) = \sum a_n t^n$ has $|a_n| \leq \binom{n+m-1}{m-1}$ for all $n \geq 0$, so $\GK(A) \leq m$.   It follows from a classical theorem of Kronecker that algebraic integers on the unit circle are roots of unity.  
References for this fact and for the argument that the precise value of $\GK(A)$ is determined by the multiplicity of $1$ as a root of $p(t)$ may be found in \cite[Corollary 2.2]{SZ1}.

If instead $|r_i| > 1$ for some $i$, then we claim that $A$ has exponential growth.  
By elementary complex analysis the radius of convergence of $\sum a_n t^n$ is $r = \big (\limsup_{n \to \infty} |a_n|^{1/n} \big)^{-1}$.   
Since $A$ is finitely generated as an algebra, $A$ is generated by $V = A_0 \oplus \dots \oplus A_d$ for some $d \geq 1$.  Then it is easy to prove that $A_n \subseteq V^n$ for all $n \geq 1$.   If $A$ has subexponential growth, then by definition we have $\limsup_{n \to \infty} (\dim_k V^n)^{1/n} = 1$, and this certainly forces $\limsup_{n \to \infty} a_n^{1/n} \leq 1$.   In particular, 
the radius of convergence of $h_A(t)$ is at least as large as $1$, and so it converges at $1/r_i$.   Plugging in $1/r_i$ into $p(t) h_A(t) = 1$ now gives $0 = 1$, a contradiction.
\end{proof}

\begin{example}
Consider the Hilbert series calculated in Exercise 1.1.  Applying Lemma~\ref{lem:finGK} shows that a weighted polynomial ring in $n$ variables has GK-dimension $n$, and that a weighted free algebra in 
more than one variable has exponential growth (it is easy to see from the intermediate value theorem that a polynomial $ 1 - \sum_{i=1}^n t^{d_i}$ has a 
real root in the interval $(0, 1)$ when $n \geq 2$.)
%Example~\ref{ex:cubic} has Hilbert series $1/(1-t)^2(1-t^2)$, and thus by the lemma above its GK-dimension is $3$, as is 
%easy to verify directly as well.  
%We explicitly calculated $\Ext^i(k,A)$ for the Jordan plane $A$ in the first lecture.  These calculations show that 
%the Jordan plane is regular with global dimension 2.  A very similar calculation shows that the quantum plane is regular.
\end{example}

\subsection{Examples}

Our next goal is to present some explicit examples of AS-regular algebras.  One interpretation of 
regular algebras is that they are the noncommutative graded algebras most analogous to 
commutative (weighted) polynomial rings.   This intuition is reinforced by the following observation.
\begin{example}
\label{ex:koszul}
Let $A = k[x_1, \dots, x_n]$ be a commutative polynomial ring in $n$ variables, with any weights  $\deg x_i = d_i$.  
Then the minimal free resolution of $k$ is 
\[
0 \to A^{r_n} \to A^{r_{n-1}} \to \dots \to A^{r_1} \to A^{r_0} \to 0,
\]
where the basis of the free $A$-module $A^{r_j}$ is naturally indexed by 
the $j^{\small \text{th}}$ wedge in the exterior algebra in $n$ symbols $e_1, e_2, \dots, e_n$, that is, 
$\Lambda^j = k \{e_{i_1} \wedge e_{i_2} \wedge \dots \wedge e_{i_j} | i_1 \leq i_2 \leq \dots \leq i_j \}$,  
and $A^{r_j} = \Lambda^j \otimes_k A$.   In particular, $r_j = \binom{n}{j}$.  The differentials are defined on basis 
elements of the free modules by 
\[
d(e_{i_1} \wedge \dots \wedge e_{i_j}) = \sum_k (-1)^{k+1} (e_{i_1} \wedge \dots \wedge \widehat{e_{i_k}} \wedge \dots \wedge e_{i_j}) x_{i_k},
\]
where as usual $\widehat{e_{i_k}}$ means that element is removed.   The various shifts on the summands of each $A^{r_j}$ which are needed to make this a graded complex are not indicated above, but are clearly uniquely determined by the weights of the $x_i$.

The complex $P_{\bullet}$ above is well-known and is called the \emph{Koszul complex}: it is standard that it is a free resolution of $k_A$ \cite[Corollary 1.6.14]{BH}.  It is also graded once the appropriate shifts are introduced, and then it is minimal by Lemma~\ref{lem:inA1}.  The complex $\underline{\Hom}_A(P_{\bullet}, A_A)$ is isomorphic to a shift of the  Koszul complex again \cite[Proposition 1.6.10]{BH} (of course, left and right modules over $A$ are canonically identified), and thus the weighted polynomial ring $k[x_1, \dots, x_n]$ is AS-regular.  

In fact, it is known that a commutative AS-regular algebra must be isomorphic to a weighted polynomial ring;  see \cite[Exercise 2.2.25]{BH}.   
\end{example}

\begin{example}
\label{ex:jq}
The Jordan and quantum planes from Examples~\ref{ex:qplane} are AS-regular of dimension 2.  Indeed, the calculations in Examples~\ref{ex:jordan-res} and \ref{ex:jordan-ext} immediately imply that the Jordan plane is regular, and the argument for the quantum plane is very similar.  More generally, the quantum polynomial ring in $n$ variables from Examples~\ref{ex:qplane} is regular; we won't give a formal proof here, but the explicit resolution in case $n = 3$ appears in Example~\ref{ex:change} below. 
\end{example}

Before we give some examples of AS-regular algebras of dimension 3, it is useful to prove a 
general result about the structure of the first few terms in a minimal free resolution of $k_A$.  We call a set of 
generators for an $k$-algebra \emph{minimal} if no proper subset generates the algebra.  Similarly, a minimal set of generators of an ideal of the free algebra is one such that no proper subset generates the same ideal (as a $2$-sided ideal).
\begin{lemma} 
\label{lem:begin}
Let $a_1, \dots, a_n$ be a minimal set of homogeneous generators for a \fig\ $k$-algebra $A$, where $\deg a_i = e_i$.
let $F = k \langle x_1, \dots, x_n \rangle$ be a weighted free algebra, where $\deg x_i = e_i$ for all $i$, and present 
$A$ as $A \cong F/I$ by sending $x_i$ to $a_i$.   Suppose that $\{ g_j | 1 \leq j \leq r \}$ is a minimal set of 
homogeneous generators of $I$ as a $2$-sided ideal, with $\deg g_j = s_j$, and 
write $g_j = \sum_i x_i g_{ij}$ in $F$.  
Let $\overline{x}$ be the image in $A$ of an element $x$ of $F$.  Then
\begin{equation}
\label{eq:begin}
\dots \lra \bigoplus_{j = 1}^r A(-s_j) \overset{d_1 = \begin{pmatrix}\overline{g_{ij}} \end{pmatrix}}{\lra} \bigoplus_{i=1}^n A(-e_i) \overset{d_0 =  \begin{pmatrix} \overline{x_1} & \dots & \overline{x_m} \end{pmatrix}}{\lra}  A \to 0
\end{equation} 
is the beginning of a minimal free resolution of $k_A$.  
\end{lemma}
\begin{proof}
The homology of the complex \eqref{eq:begin} at the $A$ spot is clearly $k_A$.   It is straightforward to check that a set of homogeneous elements minimally generates $A_{\geq 1}$ as a right ideal if and only if it is a minimal generating set for $A$ as a $k$-algebra.  Using this 
it follows that $d_0$ is a minimal surjection onto its image $A_{\geq 1}$.

Consider the right $A$-submodule $X = \ker d_0 = \{ (h_1, \dots, h_n)  | \sum_i \overline{x_i} h_i = 0 \}$ of $\bigoplus_{i=1}^n A(-e_i)$.  
We also define $Y = \{ (h_1, \dots, h_n) | \sum_i x_i h_i \in I \} \subseteq \bigoplus_{i=1}^n F(-e_i)$.
Suppose that we are given a set of elements $\{ f_{ij} | 1 \leq i \leq n, 1 \leq j \leq s \} \subseteq F$ such that $f_j = \sum_i x_i f_{ij} \in I$ for all $j$.
We leave it to the reader 
to prove that the following are all equivalent:
\begin{enumerate}
\item $\{ f_j \}$ generates $I$ as a 2-sided ideal of $F$.
\item $I = x_1 I + \dots + x_n I + \textstyle \sum_j f_j F$.
\item $Y = I^{\oplus n} + \sum_j (f_{1j}, \dots, f_{nj}) F$.
\item $X =  \sum_j (\overline{f_{1j}}, \dots, \overline{f_{nj}}) A$.
\end{enumerate}
It follows since $\{ g_j | 1 \leq j \leq r \}$ is a minimal set of homogeneous generators of $I$ that $d_1$ is a minimal surjection onto $X = \ker d_0$.
\end{proof}

\begin{remark}
\label{rem:fp}
As remarked earlier, the graded free modules in the minimal graded free resolution of $k_A$ for a weakly AS-regular algebra 
$A$ must be of finite rank \cite[Proposition 3.1]{SZ1}.  Thus Lemma~\ref{lem:begin}, together with the uniqueness of minimal graded free resolutions up to isomorphism (Lemma~\ref{lem:min}), shows that 
any weakly AS-regular algebra is finitely presented.  
\end{remark}

Proving that an algebra of dimension 3 is AS-regular is often straightforward if it has a nice Gr{\"o}bner basis.
%We give an explicit example next.  
\begin{example}
\label{ex:reg3}
Fix nonzero $a, b \in k$, and 
let 
\[
A = k \langle x, y, z \rangle/(zy + (1/b) x^2 - (a/b)yz, zx - (b/a)xz - (1/a)y^2, yx - (a/b)xy).
\]
We leave it to the reader to show that under degree lex order with $x < y< z$, the single overlap ambiguity $zyx$ resolves, 
and so the algebra $A$ has a $k$-basis $\{ x^i y^j z^k | i, j, k \geq 0 \}$ and Hilbert 
series $1/(1-t)^3$.   

We claim that 
a minimal graded free resolution of $k_A$ has the form
\[
0 \to A(-3) \overset{\begin{pmatrix}  x\\(-a/b)y \\ z \end{pmatrix}}{\lra} A(-2)^{\oplus 3} \overset{\begin{pmatrix} (1/b)x & -(b/a)z & -(a/b)y \\ -(a/b)z & -(1/a)y & x \\ y & x & 0 \end{pmatrix}}{\lra} A(-1)^{\oplus 3} \overset{\begin{pmatrix} x & y & z \end{pmatrix}}{\lra} A \to 0, 
\]
where as usual the maps are left multiplications by the matrices and the free modules are represented by column vectors.
It is straightforward, as always, to check this is a complex.   It is easy to see that $\{x, y, z \}$ is a minimal generating set for $A$, and that the three given relations are a minimal generating set of the ideal of relations.   Thus exactness at the 
$A$ and $A(-1)^{\oplus 3}$ spots follows from Lemma~\ref{lem:begin}.
Exactness at the $A(-3)$ spot requires the map $A(-3) \to A(-2)^{\oplus 3}$ to be injective.  If this fails, then there is $0 \neq v \in A$ such that $zv = yv = xv = 0$.   However, the form of the $k$-basis above shows that $xv = 0$ implies $v = 0$.  
%(Note that this only requires much weaker information 
%than knowledge that $A$ is a domain.  There is not an obvious normal element to which Lemma~\ref{lem:passup} can be applied to show 
%that $A$ is a domain.)

We lack exactness only possibly at the $A(-2)^{\oplus 3}$ spot.   We can apply a useful general argument in this case:  if an algebra 
is known in advance to have the correct Hilbert series predicted by a potential free resolution of $k_A$, and the potential resolution is known 
to be exact except in one spot, then exactness is automatic at that spot.  
In more detail, if our complex has a homology group at $A(-2)^{\oplus 3}$ with Hilbert series $q(t)$, 
then a similar computation as the one preceding Lemma~\ref{lem:finGK} 
will show that $h_A(t) = (1 - q(t))/(1-t)^3$.  But we know that $h_A(t) = 1/(1-t)^3$ by the Gr{\"o}bner basis computation, 
and this forces $q(t) = 0$.  Thus the complex is a free resolution of $k_A$, as claimed.

Now applying $\uHom_A(-, A_A)$, using Lemma~\ref{lem:homtoA} and shifting by $(-3)$, we get the complex of left 
modules 
\[
0 \la A \overset{\begin{pmatrix} x\\(-a/b)y \\ z \end{pmatrix}}{\lla} A(-1)^{\oplus 3} \overset{\begin{pmatrix} (1/b)x & -(b/a)z & -(a/b)y \\ -(a/b)z & -(1/a)y & x \\ y & x & 0 \end{pmatrix}}{\lla} A(-2)^{\oplus 3} \overset{\begin{pmatrix} x & y & z \end{pmatrix}}{\lla} A(-3) \la 0,
\]
where the free modules are now row vectors and the maps are given by right multiplication by the matrices.
The proof that this is a free resolution of $_A k$ is very similar, using a left-sided version of 
Lemma~\ref{lem:begin}, and noting that the form of the $k$-basis shows that $vz = 0$ implies $v = 0$, to get exactness at $A(-3)$.  

Thus $A$ is a regular algebra of global dimension $3$ and GK-dimension $3$.
\end{example}

Not every regular algebra of dimension $3$ which is generated by degree $1$ elements has Hilbert series 
$1/(1-t)^3$.  Here is another example.
\begin{example}
\label{ex:reg3-2}
Let $A = k \langle x, y \rangle/(yx^2 - x^2y, y^2x - xy^2)$, as in Example~\ref{ex:cubic}.  As shown 
there, $A$ has $k$-basis $\{x^i(yx)^j y^k| i, j, k \geq 0 \}$ and Hilbert series $1/(1-t)^2(1-t^2)$.
We claim that the minimal graded free resolution of $k_A$ has the following form:
\[
0 \to A(-4) \overset{\begin{pmatrix} y \\ -x  \end{pmatrix}}{\lra} A(-3)^{\oplus 2} \overset{\begin{pmatrix} -xy & -y^2 \\ x^2 & yx  \end{pmatrix}}{\lra} A(-1)^{\oplus 2} \overset{\begin{pmatrix} x & y \end{pmatrix}}{\lra} A \to 0.
\]
This is proved in an entirely analogous way as in Example~\ref{ex:reg3}.  The proof of the AS-Gorenstein 
condition also follows in the same way as in that example, and so $A$ is AS-regular of global and GK-dimension 3.  Notice from this example that the integer $\ell$ in Definition~\ref{def:AS} can be bigger than the dimension of the regular algebra in general.
\end{example}

Next, we give an example which satisfies conditions (1) and (2) of Definition~\ref{def:AS}, but not the AS-Gorenstein condition (3).  
\begin{example}
\label{ex:notgor}
Let $A = k \langle x, y \rangle/(yx)$.  The Diamond Lemma gives that this algebra has Hilbert series $1/(1-t)^2$ 
and a $k$-basis $\{x^i y^j | i , j \geq 0 \}$, so $\GK(A) = 2$.
The basis shows that left multiplication by $x$ is injective; together with Lemma~\ref{lem:begin} 
we conclude that 
\[
0 \to A(-2) \overset{\begin{pmatrix} 0 \\ x 
\end{pmatrix}}{\lra} A(-1)^{\oplus 2} \overset{\begin{pmatrix} x & y \end{pmatrix}}{\lra} A \to 0
\]
is a minimal free resolution of $k_A$.  In particular, 
$A$ has global dimension $2$ by Proposition~\ref{prop:gldim}.  

Applying $\uHom(-, A_A)$ to this free resolution and applying the shift $(-2)$ gives the complex of left $A$-modules
\[
0 \la A \overset{\begin{pmatrix} 0 \\ x 
\end{pmatrix}}{\la} A(-1)^{\oplus 2} \overset{\begin{pmatrix} x & y \end{pmatrix}}{\la} A(-2) \la 0,
\]
using Lemma~\ref{lem:homtoA}.   It is apparent that this is not a minimal free resolution of $_A k$: the fact that 
the entries of $\begin{pmatrix} 0 \\ x \end{pmatrix}$ do not span $kx + ky$ prevents this complex from being isomorphic 
to the minimal graded free resolution constructed by the left-sided version of Lemma~\ref{lem:begin}.
More explicitly, the homology at the $A$ spot is $A/Ax$, and $Ax \neq A_{\geq 1}$ (in fact $Ax = k[x]$). 
So $A$ fails the AS-Gorenstein condition and hence is not AS-regular.  It is obvious that $A$ is not a domain, and 
one may also check that $A$ is not noetherian (Exercise 2.3).
\end{example}
 The example above demonstrates the importance of the Gorenstein condition in the noncommutative case.
There are many finitely graded algebras of finite global dimension with bad properties, and for some reason 
adding the AS-Gorenstein condition seems to eliminate these bad examples.   It is not well-understood 
why this should be the case.   In fact, the answers to the following questions are unknown:
\begin{question}
\label{ques:AS}
Is an AS-regular algebra automatically noetherian?    Is an AS-regular algebra automatically a domain?  
Must an AS-regular algebra have other good homological properties, such as Auslander-regularity and the Cohen-Macaualay property?
\end{question}
\noindent
We omit the definitions of the final two properties; see, for example, \cite{Le}.  
%The properties in the question 
%are just a partial list of the good properties that seem to be forced automatically by the definition of AS-regularity.
The questions above have a positive answer for all AS-regular algebras of dimension at most 3 (which are classified), and for all known 
AS-regular algebras of higher dimension.   It is known that if a regular algebra of dimension at most $4$ is noetherian, then it is a domain \cite[Theorem 3.9]{ATV2}, but there are few other results of this kind.

In the opposite direction, one could hope that given a \fig\ algebra of finite global dimension, the AS-Gorenstein 
condition might follow from some other more basic assumption, such as the noetherian property.   There are some results along these lines: for example, Stephenson and Zhang 
proved that a \fig\ noetherian algebra which is quadratic (that is, with relations of degree 2) and of global dimension $3$ is automatically AS-regular \cite[Corollary 0.2]{SZ2}.  On the other hand, the author and Sierra have recently given examples of finitely graded quadratic algebras of global dimension 4 with Hilbert series $1/(1-t)^4$ which are noetherian domains, but which fail to be AS-regular \cite{RSi}.

\subsection{Classification of regular algebras of dimension 2}

We show next that AS-regular algebras $A$ of dimension $2$ are very special.
For simplicity, we only study the case that $A$ is generated as an $k$-algebra by $A_1$, in which case we say that $A$ is \emph{generated in degree 1}.  The classification in the general case is similar (Exercise 2.5).

\begin{theorem}
\label{thm:reg2}
Let $A$ be a \fig\ algebra which is generated in degree $1$, such that $\gldim(A) = 2$. 
\begin{enumerate}
\item If $A$ is weakly AS-regular, then $A \cong F/(g)$, 
where $F = k \langle x_1, \dots, x_n \rangle$ and  $g = \sum_{i =1}^n x_i \tau(x_i)$ for some bijective linear transformation $\tau \in \End_k(F_1)$ and 
some $n \geq 2$.

\item If $A$ is AS-regular, then $n = 2$ above and $A$ is isomorphic to either the 
Jordan or quantum plane of Examples~\ref{ex:qplane}. 
\end{enumerate}
\end{theorem}
\begin{proof}
(1) 
Since $A$ is weakly regular, it finitely presented 
(Remark~\ref{rem:fp}).
Thus $A \cong F/(g_1, \dots, g_r)$, where $F = k \langle x_1, \dots, x_n \rangle$ with $\deg x_i = 1$, and $\deg g_i = s_i$.   We may assume that the $x_i$ are a minimal set of generators and that the $g_i$ are a minimal set of homogeneous relations.   The 
minimal resolution of $k_A$ has length $2$, and so by Lemma~\ref{lem:begin}, it has the form
\[
0 \to \bigoplus_{j = 1}^r A(-s_j) \overset{\begin{pmatrix} g_{ij} \end{pmatrix}}{\lra} \bigoplus_{i=1}^n A(-1) \overset{\begin{pmatrix} x_1 & x_2 & \dots & x_n \end{pmatrix}}{\lra} A \to 0, 
\]
where $g_j = \sum_i x_i g_{ij}$.   Applying $\uHom_A(-, A_A)$ gives 
\begin{equation}
\label{eq:left1}
0 \la \bigoplus_{j = 1}^r A(s_j) \overset{\begin{pmatrix} g_{ij} \end{pmatrix}}{\la} \bigoplus_{i=1}^n A(1)\overset{\begin{pmatrix} x_1 & x_2 & \dots & x_n \end{pmatrix}}{\la} A \la 0,
\end{equation}
which should be the minimal graded free resolution of $_A k(\ell)$ for some $\ell$, by the AS-Gorenstein condition.  On the other hand, by the left-sided version of Lemma~\ref{lem:begin}, 
 the minimal free resolution of $_A k$ has the form 
\begin{equation}
\label{eq:left2}
0 \la A \overset{\begin{pmatrix} x_1 \\ x_2 \\ \dots \\ x_n \end{pmatrix}}{\lla} \bigoplus_{i=1}^n A(-1) \overset{\begin{pmatrix} h_{ji} \end{pmatrix}}{\lla}  \bigoplus_{j = 1}^r A(-s_j)  \la 0,
\end{equation}
where $g_j = \sum_i h_{ji} x_i$.
Since minimal graded free resolutions are unique up to isomorphism of complexes (Lemma~\ref{lem:min}(2)),
comparing the ranks and shifts of the free modules in \eqref{eq:left1} and \eqref{eq:left2} immediately implies that $r = 1$, $s_1 = 2$, and $\ell = 2$.  Thus the minimal free resolution of $k_A$ has the 
form  
\begin{equation}
\label{eq:right1}
0 \to A(-2) \overset{\begin{pmatrix} y_1 \\ \dots \\ y_n \end{pmatrix}}{\lra} \bigoplus_{i=1}^n A(-1) \overset{\begin{pmatrix} x_1 & \dots & x_n \end{pmatrix}}{\lra} A \to 0
\end{equation}
for some $y_i \in A_1$, and there is a single relation $g = \sum_{i=1}^n x_i y_i$ such that $A \cong F/(g)$.
Once again using that \eqref{eq:left1} is a free resolution of $_A k(2)$, the $\{y_i \}$ must span $A_1$; otherwise 
$\sum A y_i \neq A_{\geq 1}$.  Thus the $\{ y_i \}$ are a basis of $A_1$ also and 
$g = \sum x_i \tau(x_i)$ for some linear bijection $\tau$ of $F_1$.  If $n = 1$, then we have $A \cong k[x]/(x^2)$, which is well-known to have infinite global dimension (or notice that the complex \eqref{eq:right1} is not 
exact at $A(-2)$ in this case), a contradiction.
Thus $n \geq 2$.

(2)  The shape of the free resolution in part (1) implies by \eqref{eq:hs} that 
$h_A(t) = 1/(1 - nt + t^2)$.  By Lemma~\ref{lem:finGK}, $A$ has finite GK-dimension if and only if all of the roots of $(1 - nt + t^2)$ are roots of unity.  It is easy to calculate that this happens only if $n =1$ or $2$, and $n = 1$ is already excluded by part (1).  Thus $n = 2$, and 
$A \cong \langle x, y \rangle/(g)$, where $g$ has the form $x \tau(x) + y \tau(y)$ for a bijection $\tau$.
By Exercise 2.3 below, any such algebra is isomorphic to the quantum plane or the Jordan plane.
\begin{comment}
with 2 generators and one 
quadratic relation is isomorphic to the quantum plane, the Jordan plane, $k \langle x, y \rangle/(yx)$, 
or $k \langle x, y \rangle/(x^2)$.  The third algebra is not regular by Example~\ref{ex:notgor}, 
and in Exercise 1.6 the reader should have found that the fourth example is 
neither of finite global dimension nor finite GK-dimension.  Thus only the quantum and Jordan planes remain.  Alternatively, it is straightforward to prove that any linear change of variable turns a 
relation of the form $x \tau (x) + y \tau(x)$ for a bijection $\tau$ into another the relation 
$x \tau'(x) + y \tau'(y)$, for some other bijection $\tau'$.  But the relations $x^2$ and $yx$ do not have this form.  
\end{comment}
Conversely, the regularity of the quantum and Jordan planes was noted in Example~\ref{ex:jq}.
\end{proof}

In fact, part (1) of the theorem above also has a converse: any algebra of that form 
$A = k \langle x_1, \dots, x_n \rangle/(\sum_i x_i \tau(x_i))$ with $n \geq 2$ is weakly AS-regular (Exercise 2.4).  When $n \geq 3$, these are non-noetherian algebras of exponential growth which have other interesting properties \cite{Zh2}.  

\subsection{First steps in the classification of regular algebras of dimension 3}

The classification of AS-regular algebras of dimension $3$ was a major achievement.  The basic framework of the classification was laid out by Artin and Schelter in \cite{AS}, but to complete the classification required the development of the geometric techniques in the work of Artin, Tate, and Van den Bergh \cite{ATV1, ATV2}.   
 We certainly cannot give the full details of the classification result in these notes, but we present some of the easier first steps now, and give a glimpse of the main idea of the rest of the proof at the end of Lecture 3.

First, a similar argument as in the global dimension $2$ case shows that the possible shapes 
of the free resolutions of $k_A$ for regular algebras $A$ of dimension $3$ are limited.  We again focus 
on algebras generated in degree $1$ for simplicity.  For the classification of regular algebras of dimension $3$ with generators 
in arbitrary degrees, see the work of Stephenson \cite{Ste2, Ste3}.
\begin{lemma}
\label{lem:3reghs}
Let $A$ be an AS-regular algebra of global dimension $3$ which is generated in degree 1.  Then exactly one of the following holds:
\begin{enumerate}
\item $A \cong k \langle x_1, x_2, x_3 \rangle/(f_1, f_2, f_3)$, where the $f_i$ have degree $2$, 
and $h_A(t) = 1/(1-t)^3$; or 
\item $A \cong k \langle x_1, x_2 \rangle/(f_1, f_2)$, where the $f_i$ have degree $3$, and 
$h_A(t) = 1/(1-t)^2(1-t^2)$.
\end{enumerate}
\end{lemma}
\begin{proof}
As in the proof of Theorem~\ref{thm:reg2}, $A$ is finitely presented.   Let $A \cong k \langle x_1, \dots, x_n \rangle/(f_1, \dots, f_r)$ be a minimal presentation.   By Lemma~\ref{lem:begin}, the minimal free resolution of $k_A$ has the form 
\[
0 \to P_3 \to \bigoplus_{i = 1}^r A(-s_i) \to \bigoplus_{i=1}^n A(-1) \to A \to 0, 
\]
where $s_i$ is the degree of the relation $f_i$.    A similar argument as in the proof of Theorem~\ref{thm:reg2} using the AS-Gorenstein condition implies that the ranks 
and shifts of the graded free modules appearing must be symmetric:  namely, 
necessarily 
$r = n$, $P_3 = A(-\ell)$ for some $\ell$, and $s_i = \ell - 1$ for all $i$, so that the resolution of $k_A$ now has the form
\[
0 \to A(-s-1) \to \bigoplus_{i = 1}^n A(-s) \to \bigoplus_{i=1}^n A(-1) \to A \to 0.
\]
In particular, from \eqref{eq:hs} we immediately 
read off the Hilbert series $h_A(t) = 1/(- t^{s+1} + n t^s  -nt + 1)$.   By Lemma~\ref{lem:finGK}, since $A$ has finite GK-dimension, $p(t) = - t^{s+1} + n t^s  -nt + 1$ has only roots of unity for its zeros.  Note that $p(1) = 0$, and that 
$p'(1) = -(s+1) + sn - n = sn -n -s -1$.  If $n + s > 5$, then $sn - n - s -1  = (s-1)(n-1) -2 > 0$, so that $p(t)$ is increasing at $t = 1$.  Since $\lim_{t \to \infty} p(t) = - \infty$, $p(t)$ has a real root greater than $1$ in this case, a contradiction.  Thus $n + s \leq 5$.  
The case $n = 1$ gives $A \cong k[x]/(x^s)$, which is easily ruled 
out for having infinite global dimension, similarly as in the proof of Theorem~\ref{thm:reg2}.  If $s = 1$, then the relations have degree $1$, 
and the chosen generating set is not minimal, a contradiction.
If $s = n = 2$ then an easy calculation shows that the power series $h_A(t) = 1/(1 - 2t + 2t^2 - t^3)$ has first few terms 
$1 + 2t + 2t^2 + t^3 + 0t^4 + 0 t^5 + t^6 + \dots$; in particular, $A_4 = 0$ and $A_6 \neq 0$.  Since $A$ is generated in degree $1$, $A_i A_j = A_{i + j}$ for all $i, j \geq 0$, and so $A_4 = 0$ implies $A_{\geq 4} = 0$, a contradiction.  
 This leaves the cases $n = 3, s = 2$ and $n = 2, s = 3$.  
\end{proof}
\noindent  Based on the degree in which the relations occur,
we say that $A$ is a \emph{quadratic} regular algebra in the first case of the lemma above, and that $A$ is \emph{cubic} in the second case.    We have seen regular algebras of both types already (Examples~\ref{ex:reg3} and \ref{ex:reg3-2}).  Note that the cubic case has the same Hilbert series as a commutative polynomial ring in three variables with weights $1, 1, 2$; such a commutative weighted polynomial ring is of course AS-regular, but not generated in degree $1$.   
In general, no AS-regular algebra is known which does not have the Hilbert series of a weighted polynomial ring.

Suppose that $A$ is an AS-regular algebra of global dimension $3$, generated in degree $1$. 
By Lemma~\ref{lem:3reghs}, we may write the minimal graded free resolution of $k_A$ in the form 
\begin{equation}
\label{eq:umv}
0 \to A(-s-1) \overset{u}{\lra} A(-s)^{\oplus n} \overset{M}{\lra} A(-1)^{\oplus n} \overset{ v} {\lra} A \to 0,
\end{equation}
where $v = (x_1, \dots, x_n)$ ($n = 2$ or $n = 3$) and $u = (y_1, \dots, y_n)^t$, and $M$ is an $n \times n$ matrix 
of elements of degree $s-1$, where $s = 5-n$.  The $y_i$ are also a basis of $kx_1 + \dots + k x_n$, by a similar argument as in the proof of Theorem~\ref{thm:reg2}.   Recall that the minimal resolution of $k_A$ is unique only up to isomorphism of complexes.  In particular, it is easy 
to see we can make a linear change to the free basis of $A(-s)^{\oplus n}$, which will change $u$ to $u' = Pu$ and $M$ to $M' = MP^{-1}$, for some $P \in \operatorname{GL}_n(k)$.   In this way we can make $u' = (x_1, \dots, x_n)^t = v^t$.  Changing notation back, our resolution now has the form 
\begin{equation}
\label{eq:vmv}
0 \to A(-s-1) \overset{v^t}{\lra} A(-s)^{\oplus n} \overset{M}{\lra} A(-1)^{\oplus n} \overset{ v} {\lra} A \to 0.
\end{equation}

In terms of the canonical way to construct the beginning of a graded free resolution given by Lemma~\ref{lem:begin}, all we have done above is replaced the relations with some linear combinations.  Since the entries of $M$ consist of elements of degree $s-1$ and the relations of $A$ have degree $s$, the entries of $M$ and $v$ lift uniquely to 
homogeneous elements of the free algebra $F = k \langle x_1, \dots, x_n \rangle$, where $A \cong F/I$ is a minimal presentation.   The $n$ entries of the product $vM$, taking this product in the ring $F$, are a minimal generating set of the ideal $I$ of relations, as we saw in the proof of 
Lemma~\ref{lem:begin}.   Since applying $\uHom_A( -, A_A)$ to the complex \eqref{eq:vmv} gives a shift of a minimal free resolution of $_A k$, this also forces the $n$ entries of the column vector $M v^t$, with the product taken in $F$, to be 
a minimal generating set of the ideal $I$.  Thus there is a matrix $Q \in \operatorname{GL}_n(k)$ such that 
$Q Mv^t = (vM)^t$, as elements of $F$.

\begin{example}
\label{ex:change}
Consider a quantum polynomial ring $A = k \langle x, y, z \rangle/(yx - pxy, xz-qzx, zy-ryz )$.
  Then a minimal graded free resolution of $k_A$  has the form
\[
0 \to A(-3) \overset{\begin{pmatrix} rz \\ py \\ qx \end{pmatrix}}{\lra} A(-2)^{\oplus 3} \overset{\begin{pmatrix} -py & z & 0 \\ x & 0 & -rz \\ 0 & -qx & y \end{pmatrix}}{\lra} A(-1)^{\oplus 3} \overset{\begin{pmatrix} x & y & z \end{pmatrix}}{\lra} A \to 0, 
\]
by a similar argument as in Example~\ref{ex:reg3}.  By a linear change of basis of $A(-2)^{\oplus 3}$, this can be adjusted 
to have the form
\[
0 \to A(-3) \overset{\begin{pmatrix}x \\ y \\ z \end{pmatrix}}{\lra} A(-2)^{\oplus 3} \overset{M = \begin{pmatrix} 0 & pz & -rpy \\ -rqz & 0 & rx \\ qy & -pqx & 0 \end{pmatrix}}{\lra} A(-1)^{\oplus 3} \overset{v = \begin{pmatrix} x & y & z \end{pmatrix}}{\lra} A \to 0,
\]
as in \eqref{eq:vmv}.
Then in the free algebra, one easily calculates that $Q M v^t = (vM)^t$, where $Q = \operatorname{diag}(q/p, p/r, r/q)$.
\end{example}

In Artin and Schelter's original work on the classification of regular algebras \cite{AS}, the first step was to classify possible 
solutions of the equation $Q M v^t = (vM)^t$ developed above, where $M$ has entries of degree $s-1$ and $Q$ is an invertible matrix.  Since one is happy to study the algebras $A$ up to isomorphism, by a change of variables one can change 
the matrices $v^t, M, v$ in the resolution \eqref{eq:vmv} to $P^t v^t$, $P^{-1} M (P^t)^{-1}$, and $vP$ for some 
$P \in \operatorname{GL}_n(k)$, changing $Q$ to $P^{-1} Q P$.  Thus one may assume that $Q$ is in Jordan canonical form.  Artin and Schelter show in this way that there are a finite number of parametrized families of algebras containing all regular algebras up to isomorphism, where each family consists of a set of relations with unknown parameters, such as the Sklyanin family $S(a,b, c)$ given in Example~\ref{ex:skl} (for which $Q$ is the identity matrix).  It is a much more difficult problem to decide for which values of the parameters in a family one actually gets an  AS-regular algebra, and the Sklyanin family is one of the hardest in this respect:  no member of the family was proved to be regular in \cite{AS}.   In fact, these algebras are all regular except for a few special values of $a, b, c$.   At the end of the next lecture, we will outline the method  that was eventually used to prove this in \cite{ATV1}.

\subsection{Exercise Set 2}

\
\bigskip

1.  Show that if $P$ is a graded free right module of finite rank over a \fig\ algebra $A$, then there 
is a canonical isomorphism $\uHom_A(\uHom_A(P_A, A_A), {}_A A) \cong P_A$.  Using this, prove that 
if $A$ satisfies (1) and (3) of Definition~\ref{def:AS}, then $A$ automatically also satisfies the left-sided version of (3), 
namely that $\uExt^i_A({}_A k, {}_A A) \cong  \begin{cases} 0 & i \neq d \\ {} k_A(\ell) & i = d \end{cases}$.

\bigskip

2.  Let $A$ be the universal enveloping algebra of the Heisenberg Lie algebra $L = kx + ky + kz$, which has bracket defined by $[x, z] = [y,z] = 0, [x,y ] = z$.  Explicitly, $A$ has the presentation $k \langle x, y, z \rangle/(zx - xz, zy-yz, z - xy - yx)$, and thus $A$ is graded if we assign weights $\deg x =1$, $\deg y = 1$, $\deg z = 2$.   Note that $z$ can be eliminated 
from the set of generators, yielding the minimal presentation $A \cong k \langle x, y \rangle/(yx^2 - 2xyx + x^2y, y^2x - 2yxy + xy^2)$.  Now prove that $A$ is a cubic AS-regular algebra of dimension $3$.

\bigskip

3.  Consider rings of the form $A = k \langle x, y \rangle/(f)$ where $0 \neq f$ is homogeneous of degree $2$.  
Then $f$ can be written uniquely in the form $f = x \tau(x) + y \tau(y)$ for some nonzero linear transformation $\tau: kx + ky \to kx + ky$.  Write $A = A(\tau)$.

(a).  The choice of $\tau$ can be identified with a matrix $B = (b_{ij})$ by setting $\tau(x) = b_{11} x + b_{12} y$ 
and $\tau(y) = b_{21} x + b_{22} y$.   Show that there is a graded isomorphism $A(\tau) \cong A(\tau')$ if and only if the corresponding 
matrices $B$ and $B'$ are \emph{congruent}, that is $B' = C^t B C$ for some invertible matrix $C$.

(b).   Show that $A$ is isomorphic to $k \langle x, y \rangle/(f)$ for one of the four following $f$'s:
(i) $f = yx - qxy$ for some $0 \neq q \in k$ (the quantum plane); (ii) $f = yx -xy -x^2$ (the Jordan plane); 
(iii) $f = yx$; or (iv) $f = x^2$.    (Hint:  show that the corresponding matrices are a 
complete set of representatives for congruence classes of nonzero $2 \times 2$ matrices.)

(c).  Examples (i), (ii), (iii) above all have Hilbert series $1/(1-t)^2$, GK-dimension 2, and global dimension 2, as 
we have already seen (Examples~\ref{ex:firsths}, \ref{ex:jq}, \ref{ex:notgor}).  In case (iv), find the Hilbert series, GK-dimension, and global dimension of the algebra.  Show in addition that both algebras (iii) and (iv) are not noetherian.

\bigskip

4.  Prove a converse to Theorem~\ref{thm:reg2}(1).  Namely, show that if $n \geq 2$ and $\tau: F_1 \to F_1$ is any linear bijection of the space of degree 1 elements in $F = k \langle x_1, \dots, x_n \rangle$,  then $A = F/(f)$ is weakly AS-regular, where $f = \sum_i x_i \tau(x_i)$.   (See \cite{Zh2} for a more extensive study of this class of rings.)

(Hint:  if the term $x_n^2$ does not appear in $f$, then under the degree lex order with $x_1 < \dots < x_n$, $f$ has leading term of the form $x_n x_i$ for some $i < n$.  Find a $k$-basis of words 
for $A$ and hence $h_A(t)$.  Find a potential resolution of $k_A$ and use the Hilbert series to show that it is exact.  If instead $x_n^2$ appears in $f$, do a linear change of variables to reduce to the first case.)

\bigskip

5.  Classify AS-regular algebras of global dimension $2$ which are not necessarily generated in degree $1$.
More specifically,  show that any such $A$ is isomorphic to $k \langle x, y \rangle/(f)$ for some generators with weights $\deg x = d_1$, $\deg y = d_2$, 
where either (i) $f = xy - qyx$ for some $q \in k^{\times}$, or else (ii) $d_1 i = d_2$ for some $i$ and $f = yx - xy - x^{i+1}$.  Conversely, show that each of the algebras in (i), (ii) is AS-regular.
(Hint:  mimic the proof of Theorem~\ref{thm:reg2}).

\bigskip

6.   Suppose that $A$ is AS-regular of global dimension $3$.  Consider a free resolution of $k_A$ of the form 
\eqref{eq:vmv}, and the corresponding equation $Q Mv^t = (vM)^t$ (as elements of the free algebra $F = k \langle x_1, \dots, x_n \rangle $).
Let $s$ be the degree of the relations of $A$, and let $d = s+1$.

Consider the element $\pi = v M v^t$, again taking the product in the free algebra.  Write 
$\pi = \sum_{(i_1, \dots i_d) \in \{1, 2, \dots, n\}^d}  \alpha_{i_1, \dots, i_d} x_{i_1} \dots x_{i_d}$, 
where $\alpha_{i_1, \dots, i_d} \in k$.   Let $\tau$ be the automorphism of the free algebra $F$ determined by the matrix equation $(x_1, \dots, x_n) Q^{-1} = (\tau(x_1), \dots, \tau(x_n))$.  

(a).  Prove that $\pi = \sum_{(i_1, \dots i_d)}  \alpha_{i_1, \dots, i_d} \tau(x_{i_d}) x_{i_1} \dots x_{i_{d-1}}$.  
(In recent terminology, this says that $\pi$ is a \emph{twisted superpotential}.  For instance, see \cite{BSW}.)
Conclude that $\tau(\pi) = \pi$.

(b).  Show that $\tau$ preserves the ideal of relations $I$ (where $A = F/I$) and thus induces an automorphism of $A$.  (This is an important automorphism of $A$ called the \emph{Nakayama automorphism}.)

\bigskip

\section{Lecture 3: Point modules}

In this lecture, we discuss one of the ways that geometry can be found naturally, but perhaps unexpectedly, in the underlying structure of noncommutative graded rings.   Point modules and the spaces parameterizing them were first studied by Artin, Tate, and Van den Bergh \cite{ATV1} in order to complete the classification of AS-regular algebras of dimension 3.  They have turned out to be important tools much more generally, with many interesting applications.

\begin{definition}
Let $A$ be a \fig\ $k$-algebra that is generated in degree 1.   A \emph{point module} for $A$ is a graded right module $M$ such that $M$ is cyclic, generated in degree $0$, and has Hilbert series $h_M(t) = 1/(1-t)$, in other words $\dim_k M_n = 1$ for all $n \geq 0$.
\end{definition}

In this section, we are interested only in graded modules over a \fig\ algebra $A$, and so all homomorphisms of modules  
will be graded (degree preserving) unless noted otherwise.  In particular, when we speak of isomorphism classes of point modules, we will mean equivalence classes under the relation of graded isomorphism.

The motivation behind the definition of point module comes from commutative projective geometry.  
Recall that the projective space $\mb{P}^n$ over $k$ consists of equivalence classes of $n+1$-tuples $(a_0, a_1, \dots, a_n) \in k^{n+1}$ such that not all $a_i$ are $0$, where two $n+1$-tuples are equivalent if they are nonzero scalar multiples of each other.   The equivalence class of the point $(a_0, a_1, \dots, a_n)$ is written as $(a_0: a_1: \dots : a_n)$.  
Each point $p = (a_0: a_1: \dots :a_n) \in \mb{P}^n$ corresponds to a homogeneous ideal  $I = I(p)$ of $B = k[x_0, \dots, x_n]$, where $I_d = \{ f \in B_d | f(a_0, a_1, \dots, a_n) = 0 \}$.  It is easy to check that $B/I(p)$ is a point module of $B$, since vanishing at a point is a linear condition on the elements of $B_d$.  Conversely, if $M$ is a point module of $B$, then since $M$ is cyclic and generated in degree $0$, we have 
$M \cong B/J$ for some homogeneous ideal $J$ of $B$.  Necessarily $J = \ann(M)$, so $J$ is uniquely determined 
by $M$ and is the same for any two isomorphic point modules $M$ and $M'$.  Since $J_1$ is an $n$-dimensional subspace of 
$kx_0 + \dots + kx_n$, it has a $1$-dimensional orthogonal complement, in other words there is a unique up to scalar 
nonzero vector $p = (a_0: \dots : a_n)$ such that $f(a_0, \dots, a_n) = 0$ for all $f \in J_1$.  It is straightforward 
to prove that $I = I(p)$ is generated as an ideal by $I_1$ (after a change of variables, one can prove this 
for the special case $I = (x_0, \dots, x_{n-1})$, which is easy.)  Since $J_1 = I_1$ we conclude that 
$J \supseteq I$ and hence $J = I$ since they have the same Hilbert series.    

To summarize the argument of the previous paragraph, the isomorphism classes of point modules for the polynomial ring $B$ are in bijective correspondence with the points of the associated projective space.  This correspondence generalizes immediately to any \fig\ commutative algebra which is generated in degree $1$.  Namely, given any homogeneous ideal $J$ of $B = k[x_0, \dots, x_n]$, one can consider $A = k[x_0, \dots, x_n]/J$ and 
the corresponding closed subset of projective space 
\[ \cproj A = X = \{ (a_0: a_1 : \dots :a_n) \in \mb{P}^n | f(a_0, \dots, a_n) = 0\ \text{for all homogeneous}\ f \in J \}.
\]
Then the point modules of $A$ are precisely the point modules of $B$ killed by $J$, that is 
those point modules $B/I$ such that $J \subseteq I$.  So there is a bijective correspondence between isomorphism classes of point modules of $A$ and points in $X = \cproj A$.  We may also say that $X$ \emph{parametrizes} the point modules for $A$, in a sense we leave informal for now but make precise later in this lecture.

Many noncommutative graded rings also have nice parameter spaces of point modules.
\begin{example}
Consider the quantum plane $A = k \langle x, y \rangle/(yx - qxy)$ for some $q \neq 0$.  
We claim that its point modules are parametrized by $\mb{P}^1$, just as is true for a commutative polynomial ring in two variables.  To see this, note that if $M$ is a point module for $A$, then 
$M \cong A/I$ for some homogeneous \emph{right} ideal $I$ of $A$, with $\dim_k I_n = \dim_k A_n -1$ for all $n \geq 0$.  
Also, $I$ is uniquely determined by knowledge of the (graded) isomorphism class of the module $M$, since $I = \ann M_0$.
Thus it is enough to parametrize such homogeneous right ideals $I$.
But  choosing any $0 \neq f \in I_1$, the right ideal $fA$ has Hilbert series $t/(1-t)^2$ (since $A$ is a domain) and so $A/fA$ has Hilbert series $1/(1-t)^2 - t/(1-t)^2 = 1/(1-t)$, that is, it 
is already a point module.   Thus $fA = I$.  Then the point modules up to isomorphism are in bijective correspondence with the $1$-dimensional subspaces of the $2$-dimensional space $A_1$, that is, with a copy of $\mb{P}^1$.
The same argument shows the same result for the Jordan plane.
\end{example}

It is natural to consider next the point modules for a quantum polynomial ring in more than 2 variables, as in Examples~\ref{ex:qpoly}.
The answer is more complicated than one might first guess---we discuss the three variable case in Example~\ref{ex:qpolypm} below.  First, we will develop a general method 
which in theory can be used to calculate the parameter space of point modules for any finitely presented algebra generated in degree 1.  In the commutative case, we saw that the point modules for $k[x_0, \dots, x_n]/J$ were easily determined as a subset of the point modules for $k[x_0, \dots, x_n]$.  In the noncommutative case it is natural to begin similarly by examining the point modules for a free associative algebra.  

\begin{example}
\label{ex:free}
Let $A = k \langle x_0, \dots x_n \rangle$ be a  free associative algebra with $\deg x_i = 1$ for all $i$.
Fix a graded $k$-vector space of Hilbert series $1/(1-t)$, say $M = km_0 \oplus km_1 \oplus \dots$, where $m_i$ is a basis vector for the degree $i$ piece.  We think about the possible 
graded $A$-module structures on this vector space.  If $M$ is an $A$-module, then 
\begin{equation}
\label{point-eq}
m _i x_j = \lambda_{i,j} m_{i+1}
\end{equation} 
for some $\lambda_{i,j} \in k$.  It is clear 
that these constants $\lambda_{i,j}$ determine the entire module structure, since the $x_j$ generate the algebra.  Conversely, since $A$ is free on the generators $x_i$, it is easy to see that any choice of arbitrary constants $\lambda_{i,j} \in k$ does determine an $A$-module structure on $M$ via the formulas \eqref{point-eq}.
Since a point module is by definition cyclic, if we want constants $\lambda_{i,j}$ to define a point module, 
we need to make sure that for each $i$, some $x_j$ actually takes $m_i$ to a nonzero 
multiple of $m_{i+1}$.  In other words, $M$ is cyclic if and only if for each $i$, $\lambda_{i,j} \neq 0$ for some $j$.
Also, we are interested in classifying point modules up to isomorphism.  It is easy to check that 
point modules determined by sequences $\{ \lambda_{i,j} \}$, $\{ \lambda'_{i,j} \}$ as above are isomorphic precisely when for each $i$, the nonzero vectors $(\lambda_{i,0}, \dots, \lambda_{i,n})$ and $(\lambda'_{i,0}, \dots, \lambda'_{i,n})$ are scalar multiples (scale the basis vectors $m_i$ to compensate).   Thus we can account for this by 
considering each $(\lambda_{i,0}: \dots : \lambda_{i,n})$ as a point in projective $n$-space.

In conclusion, the isomorphism classes of point modules over the free algebra $A$ 
are in bijective correspondence with $\mb{N}$-indexed sequences of points in $\mb{P}^n$,  
$ \{ (\lambda_{i,0}: \dots : \lambda_{i,n} ) \in \mb{P}^n | i  \geq 0 \}$, or in other words 
points of the infinite product 
$\mb{P}^n \times \mb{P}^n \times \dots  = \prod_{i = 0}^{\infty} \mb{P}^n$.
\end{example}

%The example above shows that in general, point modules for noncommutative graded algebras are not parametrized by 
%nice projective geometric objects (an infinite product of copies of projective space is not nice, from an algebraic %geometric standpoint.)

Next, we show that we can parametrize the point modules for a finitely presented algebra by  a 
subset of the infinite product in the example above.   Suppose that $f \in k \langle x_0, \dots x_n \rangle$ is a homogeneous element of degree $m$, say $f = \sum_w a_w w$, where the sum is over words of 
degree $m$.  Consider a set of $m(n+1)$ commuting indeterminates $\{ y_{ij} | 1 \leq i \leq m, 0 \leq j \leq n \}$
and the polynomial ring $B = k[y_{ij}]$.  The \emph{multilinearization} of $f$ is the element of $B$ 
given by replacing each word $w = x_{i_1} x_{i_2} \dots x_{i_m}$ occurring in $f$ by $y_{1, i_1} y_{2, i_2} \dots y_{m, i_m}$.   Given any such multilinearization $g$ and a sequence of points $\{ p_i = (a_{i,0}: \dots : a_{i,n}) | 1 \leq i \leq m \} \subseteq \prod_{i = 1}^{m} \mb{P}^n$, 
the condition $g(p_1, \dots, p_m) = 0$, where $g(p_1, \dots, p_m)$ means the evaluation of $g$ by substituting $a_{i,j}$ for $y_{i,j}$, is 
easily seen to be well-defined.
%Beginning now we use talk about closed subschemes of a projective space.  The reader who has not studied schemes can simply replace the word %subscheme  by subset and will still be able to follow most of the rest of this lecture, excluding the final subsection.

\begin{proposition}
\label{prop:param}
Let $A \cong k \langle x_0, \dots, x_n \rangle/(f_1, \dots, f_r)$ be a finitely presented connected graded $k$-algebra, where $\deg x_i = 1$ and the $f_j$ are homogeneous of degree $d_j \geq 2$.  For each $f_i$, let $g_i$ be the multilinearization of $f_i$.  
\begin{enumerate}
\item 
The isomorphism classes of point modules for $A$ are in bijection with the closed subset $X$ of 
$\prod_{i=0}^{\infty} \mb{P}^n$ given by 
\[
X = \{ (p_0, p_1, \dots) \, | \, g_j(p_i, p_{i+1}, \dots, p_{i+d_j-1}) = 0\ \text{for all}\ 1 \leq j \leq r, i \geq 0 \}.
\]
\item 
Consider for each $m \geq 1$ the closed subset
\[
X_m = \{ (p_0, p_1, \dots, p_{m-1} ) | g_j(p_i, p_{i+1}, \dots, p_{i+d_j-1}) = 0\ \text{for all}\ 1 \leq j \leq r, 0 \leq i \leq m -d_j  \}
\]
of $\prod_{i=0}^{m-1} \mb{P}^n$.  The natural projection onto the first $m$ coordinates defines a map $\phi_m: X_{m+1} \to X_m$ for each $m$.  Then $X$ is equal to the inverse limit $\varprojlim X_m$ of the $X_m$ with 
respect to the maps $\phi_m$.  In particular, if $\phi_m$ is a bijection for all $m \geq m_0$, then the isomorphism classes of point modules of $A$ 
are in bijective correspondence with the points of $X_{m_0}$.
 \end{enumerate}
\end{proposition}
\begin{proof}
(1) Let $F =k \langle x_0, \dots, x_n \rangle$ and $J= (f_1 \dots, f_r)$, so that $A \cong F/J$.  Clearly the isomorphism classes of point modules for $A$ correspond to those point modules of $F$ which are annihilated by $J$.  
Write a point module for $F$ as $M = km_0 \oplus km_1 \oplus \dots$, as in Example~\ref{ex:free}, where $m_i x_j = \lambda_{i,j} m_{i+1}$.  Thus $M$ corresponds to the infinite sequence of points $(p_0, p_1, \dots )$, where $p_j = (\lambda_{j,0} : \dots : \lambda_{j, n})$.
The module $M$ is a point module for $A$ if and only if $m_i f_j = 0$ for all $i, j$.   If $w = x_{i_1} x_{i_2} \dots x_{i_d}$ is a word, 
then $m_i w = \lambda_{i, i_1}  \lambda_{i+1, i_2}\dots \lambda_{i+d-1, i_d} m_{i+d}$, and so 
$m_i f_j = 0$ if and only if the multilinearization $g_j$ of $f_j$ satisfies $g_j(p_i, \dots, p_{i+d_j-1}) = 0$.

(2) This is a straightforward consequence of part (1). 
 \end{proof}

In many nice examples the inverse limit in part (2) above does stabilize (that is, $\phi_m$ is a bijection for all $m \geq m_0$),
and so some closed subset $X_{m_0}$ of a finite product of projective spaces parametrizes the isomorphism classes of point modules.  

\begin{example}
\label{ex:sklpm}
Let  $k = \mb{C}$ for simplicity in this example.
Let \[
S = k \langle x, y, z \rangle/(azy + byz + cx^2, axz + bzx + cy^2, ayx + bxy + cz^2)
\] 
be the Sklyanin algebra for some $a, b, c \in k$, as in Example~\ref{ex:skl}.   Consider the closed subset $X_2 \subseteq \mb{P}^2 \times \mb{P}^2 = \{ (x_0:y_0:z_0), (x_1:y_1:z_1) \}$ given by the vanishing of the multilinearized relations 
\[
 a z_0 y_1 + b y_0 z_1 + c x_0 x_1,  a x_0 z_1 + bz_0 x_1 + cy_0y_1,  a y_0 x_1 + b x_0 y_1 + c z_0 z_1.
\]
Let $E$ be the projection of $X_2$ onto the first copy of $\mb{P}^2$.  
To calculate $E$, note that the $3$ equations can be written in the matrix form
\begin{equation}
\label{eq:mat}
\begin{pmatrix} c x_0 & a z_0 & b y_0  \\ bz_0 & c y_0 & ax_0   \\ a y_0 & bx_0 & c z_0  \end{pmatrix} \begin{pmatrix} x_1 \\ y_1 \\ z_1 \end{pmatrix} = 0.
\end{equation}
Now given $(x_0: y_0: z_0) \in \mb{P}^2$, there is at least one solution $(x_1: y_1: z_1) \in \mb{P}^2$ to this matrix equation if and only if the matrix on the left has rank at most $2$, in other words is not invertible.  Moreover, if the matrix has rank exactly $2$ then there is exactly one solution $(x_1: y_1: z_1) \in \mb{P}^2$.   Taking the determinant of the matrix, we see that the locus of $(x_0: y_0: z_0) \in \mb{P}^2$ 
such that the matrix is singular is the solution set $E$ of the equation
\[
(a^3 + b^3 + c^3)x_0y_0z_0 - abc(x_0^3 + y_0^3 + z_0^3) = 0.
\]
One may check that $E$ is a nonsingular curve (as in \cite[Section I.5]{Ha}), as long as 
\[
abc \neq 0\ \text{and}\ ((a^3 + b^3 + c^3)/3abc)^3 \neq 1,   
\]
and then $E$ is an elliptic curve since it is the vanishing of a degree $3$ polynomial \cite[Example V.1.5.1]{Ha}.
We will assume that $a, b, c$ satisfy these constraints.

A similar calculation can be used to find the second projection of $X_2$.  The three multilinearized relations also can be written in the matrix 
form 
\[
\begin{pmatrix} x_0 & y_0 & z_0 \end{pmatrix} \begin{pmatrix} c x_1 & a z_1 & b y_1  \\ b z_1 & c y_1 & ax_1   \\ a y_1 & b x_1 & c z_1 \end{pmatrix}  = 0.
\]
Because the $3 \times 3$ matrix here is simply the same as the one in \eqref{eq:mat} with the subscript $0$ replaced by $1$, an analogous argument shows that the second projection is the same curve $E$.  

One may show directly, given our constraints on $a, b, c$, that for each point $(x_0:x_1:x_2) \in E$ the corresponding matrix in \eqref{eq:mat} has rank exactly $2$.  In particular, for each $p \in E$ 
there is a unique $q \in E$ such that $(p, q) \in X_2$.  Thus $X_2 = \{ (p, \sigma(p)) | p \in E \}$ for some 
bijective function $\sigma$.  It is easy to see that $\sigma$ is a regular map by finding an explicit formula:  the cross product of the first two rows of the matrix in \eqref{eq:mat} will be orthogonal to both rows 
and hence to all rows of the matrix when it has rank $2$. This produces the  formula 
\[ \sigma(x_0:y_0:z_0) = (a^2 z_0x_0 - bc y_0^2:  b^2 y_0z_0 - ac x_0^2:  c^2x_0 y_0 - ab z_0^2),
\]
which holds on the open subset of $E$ for which the first two rows of the matrix are linearly independent.  One gets similar formulas by taking the other possible pairs of rows, and since at each point of $E$ some pair of rows is linearly independent,  the map is regular.

It now easily follows from Proposition~\ref{prop:param} that $X= \{ (p, \sigma(p), \sigma^2(p), \dots ) | p \in E \}$ is the subset of the infinite product $\prod_{i = 0}^{\infty} \mb{P}^2$ parametrizing the point modules.  Thus the maps $\phi_m$ of Proposition~\ref{prop:param} are isomorphisms for $m \geq 2$ and $X_2 \cong E$ is already in bijective correspondence with the isomorphism classes of 
point modules for the Sklyanin algebra $S$ (with parameters $a, b, c$ satisfying the constraints above).
\end{example}

\begin{example}
\label{ex:qpolypm}
Let 
\[
A = k \langle x, y, z \rangle/(zy - r yz, xz - qzx, yx - p xy)
\]
be a quantum polynomial ring, where $p, q, r \neq 0$. 

The calculation of the point modules for this example is similar as in Example~\ref{ex:sklpm}, but a bit easier, 
and so we leave the details to the reader as Exercise 3.1, and only state the answer here.
Consider $X_2 \subseteq \mb{P}^2 \times \mb{P}^2$, the closed set cut out by the multilinearized relations 
$z_i y_{i+1} - r y_i z_{i+1} = 0$,  $x_i z_{i+1} - q z_i x_{i+1} = 0$, and 
$y_i x_{i+1} - p x_i y_{i+1} = 0$.  It turns out that if $pqr = 1$, then  $X_2 = \{(s, \sigma(s))| s \in \mb{P}^2 \}$, where $\sigma: \mb{P}^2 \to \mb{P}^2$ is an automorphism.  If instead $pqr \neq 1$, letting $E = \{ (a:b:c) | abc = 0 \} \subseteq \mb{P}^2$, then $X_2 = \{ (s, \sigma(s))| s \in E \}$ for some automorphism $\sigma$ of $E$.   Note that $E$ is a union of three lines in this case.

In either case, we see that $X_2$ is the graph of an automorphism of a subset $E$ of $\mb{P}^2$, 
and as in Example~\ref{ex:sklpm} it follows that the maps $\phi_m$ of Proposition~\ref{prop:param} are isomorphisms for $m \geq 2$ and $E$ is in bijection with the isomorphism classes of point modules.
\end{example}

The examples above demonstrate a major difference between the commutative and noncommutative cases.  If $A$ is a commutative \fig\ $k$-algebra, which is generated in degree $1$ and a domain of GK-dimension $d+1$, then the projective variety $\cproj A$ parametrizing the point modules has dimension $d$.  When $A$ is noncommutative, then even if there is a nice space parameterizing the point modules, it may have dimension smaller than $d$.  For instance, the quantum polynomial ring $A$ in Example~\ref{ex:qpolypm} has GK-dimension 3, but a 1-dimensional parameter space of point modules when $pqr \neq 1$; similarly, the Sklyanin algebra for generic $a, b, c$ in Example~\ref{ex:sklpm} has point modules corresponding to an elliptic curve.  In fact, the examples above are fairly representative of the possibilities for AS-regular algebras $A$ of dimension $3$, as we will see in the final section of this lecture.

\subsection{The formal parametrization of the point modules}

The details of this section require some scheme theory.  The reader less experienced with schemes can skim this section and then read the final section of this lecture, which describes how the theory of point modules is used to help classify AS-regular algebras of dimension 3.

Up until now, we have just studied the isomorphism classes of point modules of an algebra as a set, and shown that in many cases these correspond bijectively to the points of some closed subset of a projective variety.   Of course, it should mean something stronger to say that a variety or scheme parametrizes the point modules---there should be a natural geometry intrinsically attached to the set of point modules, which is given by that scheme.  We now sketch how this may be made formal.

It is not hard to see why the set of point modules for a \fig\ algebra $A$ has a natural topology, following the basic idea of the Zariski topology.  As we have noted several times, every point module is isomorphic to 
$A/I$ for a uniquely determined graded right ideal $I$.  Then given a graded right ideal $J$ of $A$, the set of those point modules $A/I$ 
such that $I$ contains $J$ can be declared to be a closed subset in the set $X$ of isomorphism classes of point modules, and the family of such closed subsets made the basis of a topology.  But we really want a scheme structure, not just a topological space, so a more formal construction is required.  The idea is to use the ``functor of points", which is fundamental to the study of moduli spaces in algebraic geometry.  An introduction to this concept can be found in \cite[Section I.4 and Chapter VI]{EH}.  

To use this idea, one needs to formulate the objects one is trying to parametrize over arbitrary commutative base rings.  
Let $A$ be a \fig\ $k$-algebra.   Given a commutative $k$-algebra $R$, an \emph{$R$-point module for $A$} is a graded $R \otimes_k A$-module $M$ (where $R \otimes_k A$ is graded with $R$ in degree $0$) which is cyclic, generated in degree $0$, has $M_0 = R$, and such that $M_n$ is a locally free $R$-module of rank $1$ for all $n \geq 0$.    
Clearly a $k$-point module is just a point module in the sense we have already defined.  Now for each commutative $k$-algebra $R$ we let $P(R)$ be 
the set of isomorphism classes of $R$-point modules for $A$.   Then $P$ is a functor from the category 
of commutative $k$-algebras to the category of sets, where given a homomorphism of $k$-algebras $\phi: R \to S$, the function $P(R) \to P(S)$ is defined by tensoring up, that is $M \mapsto S \otimes_R M$.  We call $P$ the \emph{point functor} for $A$.

Given any $k$-scheme $X$, there is also a corresponding functor $h_X$ from $k$-algebras to sets, defined on objects by 
$R \mapsto \Hom_{k -\operatorname{schemes}}(\spec R, X)$.  The functor $h_X$ acts on a homomorphism of $k$-algebras $\phi: R \to S$ (which 
corresponds to a morphism of schemes $\wt{\phi}: \spec S \to \spec R$) by 
sending it to  
$h_X(\phi): \Hom_{k -\operatorname{schemes}}(\spec R, X) \to  \Hom_{k -\operatorname{schemes}}(\spec S, X)$,  
where $h_X(\phi)(f) = f \circ \wt{\phi}$.
We say that a functor from $k$-algebras to sets is \emph{representable} 
if it is naturally isomorphic to the functor $h_X$ for some scheme $X$.  It is a basic fact that the functor $h_X$ uniquely determines the scheme $X$; this is a version of Yoneda's lemma \cite[Proposition VI-2]{EH}.   If the point functor $P$ above associated to the \fig\ algebra $A$ is naturally isomorphic to $h_X$ for some scheme $X$, then we say that the scheme $X$ \emph{parametrizes} the point modules for the algebra $A$, or that $X$ is a \emph{fine moduli space} for the point modules of $A$.   Note that morphisms of schemes from $\spec k$ to $X$ are in bijective correspondence with the closed points of $X$, and so in particular the closed points of $X$ then correspond bijectively to the isomorphism classes of $k$-point modules.

The previous paragraph shows how to formalize the notion of the point modules for $A$ being parametrized by a scheme, 
but in practice one still needs to understand whether there exists a $k$-scheme $X$ which represents the point functor $P$.   In fact, Proposition~\ref{prop:param}, which showed how to find the point modules as a set in terms of the relations for the algebra, also gives the idea for how to find the representing scheme.  For each $m \geq 0$ we define $X_m$ to be the \emph{subscheme} of 
$\prod_{i = 0}^{m-1} \mb{P}^n$ defined by the vanishing of the multilinearized relations, as in Proposition~\ref{prop:param}(2)  (previously, we 
only considered $X_m$ as a subset).  A \emph{truncated $R$-point module of length $m+1$} for $A$ is an $R \otimes_k A$ module $M = \bigoplus_{i = 0}^m M_i$ with $M_0 = R$, which is generated in degree $0$, and such that $M_i$ is locally free of rank $1$ over $R$ for 
$0 \leq i \leq m$.  Then one can define the truncated point functor $P_m$ which sends a commutative $k$-algebra $R$ 
to the set $P_m(R)$ of isomorphism classes of truncated $R$-point modules of length $m+1$.  A rather formal 
argument shows that the elements of $P_m(R)$ are in natural bijective correspondence with elements of 
$\Hom_{k-\operatorname{schemes}}(\spec R, X_m)$, and thus $X_m$ represents the functor $P_m$ (see \cite[Proposition 3.9]{ATV1}). 
Just as in Proposition~\ref{prop:param}, one has morphisms of schemes $\phi_m: X_{m+1} \to X_m$ for each $m$, induced by projecting onto the first $m$ coordinates.  In nice cases, $\phi_m$ is an isomorphism for all $m \geq m_0$, and such cases  
the projective scheme $X_{m_0}$ represents the point functor $P$.

When it is not true that $\phi_m$ is an isomorphism for all large $m$, one must work with an inverse limit of schemes $\varprojlim X_m$ as 
the object representing the point functor $P$.   Such objects are rather unwieldy, and so it is useful to understand when the 
inverse limit does stabilize (that is, when $\phi_m$ is an isomorphism for $m \gg 0$).
One cannot expect it to stabilize in complete generality, as this already fails for the free algebra as in Example~\ref{ex:free}.  
Artin and Zhang gave an important sufficient condition for stabilization of the inverse limit in \cite{AZ2}, which we describe 
now.
\begin{definition}
A noetherian $k$-algebra $A$ is \emph{strongly noetherian} if for all commutative noetherian $k$-algebras $C$, the base extension $A \otimes_k C$ is also noetherian.
\end{definition}
\begin{theorem} \cite[Corollary E4.12]{AZ2}
\label{thm:AZ2} 
Let $A$ be a \fig\ algebra which is strongly noetherian and generated in degree $1$.  Then there is $m_0$ such that the maps of schemes $\phi_m: X_{m+1} \to X_m$ described above are isomorphisms for all $m \geq m_0$,  and the point modules for $A$ are parametrized by the projective scheme $X_{m_0}$.
\end{theorem}
\noindent In fact, Artin and Zhang studied in \cite{AZ2} the more general setting of Hilbert schemes, where one wishes to parametrize 
factors $M$ of some fixed finitely generated graded $A$-module $Q$ with a given fixed Hilbert function $f: n \mapsto \dim_k M_n$.  The result stated above is just a special case of \cite[Theorem E4.3]{AZ2}, which shows that under the same hypotheses, there is a projective scheme 
parametrizing such factors for any $Q$ and $f$.   The study of these more general moduli spaces is also useful: for example, the \emph{line modules}---cyclic modules generated in degree $0$ with 
Hilbert series $1/(1-t)^2$---have an interesting geometry for AS-regular algebras of dimension $4$ \cite{ShVa}.

The strong noetherian property is studied in detail in  \cite{ASZ}.  Many nice algebras are strongly noetherian.  For example, no non-strongly 
noetherian AS-regular algebras  are known (although as we noted in Question~\ref{ques:AS} above, it has not been proved that an AS-regular algebra must even be noetherian in general).  On the other hand, there are families of \fig\ algebras which are noetherian, but not strongly noetherian, and for which the inverse limit of truncated point schemes does not stabilize.  The simplest such examples are known as \emph{na{\"i}ve blowups}, which will be described in Lecture 5.

\subsection{Applications of point modules to regular algebras}

We now give an overview of how point modules were used in the classification of AS-regular algebras of dimension $3$ by Artin, Tate 
and Van den Bergh.  The details can be found in \cite{ATV1}.  

Let $A$ be a AS-regular algebra of global dimension $3$ which is generated in degree $1$.  By Lemma~\ref{lem:3reghs}, 
we know that $A$ is either  quadratic or cubic.  The method we are about to describe works quite uniformly in the two cases, but for simplicity 
it is easiest to consider only the quadratic case in the following discussion.   Thus we assume that $A$ has three generators and $3$ quadratic relations, and we let $X_2 \subseteq \mb{P}^2 \times \mb{P}^2$ be the subscheme defined by the multilinearizations of the three relations, as in Proposition~\ref{prop:param}. 

We have seen several examples already where $X_2$ is the graph of an automorphism of a closed subscheme of $\mb{P}^2$, and the first step is to show that this is always the case.   First, using the matrix method of Example~\ref{ex:sklpm}, it is straightforward to show that the first and second projections of $X_2$ are equal to a common closed subscheme $E \subseteq \mb{P }^2$, and that either $E = \mb{P}^2$ or else 
$E$ is a degree 3 divisor in $\mb{P}^2$, in other words the vanishing of some cubic polynomial (Exercise 3.4).  
It is important to work with subschemes here; for example, for the enveloping algebra of the Heisenberg Lie algebra of Exercise 2.2, $E$ turns out to be a triple line.  Showing that $X_2$ is the graph of an automorphism of $E$ is more subtle, and is done using some case-by-case analysis of the form of the relations \cite[Section 5]{ATV1}.

Once one has $X_2 = \{ (p, \sigma(p)) | p \in E \}$ for some automorphism $\sigma: E \to E$, one knows that 
$E$ parametrizes the point modules for $A$.  Since $E$ comes along with an embedding $i: E \subseteq \mb{P}^2$, 
there is also an invertible sheaf on $E$ defined by the pullback $\mc{L} = i^*(\mc{O}(1))$, where $\mc{O}(1)$ is 
the twisting sheaf of Serre \cite[Section II.5]{Ha}.
From the data $(E, \mc{L}, \sigma)$ one may construct a \emph{twisted homogeneous coordinate ring} $B = B(E, \mc{L}, \sigma)$.  
We will study this construction in more detail in the next lecture; for the purposes of this outline, it is important only to know that this is a certain graded ring built out of the geometry of $E$, whose properties can be analyzed using geometric techniques.   In particular, 
using algebraic geometry one can prove the following facts: (i) $B$ is generated in degree $1$; (ii) $B$ is noetherian; and (iii) $B$ has Hilbert series $1/(1-t)^3$ when $E = \mb{P}^2$ and $(1 - t^3)/(1-t)^3$ when $E$ is a cubic curve in $\mb{P}^2$.

The idea of the remainder of the proof is to study a canonical ring homomorphism 
$\phi: A \to B(E, \mc{L}, \sigma)$ built out of the geometric data coming from the point modules.   The construction of $\phi$ is quite formal, and it is automatically an isomorphism in degree $1$.  Since $B$ is generated in degree $1$, $\phi$ is surjective.  When $E = \mb{P}^2$, it follows from the Hilbert series that $A \cong B$, and so $A \cong B(\mb{P}^2, \mc{O}(1), \sigma)$.  These particular twisted homogeneous coordinate rings are known to have a fairly simple structure; for example, they can also be described as Zhang twists of commutative polynomial rings in three variables, as defined in the next lecture.
In the other case, where $E$ is a cubic curve, one wants to show that there is a normal nonzerodivisor $g \in A_3$ such that $gA = \ker \phi$ and thus $A/gA \cong B$.  If one happened to know in advance that $A$ was a domain, for example via a Gr{\"o}bner basis method, then this would easily follow from the known Hilbert series of $A$ and $B$.  However, one does not know this in general, particularly for the Sklyanin algebra.  Instead, this requires a detailed analysis of the presentation $k \langle x, y, z \rangle/J$ of $B$, again using geometry \cite[Sections 6, 7]{ATV1}. In particular, one proves that $J$ is generated as a two-sided ideal by three quadratic relations and one cubic relation.  The cubic relation provides the needed element $g$.  This also shows that $A$ is uniquely determined by $B$, since the three quadratic relations for $B$ are all of the relations for $A$.

The final classification of the quadratic AS-regular algebras $A$ which have a cubic curve $E$ as a point scheme is as follows:  there is a simple characterization of the possible geometric triples $(E, \mc{L}, \sigma)$ which can occur, and the corresponding rings $A$ are found by forgetting the cubic relation of each such $B(E, \mc{L}, \sigma)$ \cite[Theorem 3]{ATV1}.

One immediate corollary of this classification is the fact that any regular algebra of dimension $3$ is noetherian, 
since $B$ is noetherian and either $A \cong B$ or $A/gA \cong B$ (use Lemma~\ref{lem:passup}).   It is also true that all AS-regular algebras of dimension $3$ are domains \cite[Theorem 3.9]{ATV2}, though Lemma~\ref{lem:passup} does not immediately prove this in general, since the ring $B$ is a domain only if $E$ is irreducible.  For the Sklyanin algebra with generic parameters as in Example~\ref{ex:sklpm}, however, where $E$ is an elliptic curve, the ring $B$ is a domain, and hence $A$ is a domain by Lemma~\ref{lem:passup} in this case.   While this is not the only possible method for proving these basic properties of the Sklyanin algebra (see \cite{TV} for another approach), there is no method known which does not use the geometry of the elliptic curve $E$ in some way.

\subsection{Exercise Set 3}

\
\bigskip

1.  Fill in the details of the calculation of $X_2$ in Example~\ref{ex:qpolypm}, for instance using the matrix method 
of Example~\ref{ex:sklpm}.  In particular, show that $X_2$ is 
the graph of an automorphism $\sigma$ of $E \subseteq \mb{P}^2$, and find a formula for $\sigma$.
%It is helpful to use the presentation of the algebra corresponding to the resolution of $k$ in the special form 
%\eqref{eq:vmv}, which is calculated in Example~\ref{ex:change}.

\bigskip

2.  Let $A$ be the cubic regular algebra $k \langle x, y \rangle/(y^2x - xy^2,yx^2 - x^2y)$.  
Calculate the scheme parametrizing the point modules for $A$.   (Hint: study the image of $X_3 \subseteq \mb{P}^1 \times \mb{P}^1 \times \mb{P}^1$ under the projections $p_{12}, p_{23}: (\mb{P}^1)^{\times 3} \to (\mb{P}^1)^{\times 2}$.)

\bigskip

3.   Find the point modules for the algebra $A = k \langle x, y \rangle/(yx)$, by calculating 
explicitly what sequences $(p_0, p_1, p_2, \dots)$ of points in $\mb{P}^1$ as in Proposition~\ref{prop:param} 
are possible.  Show that the map $\phi_n: X_{n+1} \to X_n$ from that proposition is not an isomorphism 
for all $n \geq 1$.

\bigskip

4.  Let $A$ be AS-regular of global dimension $3$, generated in degree $1$.    Choose the free resolution of $k_A$ 
in the special form \eqref{eq:umv}, namely 
\[
0 \to A(-s-1) \overset{v^t}{\lra} A(-s)^{\oplus n} \overset{M}{\lra} A(-1)^{\oplus n} \overset{ v} {\lra} A \to 0,
\]
Let $A$ be presented by the particular relations which 
are the entries of the vector $v M$ (taking this product in the free algebra).  We also know that $(vM)^t = Q Mv^t$ in the free algebra, for some $Q \in \operatorname{GL}_3(k)$.

(a).  Generalizing 
Examples~\ref{ex:qpolypm} and \ref{ex:sklpm}, show that if $A$ is quadratic then $X_2 \subseteq \mb{P}^2 \times \mb{P}^2$ has equal first and second projections $E = p_1(X_2) = p_2(X_2) \subseteq \mb{P}^2$, and either $E = \mb{P}^2$ or $E$ is the vanishing of a cubic polynomial in $\mb{P}^2$.   

(b).  Formulate and prove an analogous result for a cubic regular algebra $A$.

\begin{comment}
\begin{enumerate}
\item[(a)] if $A$ is quadratic, then 
\item[(b)] if $A$ is cubic, then $X_3 \subseteq (\mb{P}^1)^{\times 3}$ has equal projections 
$E = p_{12}(X_3) = p_{23}(X_3) \subseteq (\mb{P}^1)^{\times 2}$ and either $E = (\mb{P}^1)^{\times 2}$ 
or $E$ is the vanishing of a degree $(2, 2)$ multlinear polynomial.
\end{enumerate}
(It is actually true that $E$ is the graph of an automorphism in either case; as mentioned in the last section of this 
lecture, this is much more subtle.)

\bigskip

5.  Verify the claim that there is a bijection $\phi_R$ from  
$P_m(R)$, the isomorphism classes of truncated $R$-point modules of length $m+1$, to $h_{X_m}(R) = \Hom_{k-\operatorname{schemes}}(\spec R, X_m)$, where $X_m$ is the subscheme of $\prod_{i=1}^m \mb{P}^n$ defined in Section 3.1.  Show also that these bijections are natural, i.e. 
given a homomorphism $\psi: R \to S$ we get $\phi_S \P_m(\psi) = h_{X_m}(S) \phi_R$.  This implies that the functors 
$h_{X_m}$ and $P_m$ are naturally isomorphic, as claimed above.  
(Hint:  use the theory of maps to projective space in \cite[II.7]{Ha}).
\end{comment}

\section{Lecture 4: Noncommutative projective schemes}
 
In the previous lecture, we saw that the parameter space of point modules is one important way that geometry appears naturally in the theory of noncommutative graded rings.  
In this lecture, we consider the more fundamental question of how to assign 
a geometric object to a noncommutative ring, generalizing the way that $\cproj A$ 
is assigned to $A$ when $A$ is commutative graded.   One possible answer, which has led to a very fruitful 
theory, is the idea of a noncommutative projective scheme defined by Artin and Zhang \cite{AZ1}.   In a nutshell, the basic idea is to give up on the actual geometric space, and instead generalize only the category of coherent sheaves to the noncommutative case.   The lack of an actual geometric space is less of a problem than one might at first think.   In fact, the study of the category of coherent sheaves on a commutative projective scheme (or its derived category) is of increasing importance in commutative algebraic geometry as an object of interest in its own right.

To begin this lecture, we will quickly review some relevant notions from the theory of schemes and sheaves, but this and the next lecture are primarily aimed at an audience already familiar with the basics in Hartshorne's book \cite{Ha}.  We also assume that the reader has familiarity with the concept of an Abelian category.  

Recall that a scheme $X$ is a locally ringed space with an open cover by affine schemes $U_{\alpha} = \spec R_{\alpha}$ \cite[Section II.2]{Ha}.  
We are primarily interested here in schemes of finite type over the base field $k$, so that $X$ has a such a cover by finitely many open sets, 
where each $R_{\alpha}$ is a finitely generated commutative $k$-algebra.   The most important way of producing projective schemes 
is by taking $\cproj$ of a graded ring.  Let $A$ be a \fig\ commutative $k$-algebra, generated by $A_1 = kx_0 + \dots + kx_n$.  For each $i$, we can localize the ring $A$ at the multiplicative system of powers of $x_i$, obtaining a ring $A_{x_i}$ which is now $\mb{Z}$-graded 
(since the inverse of $x_i$ will have degree $-1$).  Then the degree $0$ piece of $A_{x_i}$ is 
a ring notated as $A_{(x_i)}$.  The projective scheme $\cproj A$ has an open cover by the open affine schemes 
$U_i = \spec A_{(x_i)}$ \cite[Proposition II.2.5]{Ha}.  Recall that a sheaf $\mc{F}$ on a scheme $X$ is 
called \emph{quasi-coherent} if there is an open cover of $X$ 
by open affine sets $U_{\alpha} = \spec R_{\alpha}$, such that $\mc{F}(U_{\alpha})$ is the sheaf associated to an $R_{\alpha}$-module $M_{\alpha}$ for each $\alpha$ \cite[Section II.5]{Ha}.  The sheaf $\mc{F}$ is \emph{coherent} if each $M_{\alpha}$ is a finitely generated module.  If $X = \cproj A$ as above, then we can get quasi-coherent sheaves on $X$ as follows.   If $M$ is a $\mb{Z}$-graded $A$-module, then there is a sheaf $\wt{M}$ where $\wt{M}(U_i) = M_{(x_i)}$ is the degree $0$ piece of the localization $M_{x_i}$ of $M$ at the powers of $x_i$ \cite[Proposition II.5.11]{Ha}.  The sheaf $\wt{M}$ is coherent if $M$ is finitely generated.

The constructions above demonstrate how crucial localization is in commutative algebraic geometry.  The theory of localization for noncommutative rings is more limited: there is a well-behaved localization only at certain sets of elements called Ore sets (see \cite[Chapter 9]{GW}).   In particular, the set of powers of an element in a noncommutative ring is typically not an Ore set, unless the element is a normal element.   Thus it is problematic to try to develop a general theory of noncommutative schemes based around the notion of open affine cover.  There has been work in this direction for rings with 
``enough" Ore sets, however \cite{VOW}.

The actual space with a topology underlying a scheme is built out of prime ideals.  In particular, given 
a \fig\ commutative $k$-algebra $A$, as a set $\cproj A$  is the set of all homogeneous prime ideals of $A$, 
excluding the \emph{irrelevant ideal} $A_{\geq 1}$.  As a topological space, it has the Zariski topology, so the closed sets are those of the form $V(I) = \{ P \in \cproj A \, | \, P \supseteq I \}$, as $I$ varies over all homogeneous ideals of $A$.
Recall that for a not necessarily commutative ring $R$, an ideal $P$ is called \emph{prime} if 
$IJ \subseteq P$ for ideals $I, J$  implies that $I \subseteq P$ or $J \subseteq P$.   Thus one can define the space of homogeneous non-irrelevant 
prime ideals, with the Zariski topology, for any \fig\ $k$-algebra $A$.  For noncommutative rings, this is often a space which is too small to give a good geometric intuition.   The reader may verify the details of the following example in Exercise 4.1.
\begin{example}
\label{ex:primes}
Let $A = k \langle x, y \rangle/(yx - qxy)$ be the quantum plane, where $q \in k$ is not a root of unity.  Then the homogeneous prime ideals of $A$ are $(0), (x), (y), (x,y) = A_{\geq 1}$.   If $A = k \langle x, y \rangle/(yx - xy - x^2)$ is the Jordan plane, then the 
homogeneous prime ideals of $A$ are $(0), (x), (x,y) = A_{\geq 1}$.
\end{example}
\noindent 
On the other hand, both the Jordan and quantum planes have many properties in common with a commutative polynomial ring $k[x, y]$, and one would expect them to have an associated projective geometry which more closely resembles  $\mb{P}^1 = \cproj k[x, y]$.

We now begin to discuss the theory of noncommutative projective schemes, which finds a way around some of the difficulties 
mentioned above.   The key idea is that there are ways of defining and studying coherent sheaves on 
commutative projective schemes without explicit reference to an open affine cover.
Let $A$ be a \fig\ commutative $k$-algebra, and let $X = \cproj A$.
For a $\mb{Z}$-graded module $M$,  for any $n \in \mb{Z}$ we define $M_{\geq n} = \bigoplus_{i \geq n} M_i$, and 
call this a \emph{tail} of $M$.   It is easy to see that any two tails of a finitely generated graded $A$-module $M$ lead to the same coherent sheaf $\wt{M}$ on $X$.  Namely, given $n \in \mb{Z}$ and a finitely generated graded module $M$, we have the short exact sequence
\[
0 \to M_{\geq n} \to M \to M/M_{\geq n} \to 0,
\]
where the last term is finite dimensional over $k$.  Since localization is exact, and localization at the powers of an 
element $x_i \in A_1$ kills any finite-dimensional graded module, 
we see that the $A_{(x_i)}$-modules $(M_{\geq n})_{(x_i)}$ and $M_{(x_i)}$ are equal.   
Thus $\wt{M_{\geq n}} =\wt{M}$.  In fact, we have the following stronger statement.    
%Note that every projective $k$-scheme 
%is isomorphic to $\cproj A$ for a \fig\ algebra generated in degree $1$ \cite[Corollary II.5.16]{Ha}.
\begin{lemma}
\label{lem:coh}
Let $X = \cproj A$ for a commutative \fig\ $k$-algebra $A$ which is generated in degree $1$.
\begin{enumerate}
\item Every coherent sheaf on $X$ is isomorphic to $\wt{M}$ for some finitely generated graded $A$-module $M$ \cite[Proposition II.5.15]{Ha}.

\item Two finitely generated $\mb{Z}$-graded modules $M, N$ satisfy $\wt{M} \cong \wt{N}$ as sheaves if and only if there is an isomorphism of graded modules $M_{\geq n} \cong N_{\geq n}$ for some $n \in \mb{Z}$ \cite[Exercise II.5.9]{Ha}.
\end{enumerate}
\end{lemma}

We can interpret the lemma above in the following way:  the coherent sheaves on a projective scheme $\cproj A$ can 
be defined purely in terms of the $\mb{Z}$-graded modules over $A$, by identifying those 
 finitely generated $\mb{Z}$-graded modules with isomorphic tails.
%As we will see in a moment, morphisms of sheaves are also recovered from homomorphisms of the corresponding modules, so we can %actually obtain the category of coherent sheaves in this way.    
The scheme $\cproj A$ and its open cover play no role in this description.
One can use this idea to define a noncommutative analog of the category of coherent sheaves, as follows.

\begin{definition}
Let $A$ be a noetherian \fig\ $k$-algebra.  Let $\rgr A$ be the Abelian category of finitely generated $\mb{Z}$-graded 
right $A$-modules.  Let $\rtors A$ be its full subcategory of graded modules $M$ with $\dim_k M < \infty$; we call such 
modules \emph{torsion} in this context.    We define a new Abelian category $\rqgr A$.    The objects 
in this category are the same as the objects in $\rgr A$, and we let $\pi: \rgr A \to \rqgr A$ be the identity map 
on objects.  The morphisms are defined for $M, N \in \rgr A$ by 
\begin{equation}
\label{eq:hom}
\Hom_{\rqgr A}(\pi(M), \pi(N)) = \lim_{n \to \infty} \Hom_{\rgr A}(M_{\geq n}, N),
\end{equation}
where the direct limit on the right is taken over the maps of abelian groups 
$\Hom_{\rgr A}(M_{\geq n}, N) \to \Hom_{\rgr A}(M_{\geq n+1}, N)$ 
induced by the inclusion homomomorphisms $M_{\geq n+1} \to M_{\geq n}$.  
%(See \cite[]{Rot} for an introduction to the concept of a direct limit).  

The pair $(\rqgr A, \pi(A))$ is called the \emph{noncommutative projective scheme} associated to the graded ring $A$.  
The object $\pi(A)$ is called the \emph{distinguished object} and plays the role of the structure sheaf.
The map $\pi: \rgr A \to \rqgr A$ is a functor, the \emph{quotient functor}, which sends the morphism $f: M \to N$ 
to $f \vert_{M_{\geq 0}} \in \Hom(M_{\geq 0}, N)$ in the direct limit.
\end{definition}

The passage from $\rgr A$ to $\rqgr A$ in the definition above is a special case of a more general abstract construction 
called a \emph{quotient category}.  See \cite[Section 2]{AZ1} for more discussion.  It may seem puzzling at 
first that $\rgr A$ and $\rqgr A$ are defined to have the same set of objects.  However, some objects in $\rqgr A$ become 
isomorphic to each other that were not isomorphic in $\rgr A$, so there is indeed a kind of quotienting of the set of objects, at the level 
of isomorphism classes.  For example, it is easy to see from the definition that $\pi(M) \cong \pi(M_{\geq n})$ in the category $\rqgr A$, for any graded module $M$ and $n \in \mb{Z}$ (Exercise 4.2).  

\begin{example}
\label{ex:comm}
If $A$ is a commutative \fig\ $k$-algebra, generated by $A_1$, then there is an equivalence 
of categories  $\Phi: \rqgr A \to \coh X$, where $\coh X$ is the category of coherent sheaves on $X = \cproj A$, and where  $\Phi(\pi(A)) = \mc{O}_X$ \cite[Exercise II.5.9]{Ha}.  
%Given Lemma~\ref{lem:coh}, this primarily amounts to checking that for any finitely generated $\mb{Z}$-graded $A$ modules $M, N$, the morphisms of %sheaves $\Hom_{\coh X}(\wt{M}, \wt{N})$ are in natural bijection with $\lim_{n \to \infty} \Hom_{\rgr A}(M_{\geq n}, N)$.  See \cite[]{Ha}.
Thus commutative projective schemes (or more properly, their categories of coherent sheaves) are special cases of noncommutative projective schemes.
\end{example}

The definition of a noncommutative projective scheme is indicative of a general theme in noncommutative geometry.  Often there are many equivalent ways of thinking about a concept in the commutative case.  For example, the idea of a point module of a commutative \fig\ algebra $B$ and the idea of a (closed) point of $\cproj B$ (that is, a maximal element among nonirrelevant homogeneous prime ideals) are just two different ways of getting at the same thing.  But their analogues in the noncommutative case are very different, and point modules are more interesting for the quantum plane, say, than prime ideals are.   Often, in trying to generalize commutative concepts to the noncommutative case, one has to find a way to formulate the concept in the commutative case, often not the most obvious one, whose noncommutative generalization gives the best intuition.

We have now defined a noncommutative projective scheme associated to any connected \fig\ $k$-algebra.  This raises many questions, for instance, what are these categories like?  What can one do with this construction?
We will first give some more examples of these noncommutative schemes, and then we will discuss some of the interesting applications.

\subsection{Examples of noncommutative projective schemes}

It is natural to study the category $\rqgr A$ in the special case that $A$ is an Artin-Schelter regular algebra.  
We will see that for such $A$ of global dimension $2$, the category $\rqgr A$ is something familiar.  For $A$ of global dimension $3$ one obtains new kinds of categories, in general.

First, we study a useful general construction.  
\begin{definition}
Let $A$ be an $\mb{N}$-graded $k$-algebra.  Given a graded algebra automorphism $\tau: A \to A$, the \emph{Zhang twist} of $A$ by $\tau$ is a new algebra $B = A^{\tau}$.  It has the same underlying graded $k$-vector space as $A$, but a new product $\star$ defined on homogenous elements by $a \star b = a \tau^m(b)$ for $a \in B_m, b \in B_n$.
\end{definition}
The reader may easily verify that the product on $A^{\tau}$ is associative.  Given an $\mb{N}$-graded $k$-algebra $A$, 
let $\rGr A$ be its Abelian category of $\mb{Z}$-graded $A$-modules.  One of the most important features of the twisting construction is that it preserves the category of graded modules.
\begin{theorem} \cite[Theorem 1.1]{Zh1}
\label{thm:twist}
For any $\mb{N}$-graded $k$-algebra $A$ with graded automorphism $\tau: A \to A$, there is an equivalence of categories  $F: \rGr A \to \rGr A^{\tau}$.  The functor $F$ is given on objects by $M \mapsto M^{\tau}$, where $M^{\tau}$ is the same 
as $M$ as a $\mb{Z}$-graded $k$-space, but with new $A^{\tau}$-action $\star$ defined by $x \star b = x\tau^m(b)$, for $x \in M_m, b \in A_n$.  The functor $F$ acts as the identity map on morphisms.   
\end{theorem}
\begin{proof}
Given $M \in \rGr A$, the proof that $M^{\tau}$ is indeed an $A^{\tau}$-module 
is formally the same as the proof that $A^{\tau}$ is associative, and it is obvious that $F$ is a functor.  

Now $\tau: A^{\tau} \to A^{\tau}$, given by the same underlying map as $\tau: A \to A$, is 
an automorphism of $A^{\tau}$ as well.  Thus we can define the twist $(A^{\tau})^{{\tau}^{-1}}$, 
which is isomorphic to $A$ again via the identity map of the underlying $k$-space.  Thus we also 
get a functor $G: \rGr A^{\tau} \to \rGr A$ given by $N \mapsto N^{\tau^{-1}}$.  The reader may easily 
check that $F$ and $G$ are inverse equivalences of categories.
\end{proof}
\noindent The reader may find more details about the twisting construction defined above in \cite{Zh1}.  In fact,  Zhang defined a slightly more general kind of twist depending on a \emph{twisting system} instead of just a choice of graded automorphism, and proved that these more general twists of a connected \fig\ algebra $A$ with $A_1 \neq 0$ produce all $\mb{N}$-graded algebras $B$ such that $\rGr A$ and $\rGr B$ are equivalent categories
\cite[Theorem 1.2]{Zh1}.

Suppose now that $A$ is \fig\ and noetherian.  Given the explicit description of the equivalence $\rGr A \to \rGr A^{\tau}$ in Theorem~\ref{thm:twist}, it is clear that it restricts to the subcategories of finitely generated modules to give an equivalence $\rgr A \to \rgr A^{\tau}$.   Moreover, the 
subcategories of modules which are torsion (that is, finite-dimensional over $k$) also correspond, and so one gets an equivalence $\rqgr A \to \rqgr A^{\tau}$, under which the distinguished objects $\pi(A)$ and $\pi(A^{\tau})$ correspond.   Since a commutative graded ring typically has numerous noncommutative Zhang twists, this can be used to give many examples of noncommutative graded rings whose noncommutative projective schemes are simply isomorphic to commutative projective schemes.

\begin{example}
Let $A = k [x, y ]$ be a polynomial ring, with graded automorphism $\sigma: A \to A$ given 
by $\sigma(x) = x, \sigma(y) = y-x$.   In $B = A^{\sigma}$ we calculate 
$x \star y = x(y-x)$, $y \star x = xy$, $x \star x = x^2$.  Thus we have the relation
\[
x \star y - y \star x + x \star x = 0.
\]
Then there is a surjection $C = k \langle x, y \rangle/(yx - xy - x^2) \to B$, where $C$ is the Jordan plane 
from Examples~\ref{ex:qplane}.  Since $h_C(t) = 1/(1-t)^2$ and $h_B(t) = h_A(t) = 1/(1-t)^2$, $B \cong C$.
Now there is an equivalence of categories $\rqgr B \simeq \rqgr A$ by the remarks above.   Since $\rqgr A \simeq \coh \mb{P}^1$ by Example~\ref{ex:comm}, we see that the noncommutative projective scheme $\rqgr C$ associated to the Jordan plane is just the commutative scheme $\mb{P}^1$.

A similar argument shows that the quantum plane $C$ from Examples~\ref{ex:qplane} is also isomorphic to a twist of $k[x, y]$, and so 
$\rqgr C \simeq \coh \mb{P}^1$ for the quantum plane also.
%, so the noncommutative projective scheme associated to the Jordan plane is simply the commutative scheme $\mb{P}^1$.
\end{example}

\begin{example}
\label{ex:qpoly2}
Consider the quantum polynomial ring $A = k \langle x, y , z \rangle/(yx - pxy, zy - qyz, xz-rzx)$ from Example~\ref{ex:qpolypm}.   
It is easy to check that $A$ is isomorphic to a Zhang twist of $k[x, y, z]$ if and only if $pqr = 1$ (Exercise 4.3).
Thus when $pqr = 1$, one has $\rqgr A \simeq \rqgr k[x, y, z] \simeq \coh \mb{P}^2$.  We also saw in Example~\ref{ex:qpolypm} that $A$ has point modules parametrized by $\mb{P}^2$ in this case.  This can also also be proved using Zhang twists:  it is clear that the isomorphism classes of point modules of $B$ and $B^{\tau}$ are in bijection under the equivalence of categories $\rgr B \simeq \rgr B^{\tau}$, for any \fig\ algebra $B$.

On the other hand, when $pqr \neq 1$, we saw in Example~\ref{ex:qpolypm} that the point modules of $A$ are parametrized by a union of three lines.  In this case it is known that $\rqgr A$ is not equivalent to the category of coherent sheaves on any commutative projective scheme, although we do not prove this assertion here.   It is still appropriate to think of the noncommutative projective scheme $\rqgr A$ as a kind of noncommutative $\mb{P}^2$, 
but this noncommutative projective plane has only a one-dimensional subscheme of closed points on it.
\end{example}

\subsection{Coordinate rings}

Above, we associated a category $\rqgr A$ to a noetherian \fig\ algebra $A$.  
It is important that one can also work in the other direction and construct graded rings from categories.  
\begin{definition}
Let $\mc{C}$ be an Abelian category such that each $\Hom_{\mc{C}}(\mc{F}, \mc{G})$ is a $k$-vector space, and such that composition of morphisms is $k$-bilinear; we say that $\mc{C}$ is a \emph{$k$-linear} Abelian category.
Let $\mc{O}$ be an object in $\mc{C}$ and let $s: \mc{C} \to \mc{C}$ be 
an \emph{autoequivalence}, that is, an equivalence of categories from $\mc{C}$ to itself.
Then we define an $\mb{N}$-graded $k$-algebra $B = B(\mc{C}, \mc{O}, s) = \bigoplus_{n \geq 0} \Hom_{\mc{C}}(\mc{O}, s^n\mc{O})$, where the product is given on homogeneous elements $f \in B_m, g \in B_n$ by $f \star g = s^n(f) \circ  g$.
\end{definition}
%\noindent The $\mb{N}$-graded $k$-algebras $B = B(\mc{C}, \mc{O}, s)$ certainly need not be \fig\ in the generality %above, but they will be in the important special cases we study below.  

In some nice cases, the construction above allows one to recover a \fig\ algebra $A$ from the data of its noncommutative  projective  scheme $(\rqgr A, \pi(A))$ and a natural shift functor.
\begin{proposition}
\label{prop:goback}
Let $A$ be a noetherian \fig\ $k$-algebra with $A \neq k$.  Let $\mc{C} = \rqgr A$ and let $s = (1)$ be the autoequivalence of $\mc{C}$ 
induced from the shift functor $M \mapsto M(1)$ on modules $M \in \rgr A$.  Let $\mc{O} = \pi(A)$.

\begin{enumerate} 
\item There is a canonical homomorphism of rings $\phi: A \to B(\mc{C}, \mc{O}, s)$.  The kernel $\ker \phi$ is finite-dimensional over $k$ and is $0$ if $A$ is a domain.  

\item If $\uExt^1_A(k_A, A_A)$ is finite-dimensional over $k$, then $\coker \phi$ is also finite-dimensional; this is always the case if $A$ is commutative.   Moreover, if $\uExt^1_A(k_A, A_A) = 0$, then $\phi$ is surjective. 
\end{enumerate}
\end{proposition}
\begin{proof}
Let $B = B(\mc{C}, \mc{O}, s)$.   
%It is straightforward to prove that the functor $\uHom_A( M, - )$ commutes with direct sums, as long as $M$ is finitely generated. 
%\cite[Proposition 3.1(1)]{AZ1}.  
Then 
\[
B= \bigoplus_{n \geq 0} \Hom_{\mc{C}}(\mc{O}, s^n(\mc{O})) = 
\bigoplus_{n \geq 0} \lim_{i \to \infty} \Hom_{\rgr A}(A_{\geq i}, A(n)) = \lim_{i \to \infty} \uHom_A(A_{\geq i}, A)_{\geq 0}.
\]

Thus there is a map $\phi: A \to B$ which sends a homogeneous element $x \in A_n$ to the corresponding element  
$l_x \in \uHom_A(A, A)$ in the $i = 0$ part of this direct limit, where $l_x(y)= xy$.  It is easy to see 
from the definition of the multiplication in $B$ that $\phi$ is a ring homomorphism.    For each exact sequence $0 \to A_{\geq i} \to A \to A/A_{\geq i} \to 0$ 
we can apply $\uHom_A( -, A)_{\geq 0}$, and write the corresponding long exact sequence in $\uExt$.     It is easy to see that the direct limit of exact sequences of Abelian groups is exact, so we obtain an exact sequence 
\[
0 \to \lim_{i \to \infty} \uHom_A(A/A_{\geq i}, A)_{\geq 0} \to A \overset{\phi}{\to} B 
\to \lim_{i \to \infty} \uExt_A^1(A/A_{\geq i}, A)_{\geq 0} \to 0.
\]
The proof is finished by showing that $\D \lim_{i \to \infty} \uExt_A^j(A/A_{\geq i}, A)_{\geq 0}$ is finite dimensional (or $0$) 
as long as $\uExt^j_A(k, A)$ is finite dimensional (or $0$, respectively).  We ask the reader to complete the proof in Exercise 4.4.  Note that $\uHom_A(k, A)$ is certainly finite-dimensional, since $A$ is noetherian, and $\uHom_A(k, A) = 0$ if $A$ is a domain (since $A \neq k$).  The claim that 
$\uExt^1_A(k, A)$ is always finite-dimensional in the commutative case is also part of Exercise 4.4.
\end{proof}
The proposition above shows that for a noetherian \fig\ domain $A$, one recovers $A$ in large degree from its category $\rqgr A$,  
as long as $\uExt^1_A(k, A)$ is finite-dimensional.  This kind of interplay between categories and graded rings is very useful.

Interestingly, it is not always true for a noncommutative \fig\ algebra that $\uExt^1_A(k, A)$ is finite-dimensional.  We ask 
the reader to work through an example, which was first given by Stafford and Zhang in \cite{StZh}, in Exercise 4.5.  More generally, 
Artin and Zhang defined the \emph{$\chi$-conditions} in \cite{AZ1}, as follows.    A \fig\ noetherian algebra $A$ satisfies $\chi_i$ if 
$\uExt^j_A(k, M)$ is finite-dimensional for all finitely generated $\mb{Z}$-graded $A$-modules $M$ and for all $j \leq i$; the algebra $A$ satisfies $\chi$ if it satisfies $\chi_i$ for all $i \geq 0$.  If $A$ satisfies $\chi_1$, this is enough 
for the map $\phi: A \to B(\mc{C}, \mc{O}, s)$ of Proposition~\ref{prop:goback} to have finite-dimensional cokernel, but other parts of the theory (such as the cohomology we discuss below) really work well only for graded rings satisfying the full $\chi$ condition.  Fortunately, $\chi$ holds for many important classes of rings, for example noetherian AS-regular algebras \cite[Theorem 8.1]{AZ1}.

\begin{comment}
Artin and Zhang also studied the following problem:  given a $k$-linear category $\mc{C}$ and a \fig\ noetherian $B = B(\mc{C}, \mc{O}, s)$, when is $B$  actually a coordinate ring of $\mc{C}$ in the sense that $\rqgr B = \mc{C}$?   They show that this depends on a condition on the autoequivalence $s$ called \emph{ampleness} \cite[]{AZ1}.   The precise definition is too technical for these notes, but we discuss it in the context of a particular example which we study in the next section.
\end{comment}

\subsection{Twisted homogeneous coordinate rings}

Twisted homogeneous coordinate rings already made a brief appearance at the end of the previous lecture, in the outline 
of the classification of AS-regular algebras of dimension $3$.    
Because such rings are defined purely geometrically, the analysis of their properties often reduces to questions of commutative algebraic geometry.   These rings also occur naturally (for example, in the study of regular algebras, as we have already seen), and in some cases would be very difficult to study and understand without the geometric viewpoint.

We now give the precise definition of twisted homogeneous coordinate rings, relate them 
to the general coordinate ring construction of the previous section, and work through an example.   

\begin{definition}
\label{def:thcr}
Let $X$ be a projective scheme defined over the base field $k$.    Let $\mc{L}$ be an invertible sheaf on $X$, 
and let $\sigma: X \to X$ be an automorphism of $X$.  
% is a locally free sheaf of rank $1$.  Equivalently, a coherent sheaf $\mc{L}$ is invertible if there exists another sheaf $\mc{M}$ such that $\mc{L} \otimes_{\mc{O}_X} \mc{M} \cong \mc{O}_X$ \cite[Prop. II.6.12]{Ha}.  
We use the notation $\mc{F}^{\sigma}$ for the pullback sheaf $\sigma^*(\mc{F})$.  
Let $\mc{L}_n = \mc{L} \otimes \mc{L}^{\sigma} \otimes \dots \otimes (\mc{L})^{\sigma^{n-1}}$ for each $n \geq 1$, and let $\mc{L}_0 = \mc{O}_X$.   Let $H^0(X, \mc{F})$ be the global sections $\mc{F}(X)$ of a sheaf $\mc{F}$.   For any sheaf, there is a natural pullback of global sections map $\sigma^*: H^0(X, \mc{F}) \to H^0(X, \mc{F}^{\sigma})$.  

We now define a graded ring $B = B(X, \mc{L}, \sigma) = \bigoplus_{n \geq 0} B_n$, called the 
\emph{twisted homogeneous coordinate ring} associated to this data.   Set  
$B_n = H^0(X, \mc{L}_n)$ for $n \geq 0$,  and define the multiplication on $B_m \otimes_k B_n$ via 
the following chain of maps:
\[
H^0(X, \mc{L}_m) \otimes H^0(X, \mc{L}_n) \overset{1 \otimes (\sigma^m)^*}{\to} H^0(X, \mc{L}_m) \otimes H^0(X, \mc{L}_n^{\sigma^m}) \overset{\mu}{\to} H^0(X, \mc{L}_m \otimes \mc{L}_n^{\sigma^m})  = H^0(X, \mc{L}_{m+n}), 
\]  
where $\mu$ is the natural multiplication of global sections map.
\end{definition}

We can also get these rings as a special case of the construction in the previous section.
Consider $\mc{C} = \coh X$ for some projective $k$-scheme $X$.  
If $\mc{L}$ is an invertible sheaf on $X$, 
then $\mc{L} \otimes_{\mc{O}_X} -$ is an autoequivalence of $\mc{C}$.  For any automorphism $\sigma$ of $X$, the pullback map $\sigma^*(-)$  is also an autoequivalence of the category $\mc{C}$.   It is known that in fact 
an arbitrary autoequivalence of $\mc{C}$ must be a composition of these two types, in other words it must have the form $s = ( \mc{L} \otimes \sigma^*(-))$ for some $\mc{L}$ and $\sigma$ \cite[Proposition 2.15]{AV}, \cite[Corollary 6.9]{AZ1}.  Now we may define the ring $B = B(\mc{C}, \mc{O}_X, s)$, and an exercise in tracing through the definitions shows that this ring is isomorphic to the twisted homogeneous coordinate ring $B(X, \mc{L}, \sigma)$ defined above.   
%Thus the twisted homogeneous coordinate rings are all of the general coordinate rings of 
%$\coh X$ using the distinguished object $\mc{O}_X$ and an arbitrary autoequivalence.

It is also known when $\rqgr B$ recovers the category $\coh X$.  Recalling that 
$\mc{L}_n = \mc{L} \otimes \mc{L}^{\sigma} \otimes \dots \otimes \mc{L}^{\sigma^{n-1}}$, then 
$\mc{L}$ is called \emph{$\sigma$-ample} if for any coherent sheaf $\mc{F}$, one has $H^i(X,\mc{F} \otimes \mc{L}_n) = 0$ for all $n \gg 0$ and all $i \geq 1$.   When $\sigma = 1$, this is just one way to define the ampleness of $\mc{L}$ in 
the usual sense \cite[Proposition III.5.3]{Ha}.  In fact, Keeler completely characterized $\sigma$-ampleness \cite[Theorem 1.2]{Ke}.  
In particular, when $X$ has at least one $\sigma$-ample sheaf, then $\mc{L}$ is $\sigma$-ample if and only if 
$\mc{L}_n$ is ample for some $n \geq 1$, so it is easy to find $\sigma$-ample sheaves in practice.
When $\mc{L}$ is $\sigma$-ample, then $B = B(X, \mc{L}, \sigma)$ is noetherian and $\rqgr B \sim \coh X$ \cite[Theorems 1.3, 1.4]{AV}, and so this construction produces many different ``noncommutative coordinate rings" of $X$.

It can be difficult to get a intuitive feel for the twisted homogeneous coordinate ring construction, so we work out the explicit details of a simple example. 
\begin{example}
\label{ex:thcr-calc}
We calculate a presentation for $B(\mb{P}^1, \mc{O}(1), \sigma)$, where $\sigma: \mb{P}^1 \to \mb{P}^1$ is given 
by the explicit formula $(a:b) \mapsto (a: a+b)$, and $\mc{O}(1)$ is the twisting sheaf of Serre as in \cite[Sec. II.5]{Ha}.

Let $R = k[x, y]$ and $\mb{P}^1 = \cproj R$, with its explicit open affine cover $U_1 = \spec R_{(x)} = \spec k[u]$, and $U_2 = \spec R_{(y)} = \spec k[u^{-1}]$, where $u = yx^{-1}$.  Then the field of rational functions of $\mb{P}^1$ 
is explicitly identified with the field $k(u)$, the fraction field of both $R_{(x)}$ and $R_{(y)}$.   Let $\mc{K}$ 
be the constant sheaf on $\mb{P}^1$ whose value is $k(u)$ on every nonempty open set.    
In doing calculations 
with a twisted homogeneous coordinate ring on an integral scheme, it is useful to embed all invertible sheaves explicitly 
in the constant sheaf of rational functions $\mc{K}$, which is always possible by \cite[Prop. II.6.13]{Ha}.  
The sheaf $\mc{O}(1)$ is defined abstractly as the coherent sheaf $\wt{R(1)}$ associated to the graded module $R(1)$ as in \cite[Section II.5]{Ha}, but it is not hard to see that $\mc{O}(1)$ is isomorphic to the subsheaf $\mc{L}$ of $\mc{K}$ generated by the global sections $1, u$, whose sections on the two open sets of the cover 
are 
\[
\mc{L}(U_1) = 1 k[u] + u k[u] = k[u] \ \ \  \text{and}\ \ \ \mc{L}(U_2) = 1k[u^{-1}] + u k[u^{-1}] = uk[u^{-1}].
\] 
Then the space of global sections of $\mc{L}$ is just the intersection of the sections on the two 
open sets, namely $H^0(\mb{P}^1, \mc{L}) = k + ku$.
%(Though we won't encounter such subtleties in this example, it is important to note that in general a sheaf globally %generated by some set of sections can have a larger space of global sections than the $k$-span of the generating %set.)

The automorphism of $\mb{P}^1$ induces an automorphism of the field $k(u)$, defined 
on rational functions $f: \mb{P}^1 \dashrightarrow k$ by $f \mapsto f \circ \sigma$.  We call this automorphism 
$\sigma$ also, and it is straightforward to calculate the formula $\sigma(u) = u + 1$.
%since $u = yx^{-1}$ is the map $(a:b) \mapsto b/a$, so 
%precomposing with $\sigma$ gives $(a:b) \mapsto (a: a+b) \mapsto (a+b)/a = 1 + b/a$.  
Then given an invertible subsheaf $\mc{M} \subseteq \mc{K}$, such that $\mc{M}$ is generated by its global sections $V = H^0(\mb{P}^1, \mc{M})$, the subsheaf of $\mc{K}$ generated by the global sections 
$\sigma^n(V)$ is isomorphic to the pullback $\mc{M}^{\sigma^n}$.  In our example,  letting $V = k + ku$, then $\mc{L}^{\sigma^i}$ is the subsheaf of $\mc{K}$ 
generated by $\sigma^i(V) = V$ (that is, $\mc{L}^{\sigma^i} = \mc{L}$) and 
$\mc{L}_n = \mc{L} \otimes \mc{L}^{\sigma} \otimes \dots \otimes \mc{L}^{\sigma^{n-1}}$ is the 
sheaf generated by $V_n = V\sigma(V) \dots \sigma^{n-1}(V) = V^n = k + ku+ \dots + ku^{n}$.  
Also, $H^0(\mb{P}^1, \mc{L}_n) = V_n$, by intersecting the sections on each of the two open sets, as for $n = 1$ above.
(We caution that in more general examples, $\mc{L}_n$ is not isomorphic to $\mc{L}^{\otimes n}$.)  

Finally, with all of our invertible sheaves $\mc{M}, \mc{N}$ explicitly embedded in $\mc{K}$ as above, and thus their global sections embedded in $k(u)$, the pullback of global sections map $\sigma^*: H^0(\mb{P}^1, \mc{M}) \to H^0(\mb{P}^1, \mc{M}^{\sigma})$ is simply given by applying the automorphism $\sigma$ of $k(u)$, 
and the multiplication of sections map $H^0(\mb{P}^1, \mc{M}) \otimes H^0(\mb{P}^1, \mc{N}) \to H^0(\mb{P}^1, \mc{M} \otimes \mc{N})$ 
is simply multiplication in $k(u)$.  
Thus $B = B(\mb{P}^1, \mc{L}, \sigma) = \bigoplus_{n \geq 0} V_n$, with multiplication on homogeneous elements 
$f \in V_m, g \in V_n$ given by $f \star g = f\sigma^m(g)$.  It easily follows that $B$ is generated in degree $1$, 
and putting $v = 1, w = u \in V_1$, we immediately calculate  
$v \star w = (u+1), w \star v = u, v \star v= 1$, giving the relation $v \star w = w \star v + v \star v$.
It follows that there is a surjective graded homomorphism $A = k \langle x, y \rangle/(xy - yx - x^2) \to B$. 
But $A$ is easily seen to be isomorphic to the Jordan plane from Examples~\ref{ex:jordan} (apply the change 
of variable $x \mapsto -x, y \mapsto y$).  So $A$ and $B$ have the same Hilbert series, and we conclude 
that $A \cong B$.
\end{example}

The quantum plane also arises as a twisted homogeneous coordinate ring of $\mb{P}^1$, by a similar calculation as in the previous example.  Of  course, this is not the simplest way to describe the Jordan and quantum planes.  The point is that many less trivial examples, such as the important case $B(E, \mc{L}, \sigma)$ where $E$ is an elliptic curve, do not arise in a more na{\"i}ve way.  The twisted homogeneous coordinate ring formalism is the simplest way of defining these rings, and their properties are most easily analyzed using geometric techniques.
For more details about twisted homogeneous coordinate rings, see the original paper of Artin and Van den Bergh \cite{AV} and the work of Keeler \cite{Ke}.

\subsection{Further applications}
 
Once one has a category $\rqgr A$ associated to a \fig\ noetherian $k$-algebra $A$, one can try to formulate and study all 
sorts of geometric concepts, such as points, lines, closed and open subsets, and so on, using this category.
As long as a geometric concept for a projective scheme can be phrased in terms of the category of coherent sheaves, then one can attempt to transport it  to noncommutative projective schemes.   For example, see \cite{Sm1}, \cite{Sm2} for some 
explorations of the notions of open and closed subsets and morphisms for noncommutative schemes.

If $X$ is a commutative projective $k$-scheme, then for each point $x \in X$ there is a corresponding 
skyscraper sheaf $k(x) \in \coh X$, with stalks $\mc{O}_x \cong k$ and $\mc{O}_y = 0$ for all closed points $y \neq x$.  
This is obviously a simple object in the Abelian category $\coh X$ (that is, it has no subobjects other than $0$ and itself) and it is not hard to see that such skyscraper sheaves are the only simple objects in this category.
Since simple objects of $\coh X$ correspond to points of $X$, so one may think of the simple objects of $\rqgr A$ in general 
as the ``points" of a noncommutative projective scheme.  This connects nicely with the previous lecture, as follows.
\begin{example}
Let $M$ be a point module for a \fig\ noetherian $k$-algebra $A$ which is generated in degree $1$.  We claim that 
$\pi(M)$ is a simple object in the category $\rqgr A$.  For, note that the only graded submodules of $M$ are $0$ and the tails $M_{\geq n}$ for $n \geq 0$.  Also, all tails of $M$ have $\pi(M_{\geq n}) = \pi(M)$ in $\rqgr A$ (Exercise 4.1).
Given a nonzero subobject of $\pi(M)$, it has the form $\pi(N)$ for some graded $A$-module $N$, and the monomorphism  
$\pi(N) \to \pi(M)$ corresponds to some nonzero element of $\Hom_{\rgr A}(N_{\geq n}, M)$, whose image must therefore be a tail of $M$; but then the map $\pi(N) \to \pi(M)$ is an epimorphism and hence an isomorphism.  Thus $\pi(M)$ is simple, as claimed.
%This gives another reason why point modules should correspond to (closed) points in noncommutative projective geometry.

A similar proof shows that given an arbitrary finitely generated $\mb{Z}$-graded module $M$ of $A$, $\pi(M)$ is a simple object 
in $\rqgr A$ if and only if every graded submodule $N$ of $M$ has $\dim_k M/N < \infty$ (or equivalently, $M_{\geq n} \subseteq N$ for some $n$).  
Such modules are called \emph{1-critical} (with respect to Krull dimension; see \cite[Chapter 15]{GW}).
It is possible for a \fig\ algebra $A$ to have such modules $M$ which 
are bigger than point modules; for example, some AS-regular algebras of dimension 3 have $1$-critical modules $M$ with $\dim_k M_n = d$ for all $n \gg 0$, some $d > 1$  \cite[Note 8.43]{ATV2}.  The corresponding simple object $\pi(M)$ in $\rqgr A$ is sometimes called a \emph{fat point}.
\end{example}

One of the most important tools in algebraic geometry is the cohomology of sheaves.   One possible 
approach is via \v{C}ech cohomology \cite[Section III.4]{Ha}, which is defined using an open affine cover of the scheme, and thus doesn't generalize 
in an obvious way to noncommutative projective schemes.  
However, the modern formulation of sheaf cohomology due to Grothendieck, which uses injective resolutions, generalizes easily.   Recall that if $X$ is a (commutative) projective scheme and $\mc{F}$ is a quasi-coherent sheaf on $X$, then one defines its cohomology groups by $H^i(X, \mc{F}) = \Ext^i_{\Qcoh X}(\mc{O}_X, \mc{F})$ \cite[Section III.6]{Ha}.  Although the category of quasi-coherent sheaves does not have enough projectives, it has enough injectives, so such Ext groups can be defined using an injective resolution of $\mc{F}$.  It is not sufficient to work in the category of coherent sheaves here, since injective sheaves are usually non-coherent.

Given a \fig\ noetherian $k$-algebra $A$, we defined the category $\rqgr A$, which is an analog of coherent sheaves.  We did not define an analog of quasi-coherent sheaves above, for reasons of simplicity only.  In general, one may define a category $\rQgr A$, by starting with the category $\rGr A$ of all $\mb{Z}$-graded $A$-modules and defining an appropriate quotient category by the subcategory of torsion modules, where the torsion modules in this case are the direct limits of finite-dimensional modules.  The category $\rQgr A$ is the required noncommutative analog of the category of quasi-coherent sheaves.   We omit the precise definition of $\rQgr A$, which requires a slightly more complicated definition of the Hom sets than \eqref{eq:hom}; see \cite[Section 2]{AZ1}.  Injective objects and injective resolutions exist 
in $\rQgr A$, so $\Ext$ is defined.  Thus the natural generalization of Grothendieck's definition of 
cohomology is $H^i(\rQgr A, \mc{F}) = \Ext^i_{\rQgr A}(\pi(A), \mc{F})$, for any object $\mc{F} \in \rQgr A$.
See \cite[Section 7]{AZ1} for more details.

Once the theory of cohomology is in place, many other concepts related to cohomology can be studied for noncommutative projective schemes.  To give just one example, there is a good analog of Serre duality, which holds a number of important cases \cite{YZ}.
In another direction, one can study the bounded derived category $D^b(\rqgr A)$, and ask (for example) when two such derived categories are equivalent, as has been studied for categories of commutative coherent sheaves.

\begin{comment}
One of the deepest applications so far has been 
Van den Bergh's theory of noncommutative blowing up \cite[]{}.  Although it is not defined for all noncommutative projective schemes, Van den Bergh shows, for a noncommutative surface containing a commutative curve (for example $\rqgr A$ for any AS-regular algebra $A$ of global dimension $3$), how to define the blowup of the surface at a point on the curve, in a purely categorical way.   Since blowing up 
at a point is one of the crucial ingredients in the classification theory of surfaces, the theory of noncommutative blowing up is an important 
start to the classification of noncommutative surfaces.  
\end{comment}

We close this lecture by connecting it more explicitly with the second lecture.  Since AS-regular algebras $A$ are intuitively noncommutative analogs of (weighted) polynomial rings,  their noncommutative projective schemes $\rqgr A$ should be thought of as analogs of (weighted) projective spaces.
Thus these should be among the most fundamental noncommutative projective schemes, and this gives another motivation for the importance of regular algebras.  One difficult aspect of the noncommutative theory, however, is a lack of a general way to find projective embeddings.   Many important examples of \fig\ algebras $A$ are isomorphic to factor algebras of AS-regular algebras $B$, and thus $\rqgr A$ can be thought of as a closed subscheme of the noncommutative projective space $\rqgr B$.  However, there is as yet no theory showing that some reasonably general class of graded algebras must arise as factor algebras of AS-regular algebras.

\begin{comment}
We remark that $\rqgr A$, where $A$ is AS-regular of global dimension $3$ with the same Hilbert series $1/(1-t)^3$ as a 
polynomial ring in three variables, is thought of as a ``noncommutative $\mb{P}^2$".  There are further justifications for this other than the fact that regular algebras have the same homological properties as polynomial rings; Bondal and Polishchuk gave a more categorical argument that these categories are the right ones.  We saw in the last lecture that these algebras sometimes have only a curve of point modules.  Thus we think of these as noncommutative $\mb{P}^2$'s which have only a curve of points.  We also remark that $\rqgr A$ can have in general simple objects other than the images of point modules of $A$.  
These are ``fat points" which arise as the images of $1$-critical modules of larger Hilbert series.
\end{comment}

\subsection{Exercise Set 4}

\
\bigskip

1. Show that the graded prime ideals of the quantum plane $k \langle x, y \rangle/(yx - qxy)$ (for $q$ not a root of $1$) 
and the Jordan plane $k \langle x, y \rangle/(yx - xy -x^2)$ are as claimed in Example~\ref{ex:primes}.

\bigskip

2.  Let $A$ be a connected \fig\ noetherian $k$-algebra.  Given two finitely generated $\mb{Z}$-graded $A$-modules $M$ and $N$, 
prove that $\pi(M) \cong \pi(N)$ in $\rqgr A$ if and only if there is $n \geq 0$ such that $M_{\geq n} \cong N_{\geq n}$ in $\rgr A$, 
in other words two modules are isomorphic in the quotient category if and only if they have isomorphic tails.

\bigskip

3.  Prove that the quantum polynomial ring of Example~\ref{ex:qpoly2} is isomorphic to a Zhang twist of $R = k[x, y, z]$ if and only if $pqr = 1$.  (Hint: what are the degree $1$ normal elements of $R^{\sigma}$ for a given graded automorphism $\sigma$?)

\bigskip

4.  Complete the proof of Proposition~\ref{prop:goback}, in the following steps.  
\begin{enumerate}
\item[(a)] Show that if $\uExt^j_A(k, A)$ is finite dimensional, then $\lim_{i \to \infty} \uExt^j(A/A_{\geq i}, A)_{\geq 0}$ is also 
finite-dimensional.   Show that if $\uExt^j(k, A) = 0$, then $\lim_{i \to \infty} \uExt^j(A/A_{\geq i}, A)_{\geq 0} = 0$.
(Hint:  consider the long exact sequence in Ext associated to $0 \to K \to A/A_{\geq i+1} \to A/A_{\geq i} \to 0$, where $K$ is isomorphic to a direct sum of copies of $k_A(-i)$.)

\item[(b)] If $A$ is commutative, show that $\uExt^j_A(k, A)$ is finite-dimensional for all $j \geq 0$.  (Hint:  calculate $\Ext$ in two ways: with a projective resolution of $k_A$, and with an injective resolution of $A_A$.)
\end{enumerate}

\bigskip

5.  Assume that $\cha k = 0$.  
Let $B = k \langle x, y \rangle/(yx-xy -x^2)$ be the Jordan plane.  Let $A = k + By$, which is a graded subring of $B$.
It is known that the ring $A$ is noetherian (see \cite[Theorem 2.3]{StZh} for a proof).
\begin{enumerate}
\item[(a)] Show that $A$ is the \emph{idealizer} of the left ideal $By$ of $B$, that is, $A = \{z \in B | By z \subseteq By \}$.

\item[(b)]  Show that as a graded right $A$-module, $(B/A)_A$ is isomorphic to the infinite direct sum $\bigoplus_{n \geq 1} k_A(-n)$.

\item[(c)] 
Show for each $n \geq 1$ that the natural map $\uHom_A(A, A) \to \uHom_A(A_{\geq 1}, A)$ is not surjective in degree $n$, by 
finding an element in the latter group corresponding to $x \in B_n \setminus A_n$.  Conclude that the homomorphism $A \to B(\rqgr A, \pi(A), (1))$  constructed in Proposition~\ref{prop:goback} has infinite-dimensional cokernel, and that $\uExt^1_A(k, A
)$ is infinite-dimensional.
\end{enumerate}

\bigskip

\section{Lecture 5:  Classification of noncommutative curves and surfaces}

Many of the most classical results in algebraic geometry focus on the study of curves and surfaces, for 
example as described in \cite[Chapters IV, V]{Ha}.  Naturally, one would like to develop a comparatively rich 
theory of noncommutative curves and surfaces, especially their classification.   The canonical reference on noncommutative curves and surfaces is the survey article by Stafford and Van den Bergh \cite{StV}, which describes the state of the subject as of 2001.  While there are strong classification results for noncommutative curves, the classification of noncommutative surfaces is very much a work in progress.  In this lecture we describe some of the theory of noncommutative curves and surfaces, including some more recent work not described in \cite{StV}, especially the  special case of \emph{birationally commutative} surfaces.  We then close with a brief overview of some other recent themes in noncommutative projective geometry.  By its nature this lecture is more of a survey, so we will be able to give fewer details, and do not include exercises.

\subsection{Classification of  noncommutative projective curves}

While there is no single definition of what a noncommutative curve or surface should be, one obvious approach to the projective case is to take \fig\ noetherian algebras $A$ with $\GK(A) = d+1$, 
and consider the corresponding noncommutative projective schemes $\rqgr A$ as the $d$-dimensional 
noncommutative projective schemes.  Thus curves correspond to algebras of GK-dimension 2 and surfaces to algebras of GK-dimension 3.  In much of the preceding lectures we have concentrated on 
domains $A$ only, in which case one can think of $\rqgr A$ as being an analog of an integral projective scheme, or 
a variety.  We continue to focus on domains in this lecture.  

To study noncommutative projective (integral) curves, we consider 
\fig\ domains $A$ with $\GK(A) = 2$.  Artin and Stafford proved very strong results about the structure of these, 
as we will see in the next theorem.  First, we need to review a few more definitions.  Given a $k$-algebra $R$ with automorphism 
$\sigma: R \to R$, the skew-Laurent ring $R[t, t^{-1}; \sigma]$ is a $k$-algebra whose elements are Laurent polynomials 
$\sum_{i=a}^b r_i t^i$ with $a \leq b$ and $r_i \in R$, and with the unique associative multiplication rule determined by $ta = \sigma(a) t$ for all $a \in R$ (see \cite[Chapter 1]{GW}).
Assume now that $A$ is a \fig\ domain with $A \neq k$ and $\GK(A) < \infty$.  In this case, one can localize $A$ at the set of nonzero homogeneous elements in $A$, obtaining its \emph{graded quotient ring} $Q = Q_{\operatorname{gr}}(A)$.   Since every homogeneous element of $Q$ is a unit, $Q_0 = D$ is a division ring.  Moreover,  if $d \geq 1$ is minimal such that $Q_d \neq 0$, then  choosing any $0 \neq t \in Q_d$ the elements in $Q$ are Laurent polynomials of the form $\sum_{i = m}^n a_i t^i$ with $a_i \in D$.  If $\sigma: D \to D$ is the automorphism given by conjugation by $t$, that is $a \mapsto ta t^{-1}$, then one easily sees that $Q \cong D[t, t^{-1}; \sigma]$.  For example, if $A$ is commutative and generated in degree 1, 
then $Q_{\operatorname{gr}}(A) = k(X)[t, t^{-1}]$, where $k(X)$ must be the field of rational functions of $X = \cproj A$, 
since it is the field of fractions of each $A_{(x)}$ with $x \in A_1$. 
Thus, for a general \fig\ $A$ the division ring $D$ may be thought of as a noncommutative analogue of a rational function field.  
%(In fact, $D$ is intrinsic to the noncommutative projective scheme $\rQgr A$, since it is the endomorphism ring of the %injective hull of $\pi(A)$.)
\begin{comment}
For any \fig\ domain $A$, we define its \emph{$d$th Veronese ring} by $A^{(d)} = \bigoplus_{n \geq 0} A_{nd}$.  It is well-known that if $A$ is commutative \fig\ and generated in degree $1$, then $\cproj A \cong \cproj A^{(d)}$ for any $d \geq 1$. 
 Thus one does not lose any geometric information by passing to a Veronese ring.  A similar result holds for 
the noncommutative projective schemes  associated to \fig\ domains generated in degree $1$ \cite[]{AZ1}.
\end{comment}

The following theorem of Artin and Stafford shows that noncommutative projective integral curves are just
commutative curves.  Recall that our standing convention is that $k$ is algebraically closed.
\begin{theorem} \cite[Theorems 0.1, 0.2]{AS1}
\label{thm:AS}
Let $A$ be a \fig\ domain with $\GK A = 2$. 
\begin{enumerate}
\item The graded quotient ring $Q_{\operatorname{gr}}(A)$ of $A$ is isomorphic to $K[t, t^{-1}; \sigma]$ for some field $K$ with $\trdeg(K/k) = 1$ and some automorphism $\sigma: K \to K$.
\item If $A$ is generated in degree 1, then there is an injective map $\phi: A \to B(X, \mc{L}, \sigma)$ for some integral projective curve $X$ with function field $k(X) = K$, some ample invertible sheaf $\mc{L}$ on $X$, and the automorphism $\sigma: X \to X$ 
corresponding to $\sigma: K \to K$.  Moreover, $\phi$ is an isomorphism in all large degrees; in particular, 
$\rqgr A \sim \coh X$. 
\end{enumerate}
\end{theorem}
\noindent 
Artin and Stafford also gave a detailed description of those algebras $A$ satisfying the hypotheses 
of the theorem except generation in degree $1$ \cite[Theorem 0.4, 0.5]{AS}.   A typical example of this type 
with $\sigma$ of infinite order is the idealizer ring studied in Exercise 4.5.  In a follow-up paper \cite{AS2}, Artin and Stafford classified semiprime graded algebras of GK-dimension $2$; these rings are described in terms of a generalization of a twisted homogeneous coordinate ring involving a sheaf of orders on a projective curve.  

%The results above can be interpreted to say that noncommutative projective curves, that is categories $\rqgr A$ for graded %domains of GK-dimension $2$, are generally just commutative projective curves, that is categories of coherent sheaves on %projective curves.   Or alternatively, we can say that domains of GK-dimension $2$ are just 
%noncommutative coordinate rings of commutative projective curves.   We will see that the situation is certainly different for %surfaces.

Another approach to the theory of noncommutative curves is to classify categories which have all of the properties a category of coherent sheaves on a nice curve has.  Reiten and Van den Bergh proved in \cite{RV} that any connected noetherian Ext-finite hereditary Abelian category satisfying Serre duality over $k$ is either the category of coherent sheaves on a sheaf of hereditary $\mc{O}_X$-orders, where $X$ is a smooth curve, or else one of a short list of exceptional examples.   We refer 
to \cite[Section 7]{StV} for the detailed statement, and the definitions of some of the properties involved.  Intuitively, 
the categories this theorem classifies are somewhat different than those arising from graded rings, since the hypotheses 
demand properties which are analogs of properness and smoothness of the noncommutative curve, rather than projectivity.

There are many other categories studied in the literature which should arguably be thought of as examples of 
noncommutative quasi-projective curves.    For example, the category of $\mb{Z}$-graded modules over the Weyl algebra $A = k \langle x, y \rangle/(yx-xy-1)$, where $A$ is $\mb{Z}$-graded with $\deg x = 1, \deg y = -1$, can also be described as the quasi-coherent sheaves on a certain stack of dimension $1$.  See \cite{Sm3} for this geometric description, and \cite{Si1} for more details about the structure of this category.
Some other important examples are the \emph{weighted projective lines} studied by Lenzing and others (see \cite{Len} for a survey).   As of yet, there is not an overarching theory of noncommutative curves which encompasses all of the different kinds of examples mentioned above.

\subsection{The minimal model program for surfaces and Artin's conjecture}

Before discussing noncommutative surfaces, we first recall the main idea of the classification of commutative surfaces.
Recall that two integral surfaces are \emph{birational} if they have isomorphic fields of rational functions, or equivalently, if they have isomorphic open subsets \cite[Section I.4]{Ha}.  The coarse classification of projective surfaces 
divides them into birational equivalence classes.   There are various numerical invariants for projective surfaces which are constant among all surfaces in a birational class.  The most important of these is the Kodaira dimension; some others include the 
arithmetic and geometric genus \cite[Section V.6]{Ha}.  There is a good understanding of the possible birational equivalence 
classes in terms of such numerical invariants.

For the finer classification of projective surfaces, one seeks to understand the smooth projective surfaces within 
a particular birational class.  A fundamental theorem states that any surface in the class can be obtained from any other by a sequence of monoidal transformations, that is, blowups at points or the reverse process, blowdowns of exceptional curves  \cite[Theorem V.5.5]{Ha}.  This gives a specific way to relate surfaces within a class, and some important 
properties of a surface, such as the Picard group, change in a simple way under a monoidal transformation \cite[Proposition V.3.2]{Ha}.  Every birational class has at least one \emph{minimal model}, a smooth 
surface which has no exceptional curves, and every surface in the class is obtained from some sequence of blowups starting with some minimal model.   In fact, most birational classes have a unique minimal model, with the exception of the classes of rational and ruled 
surfaces \cite[Remark V.5.8.4]{Ha}.  For example, the birational class of rational surfaces---those with function field $k(x, y)$---has as minimal models $\mb{P}^2$,  $\mb{P}^1 \times \mb{P}^1$, and the other 
Hirzebruch surfaces  \cite[Example V.5.8.2]{Ha}.

An important goal in noncommutative projective geometry is to find a classification of noncommutative surfaces, 
modeled after the classification of commutative surfaces described above.  
It is easy to find an analog of birationality: as we have already mentioned, for a \fig\ domain $A$ 
its graded quotient ring has the form $Q_{\operatorname{gr}}(A) \cong D[t, t^{-1}; \sigma]$, where the division ring $D$ plays the role of a field of rational functions.  Thus the birational classification of noncommutative surfaces requires 
the analysis of which division rings $D$ occur as $Q_{\rm gr}(A)_0$ for some \fig\ domain $A$ of GK-dimension 3.  In \cite{Ar}, Artin gave a list of known families of such division rings $D$ and conjectured that these are all the possible ones, with a deformation-theoretic heuristic argument as supporting evidence.  Artin's conjecture is still open and remains one of the important  but elusive goals of noncommutative projective geometry.  See \cite[Section 10.1]{StV} for more details.  In this lecture, we will focus on the other part of the classification problem, namely, understanding how surfaces within a birational class are related.

\subsection{Birationally commutative surfaces}

Let $A$ be a \fig\ domain of finite GK-dimension, with graded ring of fractions $Q_{\rm gr}(A)  \cong D[t, t^{-1}; \sigma]$.  When $D = K$ is a field, we say that $A$ is \emph{birationally commutative}.   By the Artin-Stafford theorem (Theorem~\ref{thm:AS}), this holds automatically when $\GK(A) = 2$, that is, for noncommutative projective curves.  Of course, it is not automatic for 
surfaces.  For example, the reader may check that the 
quantum polynomial ring $A = k \langle x, y, z \rangle/(yx - p xy, zy - r yz, xz - qz x)$ from Example~\ref{ex:qpolypm}
is  birationally commutative if and only if $pqr = 1$.  As we saw in that example, this is the same condition that implies 
that the scheme parametrizing the point modules is $\mb{P}^2$, rather than three lines.

Let $A$ be some \fig\ domain of GK-dimension $3$ with $Q_{\operatorname{gr}}(A) \cong D[t, t^{-1}; \sigma]$.  The problem of classification within this birational class is to understand and relate the possible \fig\ algebras $A$ which have a graded 
quotient ring $Q$ with $Q_0 \cong D$, and their associated noncommutative projective schemes $\rqgr A$.  
For simplicity, one may focus on one slice of this problem at a time and consider only those $A$ 
with $Q_{\rm gr}(A) \cong D[t, t^{-1}; \sigma]$ for the fixed automorphism $\sigma$.

We have now seen several examples where a \fig\ algebra $A$ has a homomorphism to a twisted homogeneous coordinate ring $B(X, \mc{L}, \sigma)$:   in the classification of noncommutative projective curves (Theorem~\ref{thm:AS}), and in the sketch of the proof of the classification of AS-regular algebras of dimension 3 at the end of Lecture 3.  In fact, this is a quite general phenomenon.
\begin{theorem} \cite[Theorem 1.1]{RZ1}
\label{thm:RZ1}
Let $A$ be a \fig\ algebra which is strongly noetherian and generated in degree $1$.  By Theorem~\ref{thm:AZ2}, the maps $\phi_m$ from Proposition~\ref{prop:param} are isomorphisms for $m \geq m_0$, so that the point modules of $A$ are parametrized by the projective scheme $X = X_{m_0}$.   
Then we can canonically associate to $A$ an invertible sheaf $\mc{L}$ on $X$, an automorphism $\sigma: X \to X$, and a ring homomorphism $\phi: A \to B(X, \mc{L}, \sigma)$ which is surjective in all large degrees.  
The kernel of $\phi$ is the ideal of elements that kill all $R$-point modules of $A$, for all commutative $k$-algebras $R$.
\end{theorem}
\noindent The $\sigma$ and $\mc{L}$ in the theorem above arise naturally from the data of the point modules; for example, $\sigma$ is induced by the truncation shift map on point modules which sends a point module $P$ to $P(1)_{\geq 0}$.  
Theorem~\ref{thm:RZ1} can also be interpreted in the following way:  a strongly noetherian \fig\ algebra, generated in degree 1, 
has a unique largest factor ring determined by the point modules, and this factor ring is essentially (up to a finite dimensional vector space) a twisted homogeneous coordinate ring.   Roughly, one may also think of $X$ as the largest commutative subscheme of the noncommutative projective scheme $\rqgr A$.  

While the canonical map to a twisted homogeneous coordinate ring was the main tool in the classification of AS-regular algebras of dimension $3$, it is less powerful as a technique for understanding AS-regular algebras of global dimension $4$ and higher.  Such algebras can have small point schemes, in which case the kernel of the canonical map is too large for the map to give much interesting information.  For example, there are many examples of regular algebras of global dimension $4$ whose point scheme is $0$-dimensional, and presumably the point scheme of an AS-regular algebra of dimension 5 or higher might be empty.
There is a special class of algebras, however, which are guaranteed to have a rich supply of point modules, and for which the canonical map leads to a strong structure result.
\begin{corollary} \cite[Theorem 1.2]{RZ1}
\label{cor:RZ}
Let $A$ satisfy the hypotheses of Theorem~\ref{thm:RZ1}, and assume in addition that $A$ is a domain which is birationally 
commutative with $Q_{gr}(A) \cong K[t, t^{-1}; \sigma]$.  Then the canonical map $\phi: A \to B(X, \mc{L}, \sigma)$ described 
by Theorem~\ref{thm:RZ1} is injective, and thus is an isomorphism in all large degrees.
\end{corollary}
\noindent  
The main idea behind this corollary is that the positive part of the quotient ring itself, $K[t; \sigma]$, is a $K$-point module for $A$.
It obviously has annihilator $0$ in $A$, so no nonzero elements annihilate all point modules (over all base rings).

Corollary~\ref{cor:RZ} completely classifies birationally commutative algebras, generated in degree 1, which happen to be known in advance to be strongly noetherian.  However, there are birationally commutative  algebras of GK-dimension $3$ which are noetherian but not strongly noetherian, and so a general theory of birationally commutative surfaces needs to account for these.
\begin{example}
Let $X$ be a projective surface with automorphism $\sigma: X \to X$, and let $\mc{L}$ be an ample invertible sheaf on $X$.
Choose an ideal sheaf $\mc{I}$ defining a $0$-dimensional subscheme $Z$ of $X$. 
% (The most important special case is where $\mc{I}$ defines a single reduced point $p \in X$.)  
For each $n \geq 0$, set $\mc{I}_n = \mc{I} \mc{I}^{\sigma} \dots \mc{I}^{\sigma^{n-1}}$, and let 
$\mc{L}_n= \mc{L} \otimes \mc{L}^{\sigma} \otimes \dots  \otimes \mc{L}^{\sigma^{n-1}}$ as in Definition~\ref{def:thcr}.  Now we define the \emph{na{\"i}ve blowup algebra}
\[
R = R(X, \mc{L}, \sigma, Z) = \bigoplus_{n \geq 0}H^0(X, \mc{I}_n \otimes \mc{L}_n) \subseteq B = B(X, \mc{L}, \sigma) = \bigoplus_{n \geq 0} H^0(X, \mc{L}_n),
\]
so that $R$ is a subring of $B$.  The ring $R$ is known to be noetherian but not strongly noetherian when every point $p$ in the support of $Z$ lies on a \emph{critically dense} orbit, that is, when every infinite subset of $\{ \sigma^i(p) | i \in \mb{Z} \}$ has closure 
in the Zariski topology equal to all of $X$ \cite[Theorem 1.1]{RS2}.
\end{example}

A very explicit example of a na{\"i}ve blowup algebra is the following.
\begin{example}
\label{ex:simple}
Let $R = R(\mb{P}^2, \mc{O}(1), \sigma, Z)$, where $Z$ is the single reduced point $(1:1:1)$, and 
$\sigma(a:b:c) = (qa:rb:c)$
 for some $q, r \in k$, where $\cha k = 0$.   As long as $q$ and $r$ are algebraically independent over $\mb{Q}$, the $\sigma$-orbit of $(1:1:1)$ is critically dense and $R$ is a noetherian ring \cite[Theorem 12.3]{Ro}.  In this case one has 
\[
B(\mb{P}^2, \mc{O}(1), \sigma) \cong k \langle x, y, z \rangle/(yx - qr^{-1}xy, zy-ryz, xz - q^{-1}zx), 
\]
and $R$ is equal to the subalgebra of $B$ generated by $x-z$ and $y-z$.   
\end{example}
\noindent Historically, Example~\ref{ex:simple} was first studied by D. Jordan as a ring generated by Eulerian derivatives \cite{Jo}. %Jordan also proved by direct calculation that its point modules are not parametrized by a projective scheme.  
Later, these specific examples were shown to be noetherian in most cases \cite{Ro}, and last the more general notion of na{\"i}ve blowup put such examples in a more general geometric context \cite{KRS}.  The name na{\"i}ve blowup reflects the fact that the definition of such rings is a kind of twisted version of the Rees ring construction which is used to define a blowup in the commutative case \cite[Section II.7]{Ha}.  The relationship between the noncommutative projective schemes $\rqgr R(X, \mc{L}, \sigma, Z)$ and $\rqgr B(X, \mc{L}, \sigma) \simeq X$ is more obscure, however, and does not have the usual geometric properties one expects of a blowup \cite[Section 5]{KRS}.

\begin{comment}
Naive blowup algebras, such as the example above, are the simplest examples of algebras which are noetherian 
but not strongly noetherian, and which have point modules not parametrized by a projective scheme.  Instead, in the ring $R$ above for example, the point modules are parametrized essentially by an inverse limit of blowups of $\mb{P}^2$ at points on the orbit of $p$.  Oddly, in $\rqgr R$ the exceptional divisors in this picture become contracted, and the images of the point modules 
are now in one-to-one correspondence with points of $\mb{P}^2$.  However, the points on the $\sigma$-orbit of $p$ have become marked in some way, and the simple objects are not parametrized by $\mb{P}^2$.  Recently, Sierra and Nevins 
have shown, however, that $\mb{P}^2$ is at least a coarse moduli space for the simple objects in $\rqgr R$.
\end{comment}

The examples we have already seen are typical of birationally commutative surfaces, as the following result shows.
\begin{theorem} \cite{RS1}
\label{thm:RS1}
Let $A$ be a noetherian \fig\ $k$-algebra, generated in degree 1, with $Q_{gr}(A) \cong K[t, t^{-1}; \sigma]$, 
where $K$ is a field of transcendence degree $2$ over $k$.  Assume in addition that there exists a projective surface $Y$ with function field $K = k(Y)$, and an automorphism $\sigma: Y \to Y$ which induces  the automorphism $\sigma: K \to K$.
Then $A$ is isomorphic to a na{\"i}ve blowup algebra $A \cong R(X, \mc{L}, \sigma, Z)$,   where $X$ is a surface with $k(X) = K$, 
$\sigma: X \to X$ is an automorphism corresponding to $\sigma: K \to K$, the sheaf $\mc{L}$ is $\sigma$-ample, and 
every point of $Z$ lies on a critically dense $\sigma$-orbit.
\end{theorem}
\noindent 
Sierra has extended this theorem to the case of algebras not necessarily generated in degree $1$ \cite{Si2}.  In this setting, 
one gets a more general class of possible examples, which are a bit more technical to describe, but all of the examples are still defined in terms of sheaves on a surface $X$ with automorphism $\sigma$.  The condition in Theorem~\ref{thm:RS1} that there exists an automorphism of a projective surface $Y$ inducing $\sigma$, which is also a hypothesis in Sierra's generalization, may seem a bit mysterious.   It turns out that there are automorphisms of fields of transcendence degree $2$ which do not correspond to an automorphism of any projective surface with that field as its fraction field \cite{DF}.    A noetherian ring $A$ with a graded quotient ring $Q_{\operatorname{gr}}(A) = K[t, t^{-1}; \sigma]$, where $\sigma$ is a field automorphism of this strange type, was constructed in \cite{RSi}.   

The rings classified in Theorem~\ref{thm:RS1} are either twisted homogeneous coordinate rings $B(X, \mc{L}, \sigma)$ (these are the cases where $Z = \emptyset$ and are the only strongly noetherian ones),  or na{\"i}ve blowup algebras inside of these.  The possible smooth surfaces $X$ occurring are all birational and so are related to each other via monoidal transformations, as described above.  Each $B = B(X, \mc{L}, \sigma)$ has 
$\rqgr B \simeq \coh X$, while for each $R = R(X, \mc{L}, \sigma, Z)$, the category $\rqgr R$ can be thought of as a na{\"i}ve blowup of $\coh X$.  Thus within this birational class, it is true that the possible examples are all related by some kind of generalized blowup or blowdown procedures,  in accordance with the intuition coming from the classification of commutative projective surfaces.

\subsection{Brief overview of other topics}

In this lecture we have described one particular thread of recent research in noncommutative projective geometry, one with which we are intimately familiar.  In this final section, we give some very brief summaries of a few other themes of current research.   This list is not meant to be comprehensive.

\subsubsection{Noncommutative projective surfaces}

The reader can find a survey of some important work on noncommutative projective surfaces in the second half of \cite{StV}.     For example, there is a rich theory of noncommutative quadric surfaces, that is, noncommutative analogs of subschemes of $\mb{P}^3$ defined by a degree $2$ polynomial; see for example \cite{VdB2} and \cite{SmV}.    We should also mention Van den Bergh's theory of noncommutative blowing up, which allows one to blow up a point lying on a commutative curve contained in a noncommutative surface, under certain circumstances \cite{VdB}.  These blowups do have properties analogous to commutative blowups.  (Van den Bergh's blowups and na{\"i}ve blowups generally apply in completely different settings.)    The author, Sierra, and Stafford have begun to study classification of surfaces within the birational class of the generic Sklyanin algebra \cite{RSS}, where the various examples are expected to be related via blowups of Van den Bergh's kind.  

In a different direction, there is a deep theory of maximal orders on commutative surfaces, which are certain sheaves of algebras on the surface which are locally finite over their centers.   Chan and Ingalls \cite{CI} laid the foundations of a minimal model program for the classification of such orders, which has been studied for many special types of orders.   The reader may find an introduction to the theory in \cite{Ch}.

\subsubsection{Regular algebras of dimension 4}

Since the classification of AS-regular algebras of dimension 3 was acheived, much attention has been focused on 
regular algebras of dimension 4.   The Sklyanin algebra of dimension 4--which also has point modules paramaterized by an elliptic curve, like its analog in dimension 3--was one of the first regular algebras of dimension 4 to be intensively studied, see \cite{SS} for example.  There has been much interest in the problem of classification of AS-regular algebras of dimension 4, which have three possible Hilbert series \cite[Proposition 1.4]{LPWZ}.    Many interesting examples of 4-dimensional regular algebras have been given with point schemes of various kinds.  In addition, a number of important new constructions which produce regular algebras have been invented, for example the double Ore extensions due to Zhang and Zhang \cite{ZZ1}, and the skew graded Clifford algebras due to Cassidy and Vancliff \cite{CV}.  Another interesting technique developed in \cite{LPWZ} is the study of the $A_{\infty}$-algebra structure on the Ext algebra of a AS-regular algebra.  Some combination of these techniques has been used to successfully classify AS-regular algebras of global dimension 4, generated in degree $1$, with a nontrivial grading  by $\mb{N} \times \mb{N}$ \cite{LPWZ}, \cite{ZZ2}, \cite{RZ2}.  The general classification problem is an active topic of research.

%See \cite{CN} for some interesting recent work on noncommutative ruled surfaces inspired by  Artin's  conjecture.   

%Smith has also developed some of the categorical machinery that would allow one to study various formal properties of %noncommutative schemes, such as open and closed subschemes, integrality, and so on \cite{Sm1}, \cite{Sm2}.

\subsubsection{Calabi-Yau algebras}

The notion of Calabi-Yau algebra, which was originally defined by Ginzburg \cite{Gi}, appeared in several of the lecture courses at the MSRI workshop.  There is also a slightly more general notion called a twisted or skew Calabi-Yau algebra.  
In the particular setting of \fig\ algebras, the definition of twisted Calabi-Yau algebra is actually equivalent to the definition of AS-regular algebra \cite[Lemma 1.2]{RRZ}.  There has been much interesting work on (twisted) Calabi-Yau algebras which arise as factor algebras of path algebras of quivers, especially the study of when the relations of such algebras come from superpotentials.   One may think of this theory as a generalization of the theory of AS-regular algebras to the non-connected graded case.
The literature in this subject has grown quickly in recent years; we mention \cite{BSW} as one representative paper, which 
includes a study of the 4-dimensional Sklyanin algebras from the Calabi-Yau algebra point of view.

\subsubsection{Noncommutative Invariant Theory}

The study of the rings of invariants of finite groups acting on commutative polynomial rings is now classical.  
It is natural to generalize this to study the invariant rings of group actions on AS-regular algebras, the noncommutative analogs of polynomial rings.   A further generalization allows a finite-dimensional Hopf algebra to act on the ring instead of a group.  A number of classical theorems concerning when the ring of invariants is Gorenstein, regular, and so on, have been generalized to this context.   Two recent papers from which the reader can get an idea of the theory are \cite{KKZ} and \cite{CKWZ}.

\subsubsection{The universal enveloping algebra of the Witt algebra}

Since these lectures were originally delivered, Sierra and Walton discovered a stunning new application of point modules 
which settled a long standing open question in the theory of enveloping algebras  \cite{SW}.  We close with a brief description of their work.
Let $k$ have characteristic $0$ and consider the infinite-dimensional Lie algebra $L$ with $k$-basis $\{ x_i | i \geq 1 \}$ and bracket $[x_i, x_j] = (j-i) x_{i+j}$, which is known as the positive part of the Witt algebra.  Since $L$ is a graded Lie algebra, its universal enveloping algebra $A = U(L)$ is connected $\mb{N}$-graded, and has the Hilbert series of a polynomial ring in variables of weights $1, 2, 3, \dots $, by the PBW theorem.
Thus the function $f(n) = \dim_k A_n$ is actually the partition function, and from this one may see that $A$ has infinite GK-dimension but subexponential growth.   The question of whether $A$ is noetherian arose in the work of Dean and Small \cite{DS}.  Stephenson and Zhang proved that \fig\ algebras of exponential growth cannot be noetherian \cite{SZ1}, so it is an obvious question whether a \fig\ algebra of subexponential but greater than polynomial growth could possibly be noetherian.
The ring $A$ was an obvious test case for this question.  
%If it were noetherian it would provide a counterexample to numerous 
%other open questions about rings, such as whether a noetherian ring must have DCC on prime ideals.  
It is very difficult to do explicit calculations in $A$, but Sierra and Walton found a factor ring $A/I$ which is birationally commutative and can be described in terms of sheaves on a certain surface.  Using geometry, they proved that $A/I$ is non-noetherian, so that $A$ is non-noetherian also.  The ideal $I$ does not have obvious generators (in fact, generators of $I$ are not found in \cite{SW}), and it is unlikely that $I$ would have been discovered without using point modules, or 
that the factor ring $A/I$ could have been successfully analyzed without using geometric techniques. 

\section{Solutions to exercises}

\subsection{Solutions to Exercise Set 1}

\

1(a).  Write $A = k[x_1, \dots, x_n] = B[x_n]$ where $B = k[x_1, \dots, x_{n-1}]$.  Since 
$A = B \oplus Bx_n \oplus Bx_n^2 \oplus \dots$ we have $h_A(t) = h_B(t)[1 + t^{d_n} + t^{2d_n} + \dots] 
= h_B(t)/(1 - t^{d_n})$, and we are done by induction on $n$.

(b).  Since $F_m$ is spanned by words of degree $m$, considering the last letter in a word we have the direct sum decomposition 
$F_m = \oplus_{i=1}^n F_{m-d_i} x_i$.  Writing $h_{F}(t) = \sum a_i t^i$ we get $a_m = \sum_{i=1}^n a_{m-d_i}$ (where 
$a_i = 0$ for $i < 0$).  This leads to the Hilbert series equation $h_F(t) = \sum_{i=1}^n h_F(t) t^{d_i}$; now solve for $h_F(t)$.

\medskip

2.  Suppose that $x$ is a nonzerodivisor and a normal homogeneous element in $A$, such that $A/xA$ is a domain.  If $A$ is not 
a domain, it must have homogeneous nonzero elements $y, z$ with $yz = 0$.  Choose such $0 \neq y \in A_m$, $0 \neq z \in A_n$ 
with $m + n$ minimal.   Then $\overline{y} \overline{z}=0$ in $A/xA$, where 
$\overline{y} = y + xA$ is the image of $y$ in the factor ring.  So $\overline{y} = 0$ or $\overline{z} = 0$; without loss of generality, assume the former.    Then $y \in xA$, so $y = xy'$ with $0 \neq y'$ of smaller degree than $y$.  Then $xy'z = 0$, and $x$ is a nonzerodivisor so $y' z = 0$, contradicting minimality.  

In the example $A = k \langle x, y \rangle/(yx^2 - x^2y, y^2x - xy^2)$, the element $xy - yx$ is normal since $x(xy - yx) = -(xy-yx)x$ and similarly for $y$.  Moreover, clearly $A/(xy-yx) \cong k[x, y]$ is a noetherian domain.  Since the Hilbert series of $A$ is known to be $1/(1-t)^2(1-t^2)$, a similar argument as in the proof of Lemma~\ref{cor:nd} shows that $(xy-yx)$ must be a nonzerodivisor, so Lemma~\ref{lem:passup} applies.  

\medskip

3.   Let $v$ and $w$ be reduction unique elements.  Let $s$ be any composition of reductions such that $s(v + \lambda w)$ 
is a linear combination of reduced words, where $\lambda \in k$.   Let $t$ be some composition of reductions 
such that $ts(v)$ is a linear combination of reduced words, and let $u$ be some composition of reductions such that $uts(w)$ is a linear combination of reduced words.  Then $s(v + \lambda w) = uts(v + \lambda w)$ (since $s(v + \lambda w)$ is already a linear combination of reduced words) and $uts(v + \lambda w) = uts(v) + \lambda uts(w) = \red(v) + \lambda \red(w)$.  Thus $s(v + \lambda w)$ is independent of $s$, and the set of reduction unique elements is a subspace.  Also, $s( v + \lambda w) = \red( v + \lambda w) = \red(v) + \lambda \red(w)$ and so $\red(-)$ is linear on this subspace.

For $(3) \implies (2)$, the hypothesis is equivalent to $F = I \bigoplus V$ as $k$-spaces, where $V$ is the $k$-span of the set of reduced words.  Given any element $h$, let $s$ and $t$ be compositions of reductions such that $s(h)$ and $t(h)$ are both  linear combinations of reduced words.  Then $s(h) - t(h) \in I$ (since a reduction changes an element to one equivalent modulo $I$) and $s(h) - t(h) \in V$.   So $s(h) - t(h) \in I \cap V = 0$.

For $(2) \implies (1)$, if a word $w$ contains a non-resolving ambiguity, then there are reductions $r_1,r_2$ such that $r_1(w)$ and $r_2(w)$ cannot be made equal by performing further reductions to each.  In particular, if $s$ and $t$ are compositions of reductions such that $s r_1(w)$ and $t r_2(w)$ are both linear combinations of reduced words, they are not equal, so $w$ is not reduction unique.

\medskip

4.  The word $z^2x$ may first be resolved using $g_1$ to give $(xy+yx)x = xyx + yx^2$, which is a linear combination of reduced words, 
or using $g_2$ to give $z(xz)$, which may be further reduced giving $xz^2$ and then $x(xy+ yx) = x^2y + xyx$.  These are distinct 
so their difference is added as the new relation $g_4 = yx^2 - x^2y$.  Similarly, $g_5 = y^2x - xy^2$ is added from resolving $z^2y$ two ways.  It is straightforward to check that all ambiguities now resolve.

The basis of reduced words is all words not containing any of $z^2, zx, zy, y^2x, yx^2$, which is 
\[
\{ x^{i_1} (yx)^{i_2} y^{i_3} z^{\epsilon} | i_1, i_2, i_3 \geq 0, \epsilon \in\{0,1\} \}.
\]   
The set of such words with $\epsilon = 0$ 
has the same Hilbert series as a polynomial ring in variables of weights $1, 1, 2$, namely $1/(1-t)^2(1-t^2)$ using Exercise 1.1.  Thus $h_A(t) = (1 + t)/(1-t)^2(1-t^2) = 1/(1-t)^3$.

\begin{comment}
5.  There is one overlap ambiguity $x_k x_j x_i$ for each triple $1 \leq i < j < k \leq n$.  Write the Lie bracket with the 
general formula $[x_i, x_j] = \sum a_{ijk} x_k$.  Then 
Resolving $x_k x_j$ first gives 
\[
(x_jx_k + [x_j, x_k])x_i = x_jx_kx_i + \sum a_{jkl} x_l x_i
x_j (x_ix_k + [x_i, x_k]) + [x_j, x_k]x_i = (x_ix_j + [x_i, x_j])x_k + x_j[x_i, x_k] + [x_j, x_k]x_i = 
x_ix_jx_k + [x_i, x_j]x_k + x_j[x_i, x_k] + [x_j, x_k]x_i.
\]
Similarly, resolving first $x_j x_i$ gives
\[
x_ix_jx_k + x_k [x_i, x_j] + [x_i, x_k]x_j + x_i [x_j, x_k].
\]
Then both of these can be reduced further to give linear combinations of reduced words.  Since the overlaps lie in degree 3, 
note that all linear combinations of words of degree at most 2 will be automatically reduction unique.  Thus to show we get 
the same result after reducing the expressions above it suffices to show that 
$\Omega = [x_i, x_j]x_k + x_j[x_i, x_k] + [x_j, x_k]x_i  - (x_k [x_i, x_j] + [x_i, x_k]x_j + x_i [x_j, x_k]) = 0$ 
as elements of $U(L)$.  Now note that the relation $x_j x_i - x_i x_j - [x_i, x_j] = 0$ is valid in $U(L)$ for any pair $i, j$ (not just for $i < j$), by the antisymmetry 
of the bracket.  Thus writing $[x_i, x_j] = \sum a_l x_l$ we get 
$x_i[x_j, x_k] - [x_j, x_k]x_i = \sum a_l x_i x_l - x_l x_i = \sum a_l [x_l, x_i] = [[x_j, x_k], x_i]$ in $U(L)$.  
Thus $\Omega = -[[x_j, x_k], x_i] + [[x_j, x_i], x_k] - [[x_k, x_i], x_j] = 0$ by the Jacobi identity.
\end{comment}

\medskip

5(a).  The usual isomorphism $\Hom_A(A, A) \cong A$ given by $\phi \mapsto \phi(1)$ is easily adjusted to prove that $\uHom_A(A(-s_i), A) \cong A(s_i)$ in the graded setting.  The indicated formula follows since finite direct sums pull out of either coordinate of $\uHom$.

(b).  If $e_1, \dots, e_m$ is the standard basis for the free right module $P$ and $f_1, \dots, f_n$ is the standard basis for the free right module $Q$, then the matrix $M = (m_{ij})$ is determined by $\phi(e_j) = \sum_i f_i m_{ij}$.  We can take 
as a basis for $\uHom(P, A)$ the dual basis $e_1^*, \dots, e_m^*$ such that $e_i^*(e_i) = 1$ and $e_i^*(e_j) = 0$ for $i \neq j$.
Identifying $\uHom(P, A)$ and $\bigoplus_{i=1}^m A(s_i)$ using part (a), then $\{ e_i^* \}$ is just the standard basis of $\bigoplus_{i=1}^m A(s_i)$.
Similarly, the basis $f_1^*, \dots, f_n^*$ of $\uHom(Q, A)$ is identified with the standard basis of $\bigoplus_{j= 1}^n A(t_j)$.

Now the matrix $N = (n_{ij})$ of $\phi^*$ should satisfy $\phi^*(f_i^*) = \sum_j n_{ij} e_j^* $ since these are left modules.  
By definition we have $\phi^*(f_i^*) = f_i^* \circ \phi$ and
$f_i^* \circ \phi(e_k) = f_i^*( \sum_l f_l m_{lk}) = \sum_l  f_i^*(f_l) m_{lk} = m_{ik}$.  
We also have $(\sum_j n_{ij} e_j^*)(e_k) = \sum_j n_{ij} (e_j^*(e_k)) = n_{ik}$, using the definition of the left $A$-module structure on $\uHom(P, A)$.  This proves that $M = N$, as required.

\medskip

6(a).  This is a variation of the argument in Lemma~\ref{lem:min}.

(b). Using the minimal graded free resolution $P_{\bullet}$ of $M$ to calculate 
Tor, we see that $\Tor_i^A(M_A, {}_A k)$ is the $i^{\small \text{th}}$ homology of the complex $P_{\bullet} \otimes_A k$, 
where $P_i \otimes_A k \cong P_i/P_i A_{\geq 1}$.  Since the resolution is minimal, the maps in this complex are $0$ by Lemma~\ref{lem:min}, and thus $\Tor_i^A(M, k) \cong P_i/P_i A_{\geq 1}$ is a $k$-space of dimension equal to the minimal number of homogeneous generators of $P_i$.  In particular, the minimal free resolution has length  $\max \{i | \tor_i(M_A, {}_A k)  \neq 0 \}$, so the statement follows from (a).

(c).   Clearly if $_A k$ has projective dimension $d$, using a projective resolution in the second coordinate to calculate Tor gives 
$\max \{i | \tor_i(M_A, {}_A k)  \neq 0 \} \leq d$.  Combined with part (a) we get the first statement.  Obviously we can prove all of the same results on the other side to obtain $\pdim({}_A N) \leq \pdim(k_A)$ for any left bounded graded left module $N$.

(d).   By part (b) we have both $\pdim(k_A) \leq \pdim({}_A k)$ and $\pdim({}_A k) \leq \pdim(k_A)$, so 
$\pdim({}_A k) = \pdim(k_A)$.  Finitely generated graded modules are left bounded, and we only need to consider finitely generated modules by the result of Auslander.  Any finitely generated graded right module $M$ satisfies $\pdim(M) \leq \pdim({} _A k)$ by parts (a) and (b), so $\rgl(A) = \pdim({} _A k)$.  Similarly $\lgl(A) = \pdim(k_A)$.

\subsection{Solutions to Exercise Set 2}

\

1.  The required isomorphism follows from an application of Lemma~\ref{lem:homtoA}(1) followed by its left-sided analog.
If $A$ is weakly AS-regular, that is if (1) and (3) of Definition~\ref{def:AS} hold, then if $P_{\bullet}$ is the minimal graded free resolution of $k_A$, we know that the $P_i$ have finite rank (\cite{SZ1}) and that 
$Q_{\bullet} = \uHom_A(P_{\bullet}, A)$ is a minimal graded free resolution of $_A k$.  Applying $\uHom( -, {}_A A)$ to 
$Q_{\bullet}$  yields a complex of free right modules isomorphic to the original $P_{\bullet}$, by the isomorphism above.  
This implies that $\uExt^i_A({}_A k, {}_A A)$ is isomorphic to $k_A(\ell)$ if  $i = d$, and is $0$ otherwise.  
\bigskip

2.   By the Diamond Lemma, $A$ has $k$-basis $\{x^i (yx)^j y^k | i, j ,k \geq 0 \}$ since 
the overlap $y^2x^2$ resolves.  Thus $h_A(t) = 1/(1-t)^2(1-t^2)$.  Now using Lemma~\ref{lem:begin} and guessing at the final map we write down the potential free resolution of $k_A$ as follows:
\[
0 \to A(-4) \overset{\begin{pmatrix} y \\ x  \end{pmatrix}}{\lra} A(-3)^{\oplus 2} \overset{\begin{pmatrix} xy-2yx & y^2 \\ x^2 & yx-2xy  \end{pmatrix}}{\lra} A(-1)^{\oplus 2} \overset{\begin{pmatrix} x & y \end{pmatrix}}{\lra} A \to 0.
\]
The proof that this is indeed a free resolution of $k_A$, and the verification of the AS-Gorenstein condition, is similar to the proof in Example~\ref{ex:reg3}.

\bigskip

3(a).  It is easy to see since the algebra has one relation that the existence of a graded isomorphism $A(\tau) \to A(\tau')$ is equivalent to the existence of a change 
of variable $x' = c_{11} x + c_{12} y, y' =  c_{21} x + c_{22} y$, such that $x' \tau'(x') + y' \tau'(y')$ and $x \tau(x) + y \tau(y)$ generate the same ideal of the free algebra, that is, they are nonzero scalar multiples.  By adjusting the change of variable by a scalar, this is equivalent to finding such a change of variable with $x' \tau'(x') + y' \tau'(y') = x \tau(x) + y \tau(y)$, since our base field is 
algebraically closed.   We have $\begin{pmatrix} \tau(x) \\ \tau(y) \end{pmatrix} = B \begin{pmatrix} x \\ y \end{pmatrix}$, 
$\begin{pmatrix} \tau'(x') \\ \tau'(y') \end{pmatrix} = B' \begin{pmatrix} x' \\ y' \end{pmatrix}$ for some matrix $B'$, 
and $\begin{pmatrix} x' \\ y' \end{pmatrix} = C \begin{pmatrix} x \\ y \end{pmatrix}$.  So 
$\begin{pmatrix} x' & y' \end{pmatrix} \begin{pmatrix} \tau'(x') \\ \tau'(y') \end{pmatrix} = 
\begin{pmatrix} x & y \end{pmatrix} C^t B' C \begin{pmatrix} x \\ y \end{pmatrix}$, while 
$\begin{pmatrix} x & y \end{pmatrix} \begin{pmatrix} \tau(x) \\ \tau(y) \end{pmatrix} = 
\begin{pmatrix} x & y \end{pmatrix} B \begin{pmatrix} x \\ y \end{pmatrix}$.
Thus $A(\tau) \cong A(\tau')$ if and only if $B = C^t B' C$ for some invertible matrix $C$, that is, if 
$B$ and $B'$ are congruent.  

(b).  This is a tedious but elementary computation.

(c).  Let $A =k \langle x, y \rangle/(x^2)$.  To calculate the Hilbert series of $A$,  note that the overlap of $x^2$ with 
itself trivially resolves, so the set of words not containing $x^2$ is a $k$-basis for $A$.  let $W_n$ be the set of words of 
degree $n$ not containing $x^2$, and note that $W_n = W_{n-2}yx \cup W_{n-1}y$.  Thus $\dim_k A_n = \dim_k A_{n-1} + \dim_k A_{n-2}$ for $n \geq 2$, while $\dim_k A_0 = 1$ and $\dim_k A_1 = 2$.  In terms of Hilbert series we 
have $h_A(t) = th_A(t) + t^2 h_A(t) +1 + t$ and thus $h_A(t) = (1+t)/(1 - t - t^2)$.  Since $\dim_k A_n$ is part of the Fibonacci sequence, these numbers grow exponentially and $A$ has exponential growth; in particular $\GK(A) = \infty$.
The minimal free resolution of $k_A$ is easily calculated to begin
\[
\dots \lra A(-4) \overset{\begin{pmatrix} x \end{pmatrix}}{\lra} A(-3) \overset{\begin{pmatrix} x \end{pmatrix}}{\lra} A(-2) \overset{\begin{pmatrix} x \\ 0 \end{pmatrix}}{\lra} A(-1)^{\oplus 2} \overset{\begin{pmatrix} x & y \end{pmatrix}}{\lra} A \lra 0,
\]
after which it simply repeats, since $xA = \{ a \in A | xa = 0 \}$.  In particular, $A$ has infinite global dimension.

To see that $k \langle x, y \rangle/(yx)$ is not right noetherian, show that the right ideal generated by $y, xy, x^2y, \dots $ is not finitely generated.  Similarly, for $k \langle x, y \rangle/(x^2)$ show that the right ideal generated by $x, yx, y^2x, \dots$ is not finitely generated.

\bigskip

4.  We claim that the minimal free resolution of $k$ is 
\[
0 \to A(-2) \overset{\begin{pmatrix} \tau(x_1) \\ \dots \\ \tau(x_n) \end{pmatrix}}{\lra} \bigoplus_{i=1}^n A(-1) \overset{\begin{pmatrix} x_1 & \dots & x_n \end{pmatrix}}{\lra} A \to 0.
\]
This complex is exact by Lemma~\ref{lem:begin}, except maybe in the final spot, in other words, the last map may not be injective.   Thus it is exact if and only if it has the Hilbert series predicted by \eqref{eq:hs}, that is $h_A(t) =1 /(1 - nt + t^2)$.
In this case, the AS-Gorenstein condition follows in the same way and $A$ is weakly AS-regular.

Assume now that the leading term of $f = \sum x_i \tau(x_i)$ is $x_n x_i$ for some $i < n$, under the degree lex order with $x_1 < \dots < x_n$.  Then clearly there are no overlaps to check and a $k$-basis of $A$ consists of words not containg $x_n x_i$.  If $W_m$ is 
the set of such words in degree $m$, then $W_m = \big(\bigcup_{1 \leq j \leq n} W_{m-1} x_j \big) \setminus W_{m-2} x_n x_1$.   
This leads to the Hilbert series equation $h_A(t) = n h_A(t) t  - h_A(t) t^2$ and hence $h_A(t) =1 /(1 - nt + t^2)$.  
Thus $A$ is weakly AS-regular by the argument above.  

If instead $f$ has leading term $x_n^2$, the same argument as in Exercise $2.3$ shows that a change of variable leads 
to a new relation $f  = \sum x'_i \tau'(x'_i)$, where $\tau'$ corresponds to a matrix congruent to the matrix representing $\tau$.
But every matrix is congruent to a matrix which is $0$ in the $(n, n)$-spot:  this amounts to finding a nontrivial zero of some homogeneous degree $2$ polynomial, which is always possible since $n \geq 2$ and $k$ is algebraically closed.

\bigskip

5.  The same argument as in the proof of Theorem~\ref{thm:reg2} shows that if $A$ is AS-regular of global dimension $2$, 
minimally generated by elements of degrees $d_1 \leq d_2 \leq \dots \leq d_n$, then the free resolution of $k_A$ must have the form
\[
0 \to A(-\ell) \overset{\begin{pmatrix} \tau(x_1) \\ \dots \\ \tau(x_n) \end{pmatrix}}{\lra} \bigoplus_{i=1}^n A(-d_i) \overset{\begin{pmatrix} x_1 & \dots & x_n \end{pmatrix}}{\lra} A \to 0,
\]
where $\ell = d_i + d_{n-i}$ for all $i$, and where the $\tau(x_i)$ are another minimal generating set for the algebra.   Then 
$A = k \langle x_1, \dots, x_n \rangle/(\sum x_i \tau(x_i))$ and $A$ has Hilbert series $1/(1 - \sum_{i=1}^n t^{d_i} + t^{\ell})$.  It is easy to check that whenever $n \geq 3$, necessarily the denominator of this Hilbert series has a real root bigger 
than $1$ and hence $A$ has exponential growth.   Thus $n = 1$ or $n = 2$, and 
$n = 1$ is easily ruled out as in the proof of Theorem~\ref{thm:reg2}.  Thus the main case left is $n = 2$ and $\ell = d_1 + d_2$.
Write $x_1 = x$ and $x_2 = y$.

If $d_1 = d_2$, then we can reduce to the degree $1$ generated case simply by reassigning degrees to the elements; so 
we know from Theorem~\ref{thm:reg2} that up to isomorphism we have one of the relations $yx - qxy$ or $yx - xy -x^2$.
If $d_2 > d_1$ but $d_1 i = d_2$ for some $i$, then $A_{d_2} = kx^i + ky$, $A_{d_1} = kx$, and so clearly 
the relation has the form $x(ax^i + by) + cyx = 0$ with $b$ and $c$ nonzero.  A change of variables sends this 
to $yx - xy - x^{i+1}$ or else $yx - qxy$.  Finally, if $d_1$ does not divide $d_2$, then $A_{d_2}= ky$, $A_{d_1} = kx$, and so the relation is necessarily of the form $yx - qxy$ (after scaling).

This limits the possible regular algebras to those on the given list.  That these algebras really are regular is proved in the same way as for the Jordan and quantum planes:  the Diamond Lemma easily gives their Hilbert series, which is used to prove the obvious potential resolution of $k_A$ is exact.

\bigskip

6(a). Recall that $v = (x_1, \dots, x_n)$.   We want to express $\pi' = \sum \alpha_{i_1, \dots, i_d} \tau(x_{i_d}) x_{i_1} \dots x_{i_{d-1}}$ as a matrix product.  We have 
$\pi = \sum_{i_d} \big( \sum_{i_1, \dots, i_{d-1}} \alpha_{i_1, \dots, i_d} x_{i_1} \dots x_{i_{d-1}} \big) x_{i_d}$
and thus $vM$ is the row vector with $i_d^{\small \text{th}}$ coordinate 
$(vM)_{i_d} = \sum_{i_1, \dots, i_{d-1}} \alpha_{i_1, \dots, i_d} x_{i_1} \dots x_{i_{d-1}}$.
We see from this that $\pi'$ is equal to $(\tau(x_1), \dots, \tau(x_n)) (vM)^t = vQ^{-1} (vM)^t$.
Now using that $Q M v^t = (vM)^t$, we get $\pi' = vQ^{-1} QM v^t = vM v^t = \pi$ as claimed.
Iterating this result $d$ times gives $\pi = \tau(\pi)$.

(b).  Writing $\pi = \sum r_i x_i$ for some uniquely determined $r_i$, we see that the $r_i$ are the coordinates 
of the row vector $v M$, and hence they are a $k$-basis for the minimal set of the relations of the algebra $A$ by 
the construction of the free resolution of $k_A$ (Lemma~\ref{lem:begin}).  If we have some other $k$-basis $\{y_i \}$ of 
$kx_1 + \dots + kx_n$, writing $\pi = \sum s_i y_i$, the $s_i$ are a basis for the same vector space as the $r_i$, 
so they are also a minimal set of relations for $A$.  In particular, taking $y_{i_d} = \tau(x_{i_d})$ and applying this 
to $\tau(\pi) = \pi$ shows that  
\[ 
\{ \sum_{i_1, \dots, i_{d-1}} \alpha_{i_1, \dots, i_d} \tau(x_{i_1}) \dots \tau(x_{i_{d-1}}) | 1 \leq i_d \leq n \}
\]
is a minimal set of relations for $A$.  Thus $\tau$ preserves the ideal of relations and induces an automorphism of $A$.

\subsection{Solutions to Exercise Set 3}

\

1.  The multilinearized relations can be written in the matrix form
\[
\begin{pmatrix} 0 & z_0 & -ry_0  \\ -qz_0 & 0 & x_0  \\ y_0 & -px_0 & 0  \end{pmatrix} \begin{pmatrix}  x_1 \\ y_1 \\ z_1 \end{pmatrix} = 0.
\]
The solutions $\{ (x_0:y_0:z_0), (x_1:y_1:z_1) \} \subseteq \mb{P}^2 \times \mb{P}^2$ to this equation give $X_2$.
Thus the first projection $E$ of $X_2$ is equal to those $(x_0:y_0:z_0)$ for which the $3 \times 3$ matrix $M$ above is singular.  

Taking the determinant gives $\det M =(1-pqr)x_0y_0z_0$.  Then either $pqr = 1$ 
and $\det M = 0$ identically so $E = \mb{P}^2$, or $pqr \neq 1$ and $\det M = 0$ when $x_0y_0z_0 = 0$, that is when $E$ is 
the union of the three coordinate lines in $\mb{P}^2$.
The equations for $X_2$ can also be written in the matrix form
\[
\begin{pmatrix}  x_0 &  y_0 &  z_0 \end{pmatrix} \begin{pmatrix} 0 & -rz_1 & y_1  \\ z_1 & 0 & -qx_1  \\ -py_1 & x_1 & 0  \end{pmatrix}, 
\]
and a similar calculation shows that the second projection of $X_2$ is also equal to $E$.

Now it is easy to check in each case that for $(x_0:y_0:z_0)$ such that $M$ is singular, it has rank exactly $2$.  In fact, 
in either case for $E$, the first two rows of $M$ are independent when $z_0 \neq 0$, the first and third when $y_0 \neq 0$, and 
the last two when $x_0 \neq 0$.  This implies that for each $p \in E$, there is a unique $q \in E$ such that $(p,q) \in X_2$.  
A similar argument using the second projection shows that for each $q \in E$ there is a unique $p \in E$ such that $(p,q) \in X_2$.  
Thus $X_2 = \{ (p, \sigma(p)) | p \in E \}$ for some bijection $\sigma$, and we can find a formula for $\sigma$ by 
taking the cross product of the first two rows when $z_0 \neq 0$ and similarly in the other two cases.

Thus when $z_0 \neq 0$ we get $\sigma(x_0:y_0:z_0) = (x_0: rqy_0: qz_0)$, when $y_0 \neq 0$ 
the formula is $\sigma(x_0:y_0:z_0) = (pr x_0: ry_0: z_0)$, and when $x_0 \neq 0$ the formula is 
$\sigma(x_0:y_0:z_0) = (px_0: y_0: pq z_0)$.  These formulas are correct in both cases for $E$.  (When $pqr = 1$
and $E = \mb{P}^2$, one can easily see that the three formulas match up to give a single formula.)

\bigskip

2.  This is a similar calculation as in Exercise 1.  The multilinearized relations can be written in two matrix forms 
\[
\begin{pmatrix} x_0 & y_0 \end{pmatrix} \begin{pmatrix} -y_1y_2 & y_1 x_2 \\-x_1 y_2 &  x_1 x_2 \end{pmatrix} = \begin{pmatrix} y_0 y_1  &  -x_0 y_1 \\ y_0 x_1 &  -x_0 x_1\end{pmatrix} \begin{pmatrix} x_2 \\ y_2 \end{pmatrix}.
\]
The determinants of both $2 \times 2$ matrices appearing are identically $0$.  Thus $X_3 \subseteq \mb{P}^1 \times \mb{P}^1 \times \mb{P}^1$
has projections $\pi_{12}$ and $\pi_{23}$ which are onto.   On the other hand, the matrices have rank exactly $1$ for 
each point in $\mb{P}^1 \times \mb{P}^1$, so that for each $(p_0, p_1) \in \mb{P}^1 \times \mb{P}^1$ there is a 
unique $p_2 \in \mb{P}^1$ with $(p_0, p_1, p_2) \in X_3$; explicitly, it is easy to see that $p_2 = p_0$.  
Thus $X_3 = \{ (p_0, p_1, p_0) | p_0, p_1 \in \mb{P}^1 \}$ and clearly the full set of point modules is in bijection already 
with $X_3 \cong \mb{P}^1 \times \mb{P}^1$.  Moreover, for the automorphism $\sigma$ of $\mb{P}^1 \times \mb{P}^1$
given by $(p_0, p_1) \mapsto (p_1, p_0)$, we have $X_3 = \{ (p, q, r) \in (\mb{P}^1)^{\times 3} | (q,r) = \sigma(p,q) \}$.

\bigskip

3.  The multilinearization of the relation is $y_0 x_1$.  Its vanishing set in $\mb{P}^1 \times \mb{P}^1$, with 
coordinates $((x_0: y_0), (x_1: y_1))$, is the set $X_2 = ((1:0) \times \mb{P}^1) \bigcup (\mb{P}^1 \times (0:1))$.  
Then by construction, $X_n \subseteq (\mb{P}^1)^{\times n}$ consists of sequences of points 
$(p_0, \dots p_{n-1})$ such that $(p_i, p_{i+1}) \in E$ for all $0 \leq i \leq n-2$.  Thus
\[
X_n = \{ (p_0, \dots, p_{n-1}) | \text{for some}\ 0 \leq i \leq n-1, p_j = (1:0)\ \text{for}\ j \leq i-1, 
p_j = (0:1)\ \text{for}\ j \geq i + 1 \}.
\]
In particular, the projection map $X_{n+1} \to X_n$ collapses the set
\[
\{ (\overbrace{(1:0), (1:0), \dots, (1:0)}^n, q) | q \in \mb{P}^1 \}
\] to a single point, and thus the inverse limit of Proposition~\ref{prop:param} 
does not stabilize for any $n$.

\bigskip

4. 
(a). Assume $A$ is quadratic regular.     This exercise simply formalizes the general pattern seen in all of the examples so far.    
The entries of the matrix $M$ are of degree $1$.  Write $M = M(x, y, z)$.  Taking the three relations to be the entries of $ vM$ in the free algebra, then the multilinearized relations 
can be put into either of the two forms 
\[
[QM](x_0, y_0, z_0) \begin{pmatrix} x_1 \\ y_1 \\ z_1  \end{pmatrix} = 0,  \qquad \begin{pmatrix} x_0 & y_0 & z_0 \end{pmatrix} M(x_1, y_1, z_1) = 0.
\]
Thus if $E$ is the vanishing of  $\det M$ in $\mb{P}^2$, then $\det QM = \det Q \det M$ differs only by a scalar, so also 
has vanishing set $E$.  Thus $E$ is equal to both the first and second projections of $X_2 \subseteq \mb{P}^2 \times \mb{P}^2$, by a similar argument as we have seen in the examples.  Either $\det M$ vanishes identically and so $E = \mb{P}^2$, or else $\det M$ is a cubic polynomial, and so $E$ is a degree 3 hypersurface in $\mb{P}^2$.

(b).  Now let $A$ be cubic regular.  In this case $M$ has entries of degree $2$.  We let $N = N(x_0, y_0; x_1, y_1)$ be the matrix $M$ with its entries multilinearized.  Then 
the multilinearized relations of $A$ can be put into either of the two forms
\[
[QN](x_0, y_0; x_1, y_1) \begin{pmatrix} x_2 \\ y_2  \end{pmatrix} = 0,  \qquad \begin{pmatrix} x_0 & y_0 \end{pmatrix} N(x_1, y_1; x_2, y_2) = 0.
\]
The determinant of $N$ is a degree $(2, 2)$ multilinear polynomial and so its 
vanishing is all of $\mb{P}^1 \times \mb{P}^1$ if $\det N$ is identically $0$, or a degree $(2,2)$-hypersurface in $\mb{P}^1 \times \mb{P}^1$ otherwise.  Again $\det N$ and $\det QN$ have the same vanishing set $E$ and so the projections $p_{12}(X_3)$ and $p_{23}(X_3)$ 
are both equal to $E$,  where $X_3 \subseteq \mb{P}^1 \times \mb{P}^1 \times \mb{P}^1$.

%(As alluded to at the end of this lecture, proving that $X_2$ (in case (a)) or $X_3$ (in case (b)) is the graph of an automorphism 
%of $E$ is more subtle.)

\subsection{Solutions to exercise set 4.}

\

1.  We do the quantum plane case; the Jordan plane can be analyzed with a similar idea.  Let $A = k \langle x, y \rangle/(yx - qxy)$ where $q$ is not a root of $1$.   Recall that $\{ x^i y^j | i, j \geq 0 \}$ is a $k$-basis for $A$.  Let $I$ be any nonzero ideal of $A$ and choose a nonzero element 
$f \in I$.  Write $f = \sum_{i = 0}^n p_i(x) y^i$ for some $p_i(x) \in k[x]$ with $p_n \neq 0$.  Choose such an element with $n$ as small as possible.  Then look at $ fx - q^i x f \in I$, which is of smaller degree in $y$ and so must be zero.  But this can happen only if $f = p_n(x) y^n$, since the powers of $q$ are distinct.  A similar argument in the other variable forces $x^m y^n \in I$ for some $m, n$.  But $x$ and $y$ are normal, and so 
$x^m y^n \in I$ implies $(x)^m (y)^n \subseteq I$.   Now if $I$ is also prime, then either $(x) \subseteq I$ or else $(y) \subseteq I$.   Thus every nonzero prime ideal $I$ of $A$ contains either $x$ or $y$.

Now $A/(x) \cong k[y]$ and $k[y]$ has only $0$ and $(y)$ 
as graded prime ideals.  Similarly, $A/(y) \cong k[x]$ which has only $0$ and $(x)$ as graded prime ideals.  It follows that 
$0, (x), (y), (x, y)$ are the only graded primes of $A$.  

\bigskip

2.  If $\theta \in \Hom_{\rqgr A}(\pi(M), \pi(N))$ is an isomorphism, with inverse 
$\psi \in \Hom_{\rqgr A}(\pi(N), \pi(M))$, find module maps $\wt{\theta}: M_{\geq n} \to N$ and 
$\wt{\psi}: N_{\geq m} \to M$ representing these morphisms in the respective direct limits 
 $\lim_{n \to \infty} \Hom_{\rgr A}(M_{\geq n}, N)$ 
and $\lim_{n \to \infty}\Hom_{\rgr A}(N_{\geq n}, M)$.

 Then for any $q \geq \max(m,n)$, 
 $\wt{\psi} \vert_{N_{\geq q}} \circ \wt{\theta} \vert_{M_{\geq q}}: M_{\geq q} \to M$ represents 
$\psi \circ \theta = 1$.   Thus is equal in the direct limit $\lim_{n \to \infty} \Hom_{\rgr A}(M_{\geq n}, M)$ 
to the identity map, so for some (possibly larger) $q$, the map $\wt{\psi} \vert_{N_{\geq q}} \circ \wt{\theta} \vert_{M_{\geq q}}$ is the identity 
map from $M_{\geq q}$ onto $M_{\geq q}$.  Similarly, there must be $r$ such that 
$\wt{\theta} \vert_{M_{\geq r}} \circ \wt{\psi} \vert_{N_{\geq r}}$ gives an isomorphism from $N_{\geq r}$ onto 
$N_{\geq r}$.  This forces $M_{\geq s} \cong N_{\geq s}$ for $s = \max(q,r)$.  
The converse is similar.
\begin{comment}
If there is an isomorphism $\phi: M_{\geq n} \to N_{\geq n}$, it represents a map 
in $\Hom_{\rqgr A}(\pi(M), \pi(N)) = \lim_{n \to \infty} \Hom_{\rgr A}(M_{\geq n}, N)$ and its inverse gives a map in 
$\Hom_{\rqgr A}(\pi(N), \pi(M)) = \lim_{n \to \infty} \Hom_{\rgr A}(N_{\geq n}, M)$.  The composition of these 
in $\Hom_{\rqgr A}(\pi(M), \pi(M))$ is represented by the map $\Hom_{\rqgr A}(M_{\geq n}, M)$ which is given 
by the identity map $M_{\geq n} \to M_{\geq n}$, that is, it is the identity element in $\Hom_{\rqgr A}(\pi(M), \pi(M))$.  
Similarly, the other composition is the identity.
\end{comment}

\bigskip

3.  Consider a Zhang twist $S = R^{\sigma}$ for some graded automorphism $\sigma$.  Thus $\sigma$ acts 
as a bijection of $R_1 = kx + ky + xz$.  We claim that $w \in S_1 = R_1$ is normal in $S$ if and only if $w$ is 
an eigenvector for $\sigma$.  Since $w * S_1 = wR_1$ and $S_1 * w = R_1\sigma(w)$, we have $w$ normal in $S$ 
if and only if $w R_1 = \sigma(w) R_1$, which happens if and only if $\sigma(w) = \lambda w$ for some $\lambda$ (for example, by unique factorization), proving the claim.

Now the quantum polynomial ring $A$ of Example~\ref{ex:qpoly2} has normal elements $x, y, z$ which are 
a basis for $A_1$.  By the previous paragraph, if $\phi: A \to R^{\sigma}$ is an isomorphism, we can take 
the images of $x, y, z$ to be a basis of eigenvectors for $\sigma$; after a change of variables, we can take 
these to be the elements with those names already in $R_1$, and thus $\sigma$ is now diagonalized 
with $\sigma(x) = ax$, $\sigma(y) = by$, $\sigma(z) = cz$ for nonzero $a, b, c$.   But then 
the relations of $R^{\sigma}$ are 
\[
y * x - ab^{-1} x* y, \ z * y - bc^{-1}y * z, \ x * z - ca^{-1}z *x,
\]
where $(ab^{-1})(bc^{-1})(c a^{-1}) = 1$.  Conversely, if $pqr = 1$ then taking $a = p, b = 1, c = q^{-1}$ 
we have $ab^{-1} = p, bc^{-1} = q, ca^{-1} = r$ so that 
the twist $R^{\sigma}$ by the $\sigma$ above will have the relations of the quantum polynomial ring in Example~\ref{ex:qpoly2}.

\bigskip

4. (a)
We have the following portion of the long exact sequence in Ext:
\[
\dots \lra \uExt^i(A/A_{\geq n}, A) \overset{\phi}{\lra} \uExt^i(A/A_{\geq n+1}, A) \lra \uExt^i(\bigoplus k(-n), A) \lra \dots 
\]
as indicated in the hint.  If $\uExt^i(k, A)$ is finite dimensional, say it is contained in degrees between $m_1$ and $m_2$, 
then $\uExt^i(\bigoplus k(-n), A)$ is finite dimensional and contained in degrees between $m_1 -n$ and $m_2 -n$.  In particular, $\uExt^i(A/A_{\geq n+1}, A)/ (\im \phi)$ is finite-dimensional, and is contained in negative degrees for $n \gg 0$.  This shows by induction on $n$ that the nonnegative part of the direct limit  
$\lim_{n \to \infty} \uExt^i(A/A_{\geq n}, A)_{\geq 0}$ is also finite dimensional.   If $\uExt^i(k, A) = 0$, then 
this same exact sequence implies by induction that $\uExt^i(A/A_{\geq n}, A) = 0$ for all $n$, and so 
$\lim_{n \to \infty} \uExt^i(A/A_{\geq n}, A) = 0$ in this case.  

(b).  Calculating $\uExt^i(k, A)$ with a minimal projective resolution of $k$, since $A$ is noetherian each term in the resolution 
is noetherian, and so the homology groups $\uExt^i(k, A)$ will also be noetherian $A$-modules.  Calculating with an 
injective resolution of $A$ instead, after applying $\uHom(k_A, - )$ each term is an $A$-module killed by $A_{\geq 1}$.
Thus the homology groups $\uExt^i(k, A)$ will also be killed by $A_{\geq 1}$.  Thus $\uExt^i(k, A)$ is both finitely generated 
and killed by $A_{\geq 1}$, so it is a finite dimensional $A/A_{\geq 1} = k$-module.  
(The reason this argument fails in the noncommutative case is that calculating with the projective resolution gives 
a right $A$-module structure to the Ext groups, while the calculation with the injective resolution gives a left $A$-module 
structure to the Ext groups.  The Ext groups are then $(A, A)$-bimodules which are finitely generated on one side 
and killed by $A_{\geq 1}$ on the other, which does not force them to be finite-dimensional.  In the commutative case 
the left and right module structures must coincide.)

\bigskip

5. (a) 
We have seen that $\{ x^i y^j | i, j \geq 0 \}$ is a $k$-basis of $B$, by the Diamond Lemma.  
Thus $By$ has $k$-basis 
$\{ x^i y^j | i \geq 0, j \geq 1 \}$ and $B/By$ has $\{1, x, x^2, \dots \}$ as $k$-basis.  
The idealizer $A'$ of $By$ is certainly a graded subring, so if it is larger than $A$, then $x^n \in A'$ for some $n \geq 1$.  But then $y x^n =  x^ny + nx^{n+1} \in By$, which is clearly false since $\cha k = 0$.  Obviously $k \subseteq A'$
and thus $A' = A$.

(b).  Since $x^n B y \in A$ for all $n \geq 1$, we see that each $x^n \in B$ is annihilated by $A_{\geq 1} = By$ 
as a right $A$-module.  Thus the right  $A$-module structure of $B/A$ is the same as its $A/A_{\geq 1} = k$-vector space structure.  In particular, as a graded module it is isomorphic to $\bigoplus_{n \geq 1} k(-n)$.

(c).   Left multiplication by $x^n$ for any $n \geq 1$ gives a homomorphism of degree $n$ in $\uHom_A(A_{\geq 1}, A)$, because $x^n A_{\geq 1} \subseteq A$.   Any homomorphism in $\uHom_A(A, A)$ is equal to left multiplication by some $a \in A$, and if its restriction to $\uHom_A(A_{\geq 1}, A)$ is the same as left multiplication by $x^n$ 
we get  $a b = x^n b$ for all $b \in A_{\geq 1}$ and thus $a = x^n$ since $B$ is a domain.  This contradicts 
$x^n \not \in A$.  Thus $x^n$ is an element of $\uHom_A(A_{\geq 1}, A)$ not in the image of $\uHom_A(A, A)$ for 
each $n \geq 1$.  The map $\uHom_A(A, A) \to \uHom_A(A_{\geq 1}, A)$ is also clearly injective, 
and so from the long exact sequence we get that $\uExt^1_A(k_A, A)$ is infinite-dimensional.   

More generally, since $A$ is a domain, all of the maps 
in the direct limit $\lim_{n \to \infty} \uHom_A(A_{\geq n}, A)$ are injective, so the cokernel 
$\lim_{n \to \infty} \uExt^1_A(A/A_{\geq n}, A)$ of the map $A \to \Hom_{\rqgr A}(\pi(A), \pi(A))$ is also  infinite-dimensional.  (In fact one may show that $\Hom_{\rqgr A}(\pi(A), \pi(A)) \cong B)$.

\bigskip

\bigskip

\bibliographystyle{amsalpha}

%\bibliography{../../../Dropbox/biblio}

\begin{thebibliography}{{Van}96}



\bibitem[Ar]{Ar}
M.~Artin, \emph{Some problems on three-dimensional graded domains},
  Representation theory and algebraic geometry (Waltham, MA, 1995), Cambridge
  Univ. Press, Cambridge, 1997, pp.~1--19. 


\bibitem[AS]{AS}
Michael Artin and William~F. Schelter, \emph{Graded algebras of global
  dimension {$3$}}, Adv. in Math. \textbf{66} (1987), no.~2, 171--216.


\bibitem[ASZ]{ASZ}
M.~Artin,L.\ W.\ Small and J.\ J.\ Zhang,
Generic flatness for strongly noetherian algebras, 
 \emph{J.\ Algebra}, \textbf{221} (1999), 579-610.
 
 \bibitem[AS1]{AS1}
M.~Artin and J.~T. Stafford, \emph{Noncommutative graded domains with quadratic
  growth}, Invent. Math. \textbf{122} (1995), no.~2, 231--276.

\bibitem[AS2]{AS2}
\bysame, \emph{Semiprime graded algebras of dimension two}, J. Algebra
  \textbf{227} (2000), no.~1, 68--123. 

\bibitem[ATV1]{ATV1}
M.~Artin, J.~Tate, and M.~Van~den Bergh, \emph{Some algebras associated to
  automorphisms of elliptic curves}, The Grothendieck Festschrift, Vol.\ I,
  Birkh\"auser Boston, Boston, MA, 1990, pp.~33--85.

\bibitem[ATV2]{ATV2}
\bysame, \emph{Modules over regular algebras of dimension $3$}, Invent. Math.
  \textbf{106} (1991), no.~2, 335--388. 

\bibitem[AV]{AV}
M.~Artin and M.~Van~den Bergh, \emph{Twisted homogeneous coordinate rings}, J.
  Algebra \textbf{133} (1990), no.~2, 249--271. 

\bibitem[AZ1]{AZ1}
M.~Artin and J.~J. Zhang, \emph{Noncommutative projective schemes}, Adv. Math.
  \textbf{109} (1994), no.~2, 228--287. 

\bibitem[AZ2]{AZ2}
\bysame, \emph{Abstract {H}ilbert schemes}, Algebr. Represent. Theory
  \textbf{4} (2001), no.~4, 305--394. 

\bibitem[Aus]{Aus}
Maurice Auslander, \emph{On the dimension of modules and algebras. {III}.
  {G}lobal dimension}, Nagoya Math. J. \textbf{9} (1955), 67--77. 

\bibitem[Be]{Be}
George~M. Bergman, \emph{The diamond lemma for ring theory}, Adv. in Math.
  \textbf{29} (1978), no.~2, 178--218. 

\bibitem[BSW]{BSW}
Raf Bocklandt, Travis Schedler, and Michael Wemyss, \emph{Superpotentials and
  higher order derivations}, J. Pure Appl. Algebra \textbf{214} (2010), no.~9,
  1501--1522. 

\bibitem[BH]{BH}
Winfried Bruns and J{\"u}rgen Herzog, \emph{Cohen-{M}acaulay rings}, Cambridge
  University Press, Cambridge, 1993. 

\bibitem[CV]{CV}
Thomas Cassidy and Michaela Vancliff, \emph{Generalizations of graded
  {C}lifford algebras and of complete intersections}, J. Lond. Math. Soc. (2)
  \textbf{81} (2010), no.~1, 91--112. 

\bibitem[Ch]{Ch}
Daniel Chan, \emph{Lectures on orders}, available on the author's webpage 
\begin{verbatim}web.maths.unsw.edu.au/~danielch/
\end{verbatim}

\bibitem[CI]{CI}
Daniel Chan and Colin Ingalls, \emph{The minimal model program for orders over
  surfaces}, Invent. Math. \textbf{161} (2005), no.~2, 427--452. 

\bibitem[CKWZ]{CKWZ} Kenneth Chan, Ellen Kirkman, Chelsea Walton, and James Zhang, 
\emph{Quantum binary polyhedral groups and their actions on quantum planes}, 
 arXiv:1303.7203.

\bibitem[DS]{DS}
C.~Dean and L.~W. Small, \emph{Ring theoretic aspects of the {V}irasoro
  algebra}, Comm. Algebra \textbf{18} (1990), no.~5, 1425--1431. 


\bibitem[DF]{DF}
J.~Diller and C.~Favre, \emph{Dynamics of bimeromorphic maps of surfaces},
  Amer. J. Math. \textbf{123} (2001), no.~6, 1135--1169.


\bibitem[Ei]{Ei}
David Eisenbud, \emph{Commutative algebra, with a view toward algebraic
  geometry}, Springer-Verlag, New York, 1995, Graduate texts in mathematics,
  No. 150. 

\bibitem[EH]{EH}
David Eisenbud and Joe Harris, \emph{The geometry of schemes}, Graduate Texts
  in Mathematics, vol. 197, Springer-Verlag, New York, 2000. 

\bibitem[Gi]{Gi} 
V.~Ginzburg, \emph{Calabi-Yau algebras}, arXiv:math/0612139.


\bibitem[GW]{GW}
K.~R. Goodearl and R.~B. Warfield, Jr., \emph{An introduction to noncommutative
  {N}oetherian rings}, second ed., London Mathematical Society Student Texts,
  vol.~61, Cambridge University Press, Cambridge, 2004. 


\bibitem[Ha]{Ha}  R.~Hartshorne, {\it Algebraic geometry}, Springer-Verlag, Berlin, 2006.


\bibitem[Jo]{Jo}
David~A. Jordan, \emph{The graded algebra generated by two {E}ulerian
  derivatives}, Algebr. Represent. Theory \textbf{4} (2001), no.~3, 249--275.


\bibitem[Ke]{Ke}
Dennis~S. Keeler, \emph{Criteria for $\sigma$-ampleness}, J. Amer. Math. Soc.
  \textbf{13} (2000), no.~3, 517--532.


\bibitem[KRS]{KRS}
D.~S. Keeler, D.~Rogalski, and J.~T. Stafford, \emph{Na\"\i ve noncommutative
  blowing up}, Duke Math. J. \textbf{126} (2005), no.~3, 491--546.
  
\bibitem[KKZ]{KKZ}
E.~Kirkman, J.~Kuzmanovich, and J.~J. Zhang, \emph{Gorenstein subrings of
  invariants under {H}opf algebra actions}, J. Algebra \textbf{322} (2009),
  no.~10, 3640--3669. 

\bibitem[KL]{KL}
G{\"u}nter~R. Krause and Thomas~H. Lenagan, \emph{Growth of algebras and
  {G}elfand-{K}irillov dimension}, revised ed., American Mathematical Society,
  Providence, RI, 2000. 



\bibitem[Len]{Len}
Helmut Lenzing, \emph{Weighted projective lines and applications},
  Representations of algebras and related topics, EMS Ser. Congr. Rep., Eur.
  Math. Soc., Z\"urich, 2011, pp.~153--187. 

\bibitem[Le]{Le}
Thierry Levasseur, \emph{Some properties of noncommutative regular graded
  rings}, Glasgow Math. J. \textbf{34} (1992), no.~3, 277--300. 

\bibitem[LPWZ]{LPWZ}
D.-M. Lu, J.~H. Palmieri, Q.-S. Wu, and J.~J. Zhang, \emph{Regular algebras of
  dimension 4 and their {$A\sb \infty$}-{E}xt-algebras}, Duke Math. J.
  \textbf{137} (2007), no.~3, 537--584. 

\bibitem[RV]{RV}
I.~Reiten and M.~Van~den Bergh, \emph{Noetherian hereditary abelian categories
  satisfying {S}erre duality}, J. Amer. Math. Soc. \textbf{15} (2002), no.~2,
  295--366. 

\bibitem[RRZ]{RRZ}
M. Reyes, D. Rogalski, and J. J. Zhang, \emph{Skew Calabi-Yau algebras and homological identities}, 
 arXiv:1302.0437.


\bibitem[Ro]{Ro}
Daniel Rogalski, \emph{Generic noncommutative surfaces}, Adv. Math.
  \textbf{184} (2004), no.~2, 289--341.

\bibitem[RSi]{RSi}
D.~Rogalski and Susan~J. Sierra, \emph{Some projective surfaces of
  {GK}-dimension 4}, Compos. Math. \textbf{148} (2012), no.~4, 1195--1237.


\bibitem[RSS]{RSS}
D.~Rogalski, S.~J. Sierra, and J.~T. Stafford, \emph{Classifying Orders in the Sklyanin Algebra}, 
arXiv:1308.2213.

\bibitem[RS1]{RS1}
D.~Rogalski and J.~T. Stafford, \emph{Na\"\i ve noncommutative blowups at
  zero-dimensional schemes}, J. Algebra \textbf{318} (2007), no.~2, 794--833.
 
\bibitem[RS2]{RS2}
\bysame, \emph{A class of noncommutative projective surfaces}, Proc. Lond.
  Math. Soc. (3) \textbf{99} (2009), no.~1, 100--144. 

\bibitem[RZ1]{RZ1}
D.~Rogalski and J.~J. Zhang, \emph{Canonical maps to twisted rings}, Math. Z.
  \textbf{259} (2008), no.~2, 433--455. 

\bibitem[RZ2]{RZ2}
\bysame, \emph{Regular algebras of dimension 4 with 3 generators}, New trends
  in noncommutative algebra, Contemp. Math., vol. 562, Amer. Math. Soc.,
  Providence, RI, 2012, pp.~221--241. 


\bibitem[Rot]{Rot}
Joseph~J. Rotman, \emph{An introduction to homological algebra}, second ed.,
  Universitext, Springer, New York, 2009. 


\bibitem[ShVa]{ShVa}
Brad Shelton and Michaela Vancliff, \emph{Schemes of line modules. {I}}, J.
  London Math. Soc. (2) \textbf{65} (2002), no.~3, 575--590. 


\bibitem[Si1]{Si1}
Susan~J. Sierra, \emph{Rings graded equivalent to the {W}eyl algebra}, J.  Algebra \textbf{321} (2009), no.~2, 495--531. 

\bibitem[Si2]{Si2}
\bysame, \emph{Classifying birationally commutative projective surfaces}, Proc.
  Lond. Math. Soc. (3) \textbf{103} (2011), no.~1, 139--196. 


\begin{comment}
\bibitem[Ro2]{Rog09}
D.~Rogalski, \emph{Blowup subalgebras of the {S}klyanin algebra},
  arXiv:0912.2304, 2009.

\bibitem[RSS]{supernoetherian}
D.~Rogalski, S.~J. Sierra, and J.~T. Stafford, Algebras in which every
  subalgebra is noetherian, arXiv:1112.3869, 2011.
\end{comment}


\bibitem[SW]{SW}  Susan J. Sierra and Chelsea Walton, 
\emph{The universal enveloping algebra of the Witt algebra is not noetherian}, 
arXiv:1304.0114.


\bibitem[Sm1]{Sm1}
S.~Paul Smith, \emph{Subspaces of non-commutative spaces}, Trans. Amer. Math.
  Soc. \textbf{354} (2002), no.~6, 2131--2171. 

\bibitem[Sm2]{Sm2}
\bysame, \emph{Maps between non-commutative spaces}, Trans. Amer. Math. Soc.
  \textbf{356} (2004), no.~7, 2927--2944.




\bibitem[Sm3]{Sm3}
S.~Paul Smith, \emph{A quotient stack related to the {W}eyl algebra}, J. Algebra
  \textbf{345} (2011), 1--48.


\bibitem[SS]{SS}
S.~P. Smith and J.~T. Stafford, \emph{Regularity of the four-dimensional
  {S}klyanin algebra}, Compositio Math. \textbf{83} (1992), no.~3, 259--289.
  

\bibitem[SmV]{SmV}
S.~Paul Smith and Michel Van~den Bergh, \emph{Noncommutative quadric surfaces},
  J. Noncommut. Geom. \textbf{7} (2013), no.~3, 817--856. 

\bibitem[StZh]{StZh}
J.~T. Stafford and J.~J. Zhang, \emph{Examples in non-commutative projective
  geometry}, Math. Proc. Cambridge Philos. Soc. \textbf{116} (1994), no.~3,
  415--433. 




\bibitem[StV]{StV}
J.~T. Stafford and M.~van~den Bergh, \emph{Noncommutative curves and
  noncommutative surfaces}, Bull. Amer. Math. Soc. (N.S.) \textbf{38} (2001),
  no.~2, 171--216. 


\bibitem[Ste1]{Ste1} Darin~R. Stephenson, \emph{Noncommutative projective geometry}, 
unpublished lecture notes.


\bibitem[Ste2]{Ste2}
\bysame, \emph{Artin-{S}chelter regular algebras of global dimension
  three}, J. Algebra \textbf{183} (1996), no.~1, 55--73. 

\bibitem[Ste3]{Ste3}
\bysame, \emph{Algebras associated to elliptic curves}, Trans.
  Amer. Math. Soc. \textbf{349} (1997), no.~6, 2317--2340. 




\bibitem[SZ1]{SZ1}
Darin~R. Stephenson and James~J. Zhang, \emph{Growth of graded {N}oetherian
  rings}, Proc. Amer. Math. Soc. \textbf{125} (1997), no.~6, 1593--1605.


\bibitem[SZ2]{SZ2}
\bysame, \emph{Noetherian connected graded algebras of global dimension 3}, J.
  Algebra \textbf{230} (2000), no.~2, 474--495. 





\bibitem[TV]{TV}
John Tate and Michel van~den Bergh, \emph{Homological properties of {S}klyanin
  algebras}, Invent. Math. \textbf{124} (1996), no.~1-3, 619--647.
 


\bibitem[VdB]{VdB}
Michel Van~den Bergh, \emph{Blowing up of non-commutative smooth surfaces},
  Mem. Amer. Math. Soc. \textbf{154} (2001), no.~734, x+140. 


\bibitem[VdB2]{VdB2}
\bysame, \emph{Noncommutative quadrics}, Int. Math. Res. Not. IMRN (2011),
  no.~17, 3983--4026. 


\bibitem[VOW]{VOW}
Fred Van~Oystaeyen and Luc Willaert, \emph{Grothendieck topology, coherent
  sheaves and {S}erre's theorem for schematic algebras}, J. Pure Appl. Algebra
  \textbf{104} (1995), no.~1, 109--122. 


\bibitem[YZ]{YZ}
Amnon Yekutieli and James~J. Zhang, \emph{Serre duality for noncommutative
  projective schemes}, Proc. Amer. Math. Soc. \textbf{125} (1997), no.~3,
  697--707. 


\bibitem[Zh1]{Zh1}
J.~J. Zhang, \emph{Twisted graded algebras and equivalences of graded
  categories}, Proc. London Math. Soc. (3) \textbf{72} (1996), no.~2, 281--311.
 
\bibitem[Zh2]{Zh2}
\bysame, \emph{Non-{N}oetherian regular rings of dimension {$2$}}, Proc.
  Amer. Math. Soc. \textbf{126} (1998), no.~6, 1645--1653. 

\bibitem[ZZ1]{ZZ1}
James~J. Zhang and Jun Zhang, \emph{Double {O}re extensions}, J. Pure Appl.
  Algebra \textbf{212} (2008), no.~12, 2668--2690. 

\bibitem[ZZ2]{ZZ2}
\bysame, \emph{Double extension regular algebras of type (14641)}, J. Algebra
  \textbf{322} (2009), no.~2, 373--409. 


\end{thebibliography}

%\begin{comment}

\providecommand{\bysame}{\leavevmode\hbox to3em{\hrulefill}\thinspace}

\providecommand{\MR}{\relax\ifhmode\unskip\space\fi MR }


\providecommand{\MRhref}[2]{%

  \href{http://www.ams.org/mathscinet-getitem?mr=#1}{#2}

}

\providecommand{\href}[2]{#2}

\end{document}